\documentclass[a4paper,12pt, abstract=true]{scrartcl}

\makeatletter
\DeclareOldFontCommand{\rm}{\normalfont\rmfamily}{\mathrm}
\DeclareOldFontCommand{\sf}{\normalfont\sffamily}{\mathsf}
\DeclareOldFontCommand{\tt}{\normalfont\ttfamily}{\mathtt}
\DeclareOldFontCommand{\bf}{\normalfont\bfseries}{\mathbf}
\DeclareOldFontCommand{\it}{\normalfont\itshape}{\mathit}
\DeclareOldFontCommand{\sl}{\normalfont\slshape}{\@nomath\sl}
\DeclareOldFontCommand{\sc}{\normalfont\scshape}{\@nomath\sc}
\makeatother

\usepackage[ansinew]{inputenc}
\usepackage[ngerman,english,USenglish]{babel}
\usepackage{latexsym,amssymb,amsfonts,bm,bbm}   
\usepackage{amsmath}
\usepackage{theorem}
\usepackage{paralist}
\usepackage{xcolor}
\usepackage{graphicx}
\usepackage{subfigure} 
\usepackage{float} 
\usepackage{esint}
\usepackage{stmaryrd}
\usepackage[normalem]{ulem} 
\usepackage{blindtext}
\usepackage{hyperref}
\usepackage[numbers]{natbib}

\newtheorem{dummytheorem}{Dummy-Theorem}[section]
\newcommand{\proofendsign}{$\Box$} 
\newtheorem{definition}[dummytheorem]{Definition}
\newtheorem{lemma}[dummytheorem]{Lemma}
\newtheorem{theorem}[dummytheorem]{Theorem}
\newtheorem{proposition}[dummytheorem]{Proposition}

\newenvironment{proof}{{\noindent \bf Proof }}
 {{\hspace*{\fill}\proofendsign\par\bigskip}}
\theorembodyfont{\normalfont}
\newtheorem{remarknorm}[dummytheorem]{Remark}
\newtheorem{examplenorm}[dummytheorem]{Example}

\theorembodyfont{\rm}

\newcommand{\N}{\mathbb{N}}

\newcommand{\R}{\mathbb{R}}
\newcommand{\M}{\mathbb{M}}

\newcommand{\pr}{\mathbb{P}}
\newcommand{\ex}{\mathbb{E}}

\newcommand{\eins}{\mathbbm{1}}

\newcommand{\OFP}{(\Omega,{\cal F},\pr)}

\renewcommand{\P}{\boldsymbol{P}}
\newcommand{\Q}{\boldsymbol{Q}}

\allowdisplaybreaks

\begin{document}


\title{First-order sensitivity of the optimal value in a Markov decision model with respect to deviations in the transition probability function}

\author{
Patrick Kern\footnote{Department of Mathematics, Saarland University; {kern@math.uni-sb.de}}
\qquad\quad
Axel Simroth\footnote{Fraunhofer Institute for Transportation and Infrastructure Systems; {axel.simroth@ivi.fraunhofer.de}} 
\qquad\quad
Henryk Zähle\footnote{Department of Mathematics, Saarland University; {zaehle@math.uni-sb.de}}
}
\date{ } 

\maketitle

\begin{abstract}
Markov decision models (MDM) used in practical applications are most often less complex than the underlying `true' MDM. The reduction of model complexity is performed for several reasons. However, it is obviously of interest to know what kind of model reduction is reasonable (in regard to the optimal value) and what kind is not. In this article we propose a way how to address this question. We introduce a sort of derivative of the optimal value as a function of the transition probabilities, which can be used to measure the (first-order) sensitivity of the optimal value w.r.t.\ changes in the transition probabilities. `Differentiability' is obtained for a fairly broad class of MDMs, and the `derivative' is specified explicitly. Our theoretical findings are illustrated by means of optimization problems in inventory control and mathematical finance.
\end{abstract}



\newpage

\section{Introduction}\label{intro}


Already in the 1990th, M\"uller \cite{Mueller1997} pointed out that the impact of the transition probabilities of a Markov decision process (MDP) on the optimal value of a corresponding Markov decision model (MDM) can {\em not} be ignored for practical issues. For instance, in most cases the transition probabilities are unknown and have to be estimated by statistical methods. Moreover in many applications the `true' model is replaced by an approximate version of the `true' model or by a variant which is simplified and thus less complex. The result is that in practical applications the optimal (strategy and thus the optimal) value is most often computed on the basis of transition probabilities that differ from the underlying true transition probabilities. Therefore the sensitivity of the optimal value w.r.t.\ deviations in the transition probabilities is obviously of interest.

M\"uller \cite{Mueller1997} showed that under some structural assumptions the optimal value in a discrete-time MDM depends continuously on the transition probabilities, and he established bounds for the approximation error. In the course of this the distance between transition probabilities was measured by means of some suitable probability metrics. Even earlier, Kolonko \cite{Kolonko1983} obtained analogous bounds in a MDM in which the transition probabilities depend on a parameter. Here the distance between transition probabilities was measured by means of the distance between the respective parameters. Error bounds for the expected total reward of discrete-time Markov reward processes were also specified by van Dijk \cite{vanDijk1988} and van Dijk and Puterman \cite{vanDijkPuterman1988}. In the latter reference the authors also discussed the case of discrete-time Markov decision processes with countable state and action spaces. 

In this article, we focus on the situation where the `true' model is replaced by a less complex version (for a simple example, see Subsection \mbox{\ref{Subsec - Inventory Example - Numerical example}} in the supplementary material). The reduction of model complexity in practical applications is common and performed for several reasons. Apart from computational aspects and the difficulty of considering all relevant factors, one major point is that statistical inference for certain transition probabilities can be costly in terms of both time and money. However, it is obviously of interest to know what kind of model reduction is reasonable and what kind is not. In the following we want to propose a way how to address the latter question.

Our original motivation comes from the field of optimal logistics transportation planning, where ongoing projects like SYNCHRO-NET \cite{Synchronet} aim at stochastic decision models based on transition probabilities estimated from historical route information. Due to the lack of historical data for unlikely events, transition probabilities are often modeled in a simplified way. In fact, events with small probabilities are often ignored in the model. However, the impact of these events on the optimal value (here the minimal expected transportation costs) of the corresponding MDM may nevertheless be significant. The identification of unlikely but potentially cost sensitive events is therefore a major challenge. In logistics planning operations engineers have indeed become increasingly interested in comprehensibly quantifying the sensitivity of the optimal value w.r.t.\ the incorporation of unlikely events into the model. For background see, for instance, \cite{HolfeldSimroth2017,Holfeldetal2018}. The assessment of rare but risky events takes on greater importance also in other areas of applications; see, for instance, \cite{Komljenovic2016,Yangetal2015} and references cited therein.

By an incorporation of an unlikely event into the model we mean, for instance, that under performance of an action $a$ at some time $n$ a previously impossible transition from one state $x$ to another state $y$ gets now assigned small but strictly positive probability $\varepsilon$. Mathematically this means that the transition probability $P_n((x,a),\,\cdot\,)$ is replaced by $(1-\varepsilon)P_n((x,a),\,\bullet\,)+\varepsilon Q_n((x,a),\,\bullet\,)$ with $Q_n((x,a),\,\bullet\,):=\delta_y[\,\bullet\,]$, where $\delta_y$ is the Dirac measure at $y$.
More generally one could consider a change of the whole transition function (the family of all transition probabilities) $\P$ to $(1-\varepsilon)\P+\varepsilon\Q$ with $\varepsilon>0$ small. For operations engineers it is here interesting to know how this change affects the optimal value, ${\cal V}_{0}(\P)$. If the effect is minor, then an incorporation can be seen as superfluous, at least from a pragmatic point of view. If on the other hand the effect is significant, then the engineer should consider the option to extend the model and to make an effort to get access to statistical data for the extended model.

At this point it is worth mentioning that a change of the transition function from $\P$ to $(1-\varepsilon)\P+\varepsilon\Q$ with $\varepsilon>0$ small can also have a different interpretation than an incorporation of an (unlikely) {\em  new event}. It could also be associated with an incorporation of an (unlikely) {\em divergence from the normal transition rules}. See Subsection \mbox{\ref{Subsec - Finance Example - Numerical example}} for an example.

In this article, we will introduce an approach for quantifying the effect of changing the transition function from $\P$ to $(1-\varepsilon)\P+\varepsilon\Q$, with $\varepsilon>0$ small, on the optimal value ${\cal V}_{0}(\P)$ of the MDM. In view of $(1-\varepsilon)\P+\varepsilon\Q=\P+\varepsilon(\Q-\P)$, we feel that it is reasonable to quantify the effect by a sort of derivative of the value functional ${\cal V}_{0}$ at $\P$ evaluated at direction $\Q-\P$. To some extent the `derivative' $\dot{\cal V}_{0;\P}(\Q-\P)$ specifies the first-order sensitivity of ${\cal V}_{0}(\P)$ w.r.t.\ a change of $\P$ as above. Take into account that
\begin{equation}\label{motivation of measure for f-o sensitivity}
    {\cal V}_{0}(\P+\varepsilon(\Q-\P)) - {\cal V}_{0}(\P)\,\approx\,\varepsilon\cdot\dot{\cal V}_{0;\P}(\Q-\P) \qquad \mbox{for $\varepsilon>0$ small}.
\end{equation}

To be able to compare the first-order sensitivity for (infinitely) many different $\Q$, it is favourable to know that the approximation in (\mbox{\ref{motivation of measure for f-o sensitivity}}) is uniform in $\Q\in{\cal K}$ for preferably large sets ${\cal K}$ of transition functions. Moreover, it is not always possible to specify the relevant $\Q$ exactly. For that reason it would be also good to have robustness (i.e.\ some sort of continuity) of $\dot{\cal V}_{0;\P}(\Q-\P)$ in $\Q$. These two things induced us to focus on a variant of tangential ${\cal S}$-differentiability as introduced by Sebasti\~{a}o e Silva \cite{Sebastiao e Silva1956a} and Averbukh and Smolyanov \cite{AverbukhSmolyanov1967} (here ${\cal S}$ is a family of sets ${\cal K}$ of transition functions). In Section \mbox{\ref{Sec - Hadamard differentiability chapter}} we present a result on `${\cal S}$-differentiability' of ${\cal V}_0$ for the family ${\cal S}$ of all {\em relatively compact} sets of admissible transition functions and a reasonably broad class of MDMs, where we measure the distance between transition functions by means of metrics based on probability metrics as in \cite{Mueller1997}.

The `derivative' $\dot{\cal V}_{0;\P}(\Q-\P)$ of the optimal value functional ${\cal V}_{0}$ at $\P$ quantifies the effect of a change from $\P$ to $(1-\varepsilon)\P+\varepsilon\Q$, with $\varepsilon>0$ small, assuming that after the change the strategy $\pi$ (tuple of the underlying decision rules) is chosen such that it optimizes the
target value ${\cal V}_{0}^{\pi}(\P')$ (e.g.\ expected total costs or rewards) in $\pi$ under the new transition function $\P':=(1-\varepsilon)\P+\varepsilon\Q$. On the other hand, practitioners are also interested in quantifying the impact of a change of $\P$ when the optimal strategy (under $\P$) is kept after the change. Such a quantification would somehow answers the question: How much different does a strategy derived in a simplified MDM perform in a more complex (more realistic) variant of the MDM? Since the `derivative' $\dot{\cal V}_{0;\P}^\pi(\Q-\P)$ of the functional ${\cal V}_{0}^\pi$ under a {\em fixed} strategy $\pi$ turns out to be a building stone for the derivative $\dot{\cal V}_{0;\P}(\Q-\P)$ of the optimal value functional ${\cal V}_{0}$ at $\P$, our elaborations cover both situations anyway. For fixed strategy $\pi$ we obtain `${\cal S}$-differentiability' of ${\cal V}_0^\pi$ even for the broader family ${\cal S}$ of all {\em bounded} sets of admissible transition functions.

The `derivative' which we propose to regard as a measure for the first-order sensitivity will formally be introduced in Definition \mbox{\ref{def S differentiability}}. This definition is applicable to quite general finite time horizon MDMs and might look somewhat cumbersome at first glance. However, in the special case of a finite state space and finite action spaces, a situation one faces in many practical applications, the proposed `differentiability' boils down to a rather intuitive concept. This will be explained in Section \mbox{\ref{Sec - Discrete MDM}} of the supplementary material with a minimum of notation and terminology. In Section \mbox{\ref{Sec - Discrete MDM}} of the supplementary material we will also reformulate a backward iteration scheme for the computation of the `derivative' (which can be deduced from our main result, Theorem \mbox{\ref{Thm - hadamard cal v}}) in the discrete case, and we will discuss an example.

In Section \mbox{\ref{Sec - Formal definition of MDM}} we formally introduce quite general MDMs in the fashion of the standard monographs \cite{BaeuerleRieder2011,HernandezLasserre1996,Hinderer1970,Puterman1994}. Since it is important to have an elaborate notation in order to formulate our main result, we are very precise in Section \mbox{\ref{Sec - Formal definition of MDM}}. As a result, this section is a little longer compared to the respective sections in other articles on MDMs. In Section \mbox{\ref{Sec - Hadamard differentiability chapter}} we carefully introduce our notion of `differentiability' and state our main result concerning the computation of the `derivative' of the value functional.

In Section \mbox{\ref{Sec - Finance Example}} we will apply the results of Section \mbox{\ref{Sec - Hadamard differentiability chapter}} to assess the impact of one or more than one unlikely but substantial shock in the dynamics of an asset on the solution of a terminal wealth problem in a (simple) financial market model free of shocks. This example somehow motivates the general set-up chosen in Sections \mbox{\ref{Sec - Formal definition of MDM}}--\mbox{\ref{Sec - Hadamard differentiability chapter}}. All results of this article are proven in Sections \mbox{\ref{Sec - Proof of results from Section 4}}--\mbox{\ref{Sec - Proof of results from Section 5}} of the supplementary material. For the convenience of the reader we recall in Section \mbox{\ref{Sec - Existence of optimal strategies}} of the supplementary material a result on the existence of optimal strategies in general MDMs. Section \mbox{\ref{Sec - hoelder metric psi weak topology}} of the supplementary material contains an auxiliary topological result.


\section{Formal definition of Markov decision model} \label{Sec - Formal definition of MDM}


Let $E$ be a non-empty set equipped with a $\sigma$-algebra ${\cal E}$, referred to as {\em state space}. Let $N\in\N$ be a fixed finite time horizon (or planning horizon) in discrete time. For each point of time $n=0,\ldots,N-1$ and each state $x\in E$, let $A_n(x)$ be a non-empty set. The elements of $A_n(x)$ will be seen as the admissible {\em actions} (or {\em controls}) at time $n$ in state $x$. For each $n=0,\ldots,N-1$, let
$$
     A_n:=\bigcup_{x\in E} A_n(x) \quad\mbox{ and }\quad D_n:=\big\{(x,a)\in E\times A_n:\,a\in A_n(x)\big\}.
$$
The elements of $A_n$ can be seen as the actions that may basically be selected at time $n$ whereas the elements of $D_n$ are the possible state-action combinations at time $n$. For our subsequent analysis, we equip $A_n$ with a $\sigma$-algebra ${\cal A}_n$, and let ${\cal D}_n:=({\cal E}\otimes{\cal A}_n)\cap D_n$ be the trace of the product $\sigma$-algebra ${\cal E}\otimes{\cal A}_n$ in $D_n$. Recall that a map $P_n:D_n\times{\cal E}\rightarrow[0,1]$ is said to be a probability kernel (or Markov kernel) from $(D_n,{\cal D}_n)$ to $(E,{\cal E})$ if $P_n(\,\cdot\,,B)$ is a $({\cal D}_n,{\cal B}([0,1]))$-measurable map for any $B\in{\cal E}$, and $P_n((x,a),\,\bullet\,)\in{\cal M}_1(E)$ for any $(x,a)\in D_n$. Here ${\cal M}_1(E)$ is the set of all probability measures on $(E,{\cal E})$.


\subsection{Markov decision process}\label{Subsec - Markov decision process}


In this subsection, we will give a formal definition of an $E$-valued (discrete-time) Markov decision process (MDP) associated with a given initial state, a given transition function and a given strategy.
By definition a {\em (Markov decision) transition (probability) function} is an $N$-tuple
$$
     \P = (P_0,\ldots,P_{N-1})
$$
whose $n$-th entry $P_n$ is a probability kernel from $(D_n,{\cal D}_n)$ to $(E,{\cal E})$. In this context $P_n$ will be referred to as {\em one-step transition (probability) kernel at time $n$} (or {\em from time $n$ to $n+1$}) and the probability measure $P_n((x,a),\,\bullet\,)$ is referred to as {\em one-step transition probability at time $n$} (or {\em from time $n$ to $n+1$}) {\em given state $x$ and action $a$}. We denote by ${\cal P}$ the set of all transition functions.

We will assume that the actions are performed by a so-called $N$-stage strategy (or $N$-stage policy). An {\em ($N$-stage) strategy} is an $N$-tuple
$$
     \pi = (f_0,\ldots,f_{N-1})
$$
of decision rules at times $n=0,\ldots,N-1$, where a {\em decision rule at time $n$} is an $({\cal E},{\cal A}_n)$-measurable map $f_n:E\rightarrow A_n$ satisfying $f_n(x)\in A_n(x)$ for all $x\in E$. Note that a decision rule at time $n$ is (deterministic and) `Markovian' since it only depends on the current state and is independent of previous states and actions. We denote by $\mathbb{F}_n$ the set of {\em all} decision rules at time $n$, and assume that $\mathbb{F}_n$ is non-empty.
Hence a strategy is an element of the set $\mathbb{F}_0\times\cdots\times\mathbb{F}_{N-1}$, and this set can be seen as the set of {\em all} strategies. Moreover, we fix for any $n=0,\ldots,N-1$ some $F_n\subseteq\mathbb{F}_n$ which can be seen as the set of all {\em admissible} decision rules at time $n$. In particular, the set $\Pi:=F_0\times\cdots\times F_{N-1}$ can be seen as the set of all {\em admissible} strategies.

For any transition function $\P=(P_n)_{n=0}^{N-1}\in{\cal P}$, strategy $\pi=(f_n)_{n=0}^{N-1}\in\Pi$, and time point $n\in\{0,\ldots,N-1\}$, we can derive from $P_n$ a probability kernel $P_n^{\pi}$ from $(E,{\cal E})$ to $(E,{\cal E})$ through
\begin{equation}\label{def prob kernel E to E}
    P_n^{\pi}(x,B) := P_n\big((x,f_n(x)),B\big),	\qquad x\in E,\,B\in{\cal E}.
\end{equation}
The probability measure $P_n^\pi(x,\,\bullet\,)$ can be seen as the one-step transition probability at time $n$ given state $x$ when the transitions and actions are governed by $\P$ and $\pi$, respectively.

Now, consider the measurable space
$$
    (\Omega,{\cal F}) := (E^{N+1},{\cal E}^{\otimes(N+1)}).
$$
For any $x_0\in E$, $\P=(P_n)_{n=0}^{N-1}\in{\cal P}$, and $\pi\in\Pi$ define the probability measure
\begin{equation}\label{def P pi}
     \pr^{x_0,\P;\pi} := \delta_{x_0}\otimes P_0^\pi\otimes\cdots\otimes P_{N-1}^\pi
\end{equation}
on $(\Omega,{\cal F})$, where $x_0$ should be seen as the {\em initial state} of the MDP to be constructed. The right-hand side of (\mbox{\ref{def P pi}}) is the usual product of the probability measure $\delta_{x_0}$ and the kernels $P_0^\pi,\ldots,P_{N-1}^\pi$; for details see display (\mbox{\ref{def P pi - exactly}}) in Section \mbox{\ref{Sec - Supplement to Section 2}} of the supplementary material.  Moreover let $\boldsymbol{X}=(X_0,\ldots,X_N)$ be the identity on $\Omega$, i.e.
\begin{equation}\label{def bold X}
    X_n(x_0,\ldots,x_N) := x_n,	\qquad (x_0,\ldots,x_N)\in E^{N+1},\,n=0,\ldots,N.
\end{equation}
Note that, for any $x_0\in E$, $\P=(P_n)_{n=0}^{N-1}\in{\cal P}$, and $\pi\in\Pi$, the map $\boldsymbol{X}$ can be regarded as an $(E^{N+1},{\cal E}^{\otimes (N+1)})$-valued random variable on the probability space $(\Omega,{\cal F},\pr^{x_0,\P;\pi})$ with distribution $\delta_{x_0}\otimes P_0^\pi\otimes\cdots\otimes P_{N-1}^\pi$.

It follows from Lemma \mbox{\ref{lemma on X and P}} in the supplementary material) that for any $x_0,\widetilde{x}_0,\break x_1,\ldots,x_n\in E$, $\P=(P_n)_{n=0}^{N-1}\in{\cal P}$, $\pi=(f_n)_{n=0}^{N-1}\in\Pi$, and $n=1,\ldots,N-1$
\begin{enumerate}
    \item[{\rm (i)}] $\pr^{x_0,\P;\pi}[X_{0}\in\,\bullet\,]=\delta_{x_0}[\,\bullet\,]$,

    \item[{\rm (ii)}] $\pr^{x_0,\P;\pi}[X_{1}\in\,\bullet\,\|X_0=\widetilde{x}_0]=P_0\big((x_0,f_0(x_0)),\,\bullet\,\big)$,

    \item[{\rm (iii)}] $\pr^{x_0,\P;\pi}[X_{n+1}\in\,\bullet\,\|(X_0,X_1,\ldots,X_n)=(\widetilde{x}_0,x_1,\ldots,x_n)]= P_n\big((x_n,f_n(x_n)),\,\bullet\,\big)$,

    \item[{\rm (iv)}] $\pr^{x_0,\P;\pi}[X_{n+1}\in\,\bullet\,\|X_n=x_n]= P_n\big((x_n,f_n(x_n)),\,\bullet\,\big)$.
\end{enumerate}
The formulation of (ii)--(iv) is somewhat sloppy, because in general a (regular version of the) factorized conditional distribution of $X$ given $Y$ under $\pr^{x_0,\P;\pi}$ (evaluated at a fixed set $B\in{\cal E}$) is only $\pr_Y^{x_0,\P;\pi}$-a.s.\ unique. So assertion (iv) in fact means that the probability kernel $P_n((\,\cdot\,,f_n(\,\cdot\,)),\,\bullet\,)$ provides a (regular version of the) factorized conditional distribution of $X_{n+1}$ given $X_n$ under $\pr^{x_0,\P;\pi}$, and analogously for (ii) and (iii). Note that the factorized conditional distribution in part (ii) is constant w.r.t.\ $\widetilde{x}_0\in E$.
Assertions (iii) and (iv) together imply that the temporal evolution of $X_n$ is Markovian. This justifies the following terminology.

\begin{definition}[MDP]\label{def MDP}
Under law $\pr^{x_0,\P;\pi}$ the random variable $\boldsymbol{X}=(X_0,\ldots,X_N)$ is called (discrete-time) Markov decision process (MDP) associated with initial state $x_0\in E$, transition function $\P\in{\cal P}$ and strategy $\pi\in\Pi$.
\end{definition}


\subsection{Markov decision model and value function} \label{Subsec - MDM}


Maintain the notation and terminology introduced in Subsection \mbox{\ref{Subsec - Markov decision process}}. In this subsection, we will first define a (discrete-time) Markov decision model (MDM) and introduce subsequently the corresponding value function. The latter will be derived from a reward maximization problem.
Fix $\boldsymbol{P}\in{\cal P}$, and let for each point of time $n=0,\ldots,N-1$
$$
     r_n : D_n\longrightarrow\R
$$
be a $({\cal D}_n,{\cal B}(\R))$-measurable map, referred to as {\em one-stage reward function}. Here $r_n(x,a)$ specifies the one-stage reward when action $a$ is taken at time $n$ in state $x$. Let
$$
     r_N : E\longrightarrow\R
$$
be an $({\cal E},{\cal B}(\R))$-measurable map, referred to as {\em terminal reward function}. The value $r_N(x)$ specifies the reward of being in state $x$ at terminal time $N$.

Denote by $\boldsymbol{A}$ the family of all sets $A_n(x)$, $n=0,\ldots,N-1$, $x\in E$, and set $\boldsymbol{r}:=(r_n)_{n=0}^N$. Moreover let $\boldsymbol{X}$ be defined as in (\mbox{\ref{def bold X}}) and recall Definition \mbox{\ref{def MDP}}. Then we define our MDM as follows.

\begin{definition}[MDM]\label{def MDM}
The quintuple $(\boldsymbol{X},\boldsymbol{A},\P,\Pi,\boldsymbol{r})$ is called (discrete-time) Markov decision model (MDM) associated with the family of action spaces $\boldsymbol{A}$, transition function $\P\in{\cal P}$, set of admissible strategies $\Pi$, and reward functions $\boldsymbol{r}$.
\end{definition}

In the sequel we will always assume that a MDM $(\boldsymbol{X},\boldsymbol{A},\P,\Pi,\boldsymbol{r})$ satisfies the following Assumption {\bf (A)}. In Subsection \mbox{\ref{Subsec - Bounding functions}} we will discuss some conditions on the MDM under which Assumption {\bf (A)} holds.
We will use $\ex_{n,x_n}^{x_0,\P;\pi}$ to denote the expectation w.r.t.\ the factorized conditional distribution $\pr^{x_0,\P;\pi}[\,\bullet\,\| X_n=x_n]$. For $n=0$, we clearly have $\pr^{x_0,\P;\pi}[\,\bullet\,\| X_0=x_0]=\pr^{x_0,\P;\pi}[\,\bullet\,]$ for every $x_0\in E$; see Lemma \mbox{\ref{lemma on X and P}} in the supplementary material. In what follows we use the convention that the sum over the empty set is zero.

\par\medskip

\noindent{\bf Assumption (A)}:\,
$
\sup_{\pi=(f_n)_{n=0}^{N-1}\in\Pi}\ex_{n,x_n}^{x_0,\P;\pi}[\,\sum_{k=n}^{N-1}|r_k(X_k,f_k(X_k))|+ |r_N(X_N)|\,]<\infty
$
for any $x_n\in E$ and $n=0,\ldots,N$.

\par\medskip

Under Assumption {\bf (A)} we may define in a MDM $(\boldsymbol{X},\boldsymbol{A},\P,\Pi,\boldsymbol{r})$ for any $\pi=(f_n)_{n=0}^{N-1}\in\Pi$ and $n=0,\ldots,N$ a map $V_{n}^{\P;\pi}:E\rightarrow\R$ through
\begin{equation}\label{exp tot reward}
    V_n^{\P;\pi}(x_n) := \ex_{n,x_n}^{x_0,\P;\pi}\Big[\sum_{k=n}^{N-1} r_k(X_k,f_k(X_k)) + r_N(X_N)\Big].
\end{equation}
As a factorized conditional expectation this map is $({\cal E},{\cal B}(\R))$-measurable (for any $\pi\in\Pi$ and $n=0,\ldots,N$). Note that for $n=1,\ldots,N$ the right-hand side of (\mbox{\ref{exp tot reward}}) does {\em not} depend on $x_0$; see Lemma \mbox{\ref{lemma on integrals of functions of X}} in the supplementary material. Therefore the map $V_n^{\P;\pi}(\cdot)$ need {\em not} be equipped with an index $x_0$.

The value $V_{n}^{\P;\pi}(x_n)$ specifies the {\em expected total reward from time $n$ to $N$} of $\boldsymbol{X}$ under $\pr^{x_0,\P;\pi}$ when strategy $\pi$ is used and $\boldsymbol{X}$ is in state $x_n$ at time $n$. It is natural to ask for those strategies $\pi\in\Pi$ for which the expected total reward from time $0$ to $N$ is maximal for all initial states $x_0\in E$. This results in the following optimization problem:
\begin{equation}\label{maximization problem}
    V_0^{\P;\pi}(x_0) \longrightarrow \max\ \mbox{(in $\pi\in\Pi$)\,!}
\end{equation}
If a solution $\pi^{\P}$ to the optimization problem (\mbox{\ref{maximization problem}}) (in the sense of Definition \mbox{\ref{def opt strategy}} ahead) exists, then the corresponding maximal expected total reward is given by the so-called {\em value function (at time $0$)}.

\begin{definition}[Value function] \label{def value function}
For a MDM $(\boldsymbol{X},\boldsymbol{A},\P,\Pi,\boldsymbol{r})$ the value function at time $n\in\{0,\ldots,N\}$ is the map $V_{n}^{\P}:E\rightarrow\R$ defined by
\begin{equation}\label{value function}
    V_n^{\P}(x_n) := \sup_{\pi\in\Pi}V_n^{\P;\pi}(x_n).
\end{equation}
\end{definition}

Note that the value function $V_{n}^{\P}$ is well defined due to Assumption {\bf (A)} but not necessarily $({\cal E},{\cal B}(\R))$-measurable. The measurability holds true, for example, if the sets $F_{n},\ldots,F_{N-1}$ are at most countable or if conditions (a)--(c) of Theorem \mbox{\ref{Thm - existence opt strategy}} in the supplementary material) are satisfied; see also Remark \mbox{\ref{remark existence optimal strat}}(i) in the supplementary material.

\begin{definition}[Optimal strategy]\label{def opt strategy}
In a MDM $(\boldsymbol{X},\boldsymbol{A},\P,\Pi,\boldsymbol{r})$ a strategy $\pi^{\P}\in\Pi$ is called optimal w.r.t.\ $\P$ if 
\begin{equation}\label{def opt strategy - eq10}
    V_0^{\P;\pi^{\P}}(x_0) = V_0^{\P}(x_0)\quad\mbox{for all }x_0\in E.
\end{equation}
In this case $V_0^{\P;\pi^{\P}}(x_0)$ is called optimal value (function), and we denote by $\Pi(\P)$ the set of all optimal strategies w.r.t.\ $\P$. Further, for any given $\delta>0$, a strategy $\pi^{\P;\delta}\in\Pi$ is called $\delta$-optimal w.r.t.\ $\P$ in a MDM $(\boldsymbol{X},\boldsymbol{A},\P,\Pi,\boldsymbol{r})$ if
\begin{equation}\label{def opt strategy - eq20}
	V_0^{\P}(x_0) - \delta \le V_0^{\P;\pi^{\P;\delta}}(x_0) \quad\mbox{for all }x_0\in E,
\end{equation}
and we denote by $\Pi(\P;\delta)$ the set of all $\delta$-optimal strategies w.r.t.\ $\P$.
\end{definition}

Note that condition (\mbox{\ref{def opt strategy - eq10}}) requires that $\pi^{\P}\in\Pi$ is an optimal strategy for {\em all} possible initial states $x_0\in E$. Though, in some situations it might be sufficient to ensure that $\pi^{\P}\in\Pi$ is an optimal strategy only for some fixed initial state $x_0$.
For a brief discussion of the existence and computation of optimal strategies, see Section \mbox{\ref{Sec - Existence of optimal strategies}} of the supplementary material.

\begin{remarknorm}\label{remark non Markovian strat}
(i) In practice, the choice of an action can possibly be based on historical observations of states and actions. In particular one could relinquish the Markov property of the decision rules and allow them to depend also on previous states and actions. Then one might hope that the corresponding (deterministic) history-dependent strategies improve the optimal value of a MDM $(\boldsymbol{X},\boldsymbol{A},\P,\Pi,\boldsymbol{r})$. However, it is known that the optimal value of a MDM $(\boldsymbol{X},\boldsymbol{A},\P,\Pi,\boldsymbol{r})$ can {\em not} be enhanced by considering history-dependent strategies; see, e.g., Theorem 18.4 in \cite{Hinderer1970} or Theorem 4.5.1 in \cite{Puterman1994}.

(ii) Instead of considering the reward maximization problem (\mbox{\ref{maximization problem}}) one could as well be interested in minimizing expected total costs over the time horizon $N$. In this case, one can maintain the previous notation and terminology when regarding the functions $r_n$ and $r_N$ as the one-stage costs and the terminal costs, respectively. The only thing one has to do is to replace ``$\sup$'' by ``$\inf$'' in the representation (\mbox{\ref{value function}}) of the value function. Accordingly, a strategy $\pi^{\P;\delta}\in\Pi$ will be $\delta$-optimal for a given $\delta>0$ if in condition (\mbox{\ref{def opt strategy - eq20}}) ``$-\delta$'' and ``$\le$'' are replaced by ``$+\delta$'' and ``$\ge$''.
{\hspace*{\fill}$\Diamond$\par\bigskip}
\end{remarknorm}


\section{`Differentiability' in $\P$ of the optimal value} \label{Sec - Hadamard differentiability chapter}


In this section, we show that the value function of a MDM, regarded as a real-valued functional on a set of transition functions, is `differentiable' in a certain sense. The notion of `differentiability' we use for functionals that are defined on a set of admissible transition functions will be introduced in Subsection \mbox{\ref{Subsec - Hadamard differentiability definition}}. The motivation of our notion of `differentiability' was discussed subsequent to (\mbox{\ref{motivation of measure for f-o sensitivity}}). Before defining `differentiability' in a precise way, we will explain in Subsections \mbox{\ref{Subsec - Metric on set of probability measures}}--\mbox{\ref{Subsec - Metric on set of transition functions}} how we measure the distance between transition functions. In Subsections \mbox{\ref{Subsec - Hadamard differentiability of the optimal value functional}}--\mbox{\ref{Subsec - alternative rep Frechet derivative}} we will specify the `Hadamard derivative' of the value function.
At first, however, we will discuss in Subsection \mbox{\ref{Subsec - Bounding functions}} some conditions under which Assumption {\bf (A)} holds true. Throughout this section, $\boldsymbol{A}$, $\Pi$, and $\boldsymbol{r}$ are fixed.


\subsection{Bounding functions} \label{Subsec - Bounding functions}


Recall from Section \mbox{\ref{Sec - Formal definition of MDM}} that ${\cal P}$ stands for the set of all transition functions, i.e.\ of all $N$-tuples $\P=(P_n)_{n=0}^{N-1}$ of probability kernels $P_n$ from $(D_n,{\cal D}_n)$ to $(E,{\cal E})$. Let $\psi:E\rightarrow\R_{\ge 1}$ be an $({\cal E},{\cal B}(\R_{\ge 1}))$-measurable map, referred to as {\em gauge function}, where $\R_{\ge 1}:=[1,\infty)$. Denote by $\M(E)$ the set of all $({\cal E},{\cal B}(\R))$-measurable maps $h\in\R^E$, and let $\M_\psi(E)$ be the set of all $h\in\M(E)$ satisfying $\|h\|_{\psi}:=\sup_{x\in E}|h(x)|/\psi(x)<\infty$. The following definition is adapted from \cite{BaeuerleRieder2011,Mueller1997,Wessels1977}. Conditions (a)--(c) of this definition are sufficient for the well-definiteness of $V_{n}^{\P;\pi}$ (and $V_{n}^{\P}$); see Lemma \mbox{\ref{lemma suff cond for stand cond - pre}} ahead.

\begin{definition}[Bounding function] \label{def bounding function}
Let ${\cal P}'\subseteq{\cal P}$. A gauge function $\psi:E\to\R_{\ge 1}$ is called a bounding function for the family of MDMs $\{(\boldsymbol{X},\boldsymbol{A},\P,\Pi,\boldsymbol{r}):\P\in{\cal P}'\}$ if there exist
finite constants $K_1,K_2,K_3>0$ such that the following conditions hold for any $n=0,\ldots,N-1$ and $\P=(P_n)_{n=0}^{N-1}\in{\cal P}'$.
\vspace{-1.5mm}
\begin{enumerate}
    \item[{\rm (a)}] $|r_n(x,a)| \le K_1 \psi(x)$\ for all $(x,a)\in D_n$.

    \item[{\rm (b)}] $|r_N(x)| \le K_2 \psi(x)$\ for all $x\in E$.

    \item[{\rm (c)}] $\int_E\psi(y)\,P_n\big((x,a),dy\big)\le K_3 \psi(x)$\ for all $(x,a)\in D_n$.
\end{enumerate}
If ${\cal P}'=\{\P\}$ for some $\P\in{\cal P}$, then $\psi$ is called a bounding function for the MDM $(\boldsymbol{X},\boldsymbol{A},\P,\Pi,\boldsymbol{r})$.
\end{definition}

Note that the conditions in Definition \mbox{\ref{def bounding function}} do {\em not} depend on the set $\Pi$. That is, the terminology {\em bounding function} is independent of the set of all (admissible) strategies. Also note that conditions (a) and (b) can be satisfied by unbounded reward functions.

The following lemma, whose proof can be found in Subsection \mbox{\ref{Subsec - proof lemma suff cond for stand cond - pre}} of the supplementary material, ensures that Assumption {\bf (A)} is satisfied when the underlying MDM possesses a bounding function.

\begin{lemma}\label{lemma suff cond for stand cond - pre}
Let ${\cal P}'\subseteq{\cal P}$. If the family of MDMs $\{(\boldsymbol{X},\boldsymbol{A},\P,\Pi,\boldsymbol{r}):\P\in{\cal P}'\}$ possesses a bounding function $\psi$, then Assumption {\bf (A)} is satisfied for any $\P\in{\cal P}'$. Moreover, the expectation in Assumption {\bf (A)} is even uniformly bounded w.r.t.\ $\P\in{\cal P}'$, and $V_n^{\P;\pi}(\cdot)$ is contained in $\M_\psi(E)$ for any $\P\in{\cal P}'$, $\pi\in\Pi$, and $n=0,\ldots,N$.
\end{lemma}


\subsection{Metric on set of probability measures}\label{Subsec - Metric on set of probability measures}


In Subsection \mbox{\ref{Subsec - Hadamard differentiability definition}} we will work with a (semi-) metric (on a set of transition functions) to be defined in (\mbox{\ref{def metric transition functions}}) below. As it is common in the theory of probability metrics (see, e.g., p.\,10\,ff in \cite{Rachev1991}), we allow the distance between two probability measures and the distance between two transition functions to be infinite. That is, we adapt the axioms of a (semi-) metric but we allow a (semi-) metric to take values in $\overline\R_{\ge 0}:=\R_{\ge 0}\cup\{\infty\}$ rather than only in $\R_{\ge 0}:=[0,\infty)$.

Let $\psi$ be any gauge function, and denote by ${\cal M}_1^\psi(E)$ the set of all $\mu\in{\cal M}_1(E)$ for which $\int_E \psi\,d\mu<\infty$. Note that the integral $\int_E h\,d\mu$ exists and is finite for any $h\in\M_\psi(E)$ and $\mu\in{\cal M}_1^\psi(E)$. For any fixed $\M\subseteq\M_\psi(E)$, the distance between two probability measures $\mu,\nu\in{\cal M}_1^\psi(E)$ can be measured by
\begin{equation}\label{def integral prob metric}
	d_{\M}(\mu,\nu) := \sup_{h\in\M}\Big|\int_E h\,d\mu - \int_E h\,d\nu\Big|.
\end{equation}
Note that (\mbox{\ref{def integral prob metric}}) indeed defines a map $d_{\M}:{\cal M}_1^\psi(E)\times{\cal M}_1^\psi(E)\rightarrow\overline{\R}_+$ which is symmetric and fulfills the triangle inequality, i.e.\ $d_{\M}$ provides a semi-metric. If $\M$ separates points in ${\cal M}_1^\psi(E)$ (i.e.\ if any two $\mu,\nu\in{\cal M}_1^\psi(E)$ coincide when $\int_E h\,d\mu=\int_E h\,d\nu$ for all $h\in\M$), then $d_{\M}$ is even a metric. It is sometimes called {\em integral probability metric} or {\em probability metric with a $\zeta$-structure}; see \cite{Mueller1997b,Zolotarev1983}. In some situations the (semi-) metric $d_\M$ (with $\M$ fixed) can be represented by the right-hand side of (\mbox{\ref{def integral prob metric}}) with $\M$ replaced by a different subset $\M'$ of $\M_\psi(E)$. Each such set $\M'$ is said to be a {\em generator of $d_\M$}. The largest generator of $d_\M$ is called the {\em maximal generator of $d_\M$} and denoted by $\overline{\M}$. That is, $\overline{\M}$ is defined to be the set of all $h\in\M_\psi(E)$ for which $|\int_Eh\,d\mu-\int_Eh\,d\nu|\le d_\M(\mu,\nu)$ for all $\mu,\nu\in{\cal M}_1^\psi(E)$.

We now give some examples for the distance $d_{\M}$.
The metrics in the first four examples were already mentioned in \cite{Mueller1997,Mueller1997b}.
In the last three examples $d_\M$ metricizes the $\psi$-weak topology. The latter is  defined to be the coarsest topology on ${\cal M}_1^\psi(E)$ for which all mappings $\mu\mapsto\int_E h\,d\mu$, $h\in \mathbb{C}_\psi(E)$, are continuous. Here $\mathbb{C}_\psi(E)$ is the set of all continuous functions in $\M_\psi(E)$. If specifically $\psi\equiv 1$, then ${\cal M}_1^\psi(E)={\cal M}_1(E)$ and the $\psi$-weak topology is nothing but the classical weak topology. In Section 2 in \cite{Kraetschmeretal2017} one can find characterizations of those subsets of ${\cal M}_1^\psi(E)$ on which the relative $\psi$-weak topology coincides with the relative weak topology.

\begin{examplenorm}\label{examples prob metric - tv}
Let $\psi:\equiv 1$ and $\M:=\M_{\rm{\scriptsize{TV}}}$, where $\M_{\rm{\scriptsize{TV}}}:=\{\eins_B : B\in{\cal E}\}\subseteq\M_\psi(E)$. Then $d_\M$ equals the {\em total variation metric} $d_{\rm{\scriptsize{TV}}}(\mu,\nu) := \sup_{B\in{\cal E}}|\mu[B] - \nu[B]|$. The set $\M_{\rm{\scriptsize{TV}}}$ clearly separates points in ${\cal M}_1^{\psi}(E)={\cal M}_1(E)$. The maximal generator of $d_{\rm{\scriptsize{TV}}}$ is the set $\overline{\M}_{\rm{\scriptsize{TV}}}$ of all $h\in\M(E)$ with ${\rm sp}(h):=\sup_{x\in E} h(x) - \inf_{x\in E} h(x)\le 1$; see Theorem 5.4 in \cite{Mueller1997b}.
{\hspace*{\fill}$\Diamond$\par\bigskip}
\end{examplenorm}

\begin{examplenorm}\label{examples prob metric - kolmogorov}
For $E=\R$, let $\psi:\equiv 1$ and $\M:=\M_{\rm{\scriptsize{Kolm}}}$, where $\M_{\rm{\scriptsize{Kolm}}}:=\{\eins_{(-\infty,t]} : t\in\R\}\subseteq\M_\psi(\R)$. Then $d_\M$ equals the {\em Kolmogorov metric} $d_{\rm{\scriptsize{Kolm}}}(\mu,\nu) := \sup_{t\in\R}|F_{\mu}(t) - F_{\nu}(t)|$, where $F_\mu$ and $F_\nu$ refer to the distribution functions of $\mu$ and $\nu$, respectively. The set $\M_{\rm{\scriptsize{Kolm}}}$ clearly separates points in ${\cal M}_1^{\psi}(\R)={\cal M}_1(\R)$. The maximal generator of $d_{\rm{\scriptsize{Kolm}}}$ is the set $\overline{\M}_{\rm{\scriptsize{Kolm}}}$ of all $h\in\R^\R$ with $\mathbb{V}(h)\le 1$, where $\mathbb{V}(h)$ denotes the total variation of $h$; see Theorem 5.2 in \cite{Mueller1997b}.
{\hspace*{\fill}$\Diamond$\par\bigskip}
\end{examplenorm}

\begin{examplenorm}\label{examples prob metric - bounded lipschitz}
Assume that $(E,d_E)$ is a metric space and let ${\cal E}:={\cal B}(E)$. Let $\psi:\equiv 1$ and $\M:=\M_{\rm{\scriptsize{BL}}}$, where $\M_{\rm{\scriptsize{BL}}}:=\{h\in\R^E: \|h\|_{\rm{\scriptsize{BL}}}\leq 1 \}\subseteq\M_\psi(E)$ with $\|h\|_{\rm{\scriptsize{BL}}}:=\max\{\|h\|_{\infty},\,\|h\|_{\rm{\scriptsize{Lip}}}\}$
for $\|h\|_{\infty}:=\sup_{x\in E}|h(x)|$ and $\|h\|_{\rm{\scriptsize{Lip}}}:=\sup_{x,y\in E:\,x\neq y}|h(x)-\break h(y)|/d_E(x,y)$. Then $d_\M$ is nothing but the {\em bounded Lipschitz metric} $d_{\rm{\scriptsize{BL}}}$. The set $\M_{\rm{\scriptsize{BL}}}$ separates points in ${\cal M}_1^{\psi}(E)={\cal M}_1(E)$; see Lemma 9.3.2 in \cite{Dudley2002}. Moreover it is known (see, e.g., Theorem 11.3.3 in \cite{Dudley2002}) that if $E$ is separable then $d_{\rm{\scriptsize{BL}}}$ metricizes the weak topology on ${\cal M}_1^\psi(E)={\cal M}_1(E)$.
{\hspace*{\fill}$\Diamond$\par\bigskip}
\end{examplenorm}

\begin{examplenorm}\label{examples prob metric - kantorovich}
Assume that $(E,d_E)$ is a metric space and let ${\cal E}:={\cal B}(E)$. For some fixed $x'\in E$, let $\psi(x):= 1 + d_E(x,x')$ and $\M:=\M_{\rm{\scriptsize{Kant}}}$, where $\M_{\rm Kant}:=\{h\in\R^E: \|h\|_{\rm{\scriptsize{Lip}}}\leq 1 \}\subseteq\M_\psi(E)$ with $\|h\|_{\rm{\scriptsize{Lip}}}$ as in Example \mbox{\ref{examples prob metric - bounded lipschitz}}. Then $d_\M$ is nothing but the {\em Kantorovich metric} $d_{\rm{\scriptsize{Kant}}}$. The set $\M_{\rm{\scriptsize{Kant}}}$ separates points in ${\cal M}_1^{\psi}(E)$, because $\M_{\rm{\scriptsize{BL}}}$ ($\subseteq \M_{\rm{\scriptsize{Kant}}}$) does.
It is known (see, e.g., Theorem 7.12 in \cite{Villani2003}) that if $E$ is complete and separable then $d_{\rm{\scriptsize{Kant}}}$ metricizes the
$\psi$-weak topology on ${\cal M}_1^\psi(E)$.

Recall from \cite{Vallender1974} that for $E=\R$ the {\em $L^1$-Wasserstein metric} $d_{{\rm{\scriptsize{Wass}}}_1}(\mu,\nu) := \int_{-\infty}^\infty |F_\mu(t) - F_\nu(t)|\,dt$
coincides with the Kantorovich metric. In this case the $\psi$-weak topology is also referred to as $L^1$-weak topology. Note that the $L^1$-Wasserstein metric is a conventional metric for measuring the distance between probability distributions; see, for instance, \cite{DallAglio1956,KantorovichRubinstein1958,Vallender1974} for the general concept and \cite{Bellinietal2014,Kieseletal2016,Kraetschmeretal2012,KraetschmerZaehlel2017} for recent applications.
{\hspace*{\fill}$\Diamond$\par\bigskip}
\end{examplenorm}

Although the Kantorovich metric is a popular and well established metric, for the application in Section \mbox{\ref{Sec - Finance Example}} we will need the following generalization from $\alpha=1$ to $\alpha\in(0,1]$.

\begin{examplenorm}\label{examples prob metric - hoelder}
Assume that $(E,d_E)$ is a metric space and let ${\cal E}:={\cal B}(E)$. For some fixed $x'\in E$ and $\alpha\in(0,1]$, let $\psi(x):= 1 + d_E(x,x')^\alpha$ and $\M:=\M_{\rm{\scriptsize{H\ddot{o}l}},\alpha}$, where $\M_{\rm{\scriptsize{H\ddot{o}l}},\alpha}:=\{h\in\R^E: \|h\|_{\rm{\scriptsize{H\ddot{o}l}},\alpha}\leq 1 \}\subseteq\M_\psi(E)$ with $\|h\|_{\rm{\scriptsize{H\ddot{o}l}},\alpha}:=\sup_{x,y\in E:\,x\neq y}|h(x)\break-h(y)|/d_E(x,y)^\alpha$. The set $\M_{\rm{\scriptsize{H\ddot{o}l}},\alpha}$ separates points in ${\cal M}_1^{\psi}(E)$ (this follows with similar arguments as in the proof of Lemma 9.3.2 in \cite{Dudley2002}). Then $d_{\M}$ provides a metric on ${\cal M}_1^\psi(E)$ which we denote by $d_{\rm{\scriptsize{H\ddot{o}l}},\alpha}$ and refer to as {\em H{\"o}lder-$\alpha$ metric}. Especially when dealing with risk averse utility functions (as, e.g., in Section \mbox{\ref{Sec - Finance Example}}) this metric can be beneficial. Lemma \mbox{\ref{lemma - d beta hoelder - psi weak topology}} in Section \mbox{\ref{Sec - hoelder metric psi weak topology}} of the supplementary material shows that if $E$ is complete and separable then $d_{\rm{\scriptsize{H\ddot{o}l}},\alpha}$ metricizes the $\psi$-weak topology on ${\cal M}_1^\psi(E)$.
{\hspace*{\fill}$\Diamond$\par\bigskip}
\end{examplenorm}


\subsection{Metric on set of transition functions}\label{Subsec - Metric on set of transition functions}


Maintain the notation from Subsection \mbox{\ref{Subsec - Metric on set of probability measures}}. Let us denote by $\overline{\cal P}_\psi$ the set of all transition functions $\P=(P_n)_{n=0}^{N-1}\in{\cal P}$ satisfying $\int_E\psi(y)\,P_n((x,a),dy)<\infty$ for all $(x,a)\in D_n$ and $n=0,\ldots,N-1$. That is, $\overline{\cal P}_\psi$ consists of those transition functions $\P=(P_n)_{n=0}^{N-1}\in{\cal P}$ with $P_n((x,a),\,\bullet\,)\in{\cal M}_1^\psi(E)$ for all $(x,a)\in D_n$ and $n=0,\ldots,N-1$. Hence, for the elements $\P=(P_n)_{n=0}^{N-1}$ of $\overline{\cal P}_\psi$ all integrals of the shape $\int_E h(y)\,P_n((x,a),dy)$, $h\in\M_\psi(E)$, $(x,a)\in D_n$, $n=0,\ldots,N-1$, exist and are finite. In particular, for two transition functions $\P=(P_n)_{n=0}^{N-1}$ and $\Q=(Q_n)_{n=0}^{N-1}$ from $\overline{\cal P}_\psi$ the distance $d_\M(P_n((x,a),\,\bullet\,),Q_n((x,a),\,\bullet\,))$ is well defined for all $(x,a)\in D_n$ and $n=0,\ldots,N-1$ (recall that $\M\subseteq\M_\psi(E)$). So we can define the distance between two transition functions $\P=(P_n)_{n=0}^{N-1}$ and $\Q=(Q_n)_{n=0}^{N-1}$ from $\overline{\cal P}_\psi$ by
\begin{equation}\label{def metric transition functions}
    d_{\infty,\M}^{\phi}(\P,\Q):=\max_{n=0,\ldots,N-1}\sup_{(x,a)\in D_n}\,\frac{1}{\phi(x)}\cdot d_{\M}\Big(P_n\big((x,a),\,\bullet\,\big),Q_n\big((x,a),\,\bullet\,\big)\Big)
\end{equation}
for another gauge function $\phi:E\to\R_{\ge 1}$. Note that (\mbox{\ref{def metric transition functions}}) defines a semi-metric $d_{\infty,\M}^{\phi}:\overline{\cal P}_\psi\times\overline{\cal P}_\psi\rightarrow\overline{\R}_{\ge 0}$ on $\overline{\cal P}_\psi$ which is even a metric if $\M$ separates points in ${\cal M}_1^\psi(E)$.

Maybe apart from the factor $1/\phi(x)$, the definition of $d_{\infty,\M}^{\phi}(\P,\Q)$ in (\mbox{\ref{def metric transition functions}}) is quite natural and in line with the definition of a distance introduced by M\"uller \cite[p.\,880]{Mueller1997}. In \cite{Mueller1997}, M\"uller considers time-homogeneous MDMs, so that the transition kernels do not depend on $n$. He fixed a state $x$ and took the supremum only over all admissible actions $a$ in state $x$. That is, for any $x\in E$ he defined the distance between $P((x,\,\cdot\,),\,\bullet\,)$ and $Q((x,\,\cdot\,),\,\bullet\,)$ by $\sup_{a\in A(x)}d_{\M}(P((x,a),\,\bullet\,),Q((x,a),\,\bullet\,))$. To obtain a reasonable distance between $P_n$ and $Q_n$ it is however natural to take the supremum of the distance between $P_n((x,\,\cdot\,),\,\bullet\,)$ and $Q_n((x,\,\cdot\,),\,\bullet\,)$ uniformly over $a$ {\em and} over $x$.

The factor $1/\phi(x)$ in (\mbox{\ref{def metric transition functions}}) causes that the (semi-) metric $d_{\infty,\M}^{\phi}$ is less strict compared to the (semi-) metric $d_{\infty,\M}^1$ which is defined as in (\mbox{\ref{def metric transition functions}}) with $\phi:\equiv 1$. For a motivation of considering the factor $1/\phi(x)$, see part (iii) of Remark \mbox{\ref{remark S differentiability}} and the discussion afterwards.


\subsection{Definition of `differentiability'} \label{Subsec - Hadamard differentiability definition}


Let $\psi$ be any gauge function, and fix some ${\cal P}_\psi\subseteq\overline{\cal P}_\psi$ being closed under mixtures (i.e.\ $(1-\varepsilon)\boldsymbol{P}+\varepsilon\boldsymbol{Q}\in{\cal P}_\psi$ for any $\boldsymbol{P},\boldsymbol{Q}\in{\cal P}_\psi$, $\varepsilon\in(0,1)$). The set ${\cal P}_\psi$ will be equipped with the distance $d_{\infty,\M}^{\phi}$ introduced in (\mbox{\ref{def metric transition functions}}). In Definition \mbox{\ref{def S differentiability}} below we will introduce a reasonable notion of `differentiability' for an arbitrary functional ${\cal V}:{\cal P}_\psi\rightarrow L$ taking values in a normed vector space $(L,\|\cdot\|_L)$. It is related to the general functional analytic concept of (tangential) ${\cal S}$-differentiability introduced by Sebasti\~{a}o e Silva \cite{Sebastiao e Silva1956a} and Averbukh and Smolyanov \cite{AverbukhSmolyanov1967}; see also \cite{Fernholz1983,Gilletal1989,Shapiro1990} for applications. However, ${\cal P}_{\psi}$ is $\textit{not}$ a vector space. This implies that Definition \mbox{\ref{def S differentiability}} differs from the classical notion of (tangential) ${\cal S}$-differentiability. For that reason we will use inverted commas and write `${\cal S}$-differentiability' instead of ${\cal S}$-differentiability. Due to the missing vector space structure, we in particular need to allow the tangent space to depend on the point $\P\in{\cal P}_\psi$ at which ${\cal V}$ is differentiated. The role of the `tangent space' will be played by the set
\begin{equation*}\label{tangential space}
    {\cal P}_{\psi}^{\P;\pm} := \{\Q-\P:\,\Q\in{\cal P}_\psi\}
\end{equation*}
whose elements $\Q-\P:=(Q_0-P_0,\ldots,Q_{N-1}-P_{N-1})$ can be seen as signed transition functions. In Definition \mbox{\ref{def S differentiability}} we will employ the following terminology.

\begin{definition}\label{def continuity functional}
Let $\M\subseteq\M_\psi(E)$, $\phi$ be another gauge function, and fix $\P\in{\cal P}_\psi$. A map ${\cal W}:{\cal P}_{\psi}^{\P;\pm}\rightarrow L$ is said to be $(\M,\phi)$-continuous if the mapping $\Q\mapsto{\cal W}(\Q-\P)$ from ${\cal P}_\psi$ to $L$ is $(d_{\infty,\M}^{\phi},\|\cdot\|_L)$-continuous.
\end{definition}

For the following definition it is important to note that $\P + \varepsilon(\Q - \P)$ lies in ${\cal P}_\psi$ for any $\P,\Q\in{\cal P}_\psi$ and $\varepsilon\in(0,1]$.

\begin{definition}[`${\cal S}$-differentiability']\label{def S differentiability}
Let $\M\subseteq\M_\psi(E)$, $\phi$ be another gauge function, and fix $\P\in{\cal P}_{\psi}$. Moreover let ${\cal S}$ be a system of subsets of ${\cal P}_{\psi}$. A map ${\cal V}:{\cal P}_{\psi}\rightarrow L$ is said to be `${\cal S}$-differentiable' at $\P$ w.r.t.\ $(\M,\phi)$ if there exists an $(\M,\phi)$-continuous map $\dot{\cal V}_{\P}:{\cal P}_{\psi}^{\P;\pm}\rightarrow L$ such that
\begin{equation}\label{def eq for S-like D}
    \lim_{m\to\infty}\Big\|\frac{{\cal V}(\P+\varepsilon_m(\Q-\P))-{\cal V}(\P)}{\varepsilon_m} - \dot{\cal V}_{\P}(\Q-\P)\Big\|_L=0  \quad\mbox{uniformly in $\Q\in{\cal K}$}
\end{equation}
for every ${\cal K}\in{\cal S}$ and every sequence $(\varepsilon_m)\in(0,1]^{\N}$ with $\varepsilon_m\to 0$. In this case, $\dot{\cal V}_{\P}$ is called `${\cal S}$-derivative' of ${\cal V}$ at $\P$ w.r.t.\ $(\M,\phi)$.
\end{definition}

Note that in Definition \mbox{\ref{def S differentiability}} the derivative is {\em not} required to be linear (in fact the derivative is not even defined on a vector space). This is another point where Definition \mbox{\ref{def S differentiability}} differs from the functional analytic definition of (tangential) ${\cal S}$-differentiabil\-ity. However, non-linear derivatives are common in the field of mathematical optimization; see, for instance, \cite{Roemisch2004,Shapiro1990}.

\begin{remarknorm}\label{remark S differentiability}
(i) At least in the case $L=\R$, the `${\cal S}$-derivative' $\dot {\cal V}_{\P}$ evaluated at $\Q - \P$, i.e.\ $\dot{\cal V}_{\P}(\Q-\P)$, can be seen as a measure for the first-order sensitivity of the functional ${\cal V}: {\cal P}_{\psi}\to\R$ w.r.t.\ a change of the argument from $\P$ to $(1-\varepsilon)\P + \varepsilon\Q$, with $\varepsilon>0$ small, for some given transition function $\Q$.

(ii) The prefix `${\cal S}$-' in Definition \mbox{\ref{def S differentiability}} provides the following information. Since the convergence in (\mbox{\ref{def eq for S-like D}}) is required to be uniform in $\Q\in{\cal K}$, the values of the first-order sensitivities $\dot{\cal V}_{\P}(\Q-\P)$, $\Q\in{\cal K}$, can be compared with each other with clear conscience for any fixed ${\cal K}\in{\cal S}$. It is therefore favorable if the sets in ${\cal S}$ are large. However, the larger the sets in ${\cal S}$, the stricter the condition of `${\cal S}$-differentiability'.

(iii) The subset $\M$ ($\subseteq\M_\psi(E)$) and the gauge function $\phi$ tell us in a way how `robust' the `${\cal S}$-derivative' $\dot {\cal V}_{\P}$ is w.r.t.\ changes in $\boldsymbol{Q}$: The smaller the set $\M$ and the `steeper' the gauge function $\phi$, the less strict the metric $d_{\infty,\M}^{\phi}(\P,\Q)$ (given by (\mbox{\ref{def metric transition functions}})), and therefore the more robust $\dot{\cal V}_{\P}(\Q-\P)$ in $\boldsymbol{Q}$. It is thus favorable if the set $\M$ is small and the gauge function $\phi$ is `steep'. However, the smaller $\M$ and the `steeper' $\phi$, the stricter the condition of `${\cal S}$-differentiability'. More precisely, if $\M_1\subseteq\M_2$ and $\phi_1\ge\phi_2$ then `${\cal S}$-differentiability' w.r.t.\ $(\M_1,\phi_1)$ implies `${\cal S}$-differentiability' w.r.t.\ $(\M_2,\phi_2)$. Also note that in general the choice of ${\cal S}$ in Definition \mbox{\ref{def S differentiability}} is {\em not} influenced by the choice of the pair $(\M,\phi)$, and vice versa.
{\hspace*{\fill}$\Diamond$\par\bigskip}
\end{remarknorm}

In the general framework of our main result (Theorem \mbox{\ref{Thm - hadamard cal v}}) we can {\em not} choose $\phi$ `steeper' than the gauge function $\psi$ which plays the role of a bounding function there. Indeed, the proof of $(\M,\psi)$-continuity of the map $\dot{\cal V}_{\P}:{\cal P}_{\psi}^{\P;\pm}\rightarrow\R$ in Theorem \mbox{\ref{Thm - hadamard cal v}} does {\em not} work anymore if $d_{\infty,\M}^{\psi}$ is replaced by $d_{\infty,\M}^{\phi}$ for any gauge function $\phi$ `steeper' than $\psi$. And here it does {\em not} matter how exactly ${\cal S}$ is chosen.

In the application in Section \mbox{\ref{Sec - Finance Example}}, the set $\{\Q_{\Delta,\tau}: \Delta\in[0,\delta]\}$ should be contained in ${\cal S}$ (for details see Remark \mbox{\ref{Ex fin - remark relatively compactness}}). This set can be shown to be (relatively) compact w.r.t.\ $d_{\infty,\M}^{\phi}$ for $\phi(x)=\psi(x)$ ($:=1+u_{\alpha}(x)$) but not for any `flatter' gauge function $\phi$. So, in this example, and certainly in many other examples, relatively compact subsets of ${\cal P}_\psi$ w.r.t.\ $d_{\infty,\M}^{\psi}$ should be contained in ${\cal S}$. It is thus often beneficial to know that the value functional is `differentiable' in the sense of part (b) of the following Definition \mbox{\ref{def Gateaux-Hadamard-Frechet differentiability}}.

The terminology of Definition \mbox{\ref{def Gateaux-Hadamard-Frechet differentiability}} is motivated by the functional analytic analogues. Bounded and relatively compact sets in the (semi-) metric space $({\cal P}_\psi,d_{\infty,\M}^{\phi})$ are understood in the conventional way. A set ${\cal K}\subseteq{\cal P}_\psi$ is said to be bounded (w.r.t.\ $d_{\infty,\M}^{\phi}$) if there exist $\P'\in{\cal P}_\psi$ and $\delta>0$ such that $d_{\infty,\M}^{\phi}(\Q,\P')\le \delta$ for every $\Q\in{\cal K}$. It is said to be relatively compact (w.r.t.\ $d_{\infty,\M}^{\phi}$) if for every sequence $(\Q_m)\in{\cal K}^{\N}$ there exists a subsequence $(\Q'_m)$ of $(\Q_m)$ such that $d_{\infty,\M}^{\phi}(\Q_m',\Q)\to 0$ for some $\Q\in{\cal P}_\psi$. The system of all bounded sets and the system of all relatively compact sets (w.r.t.\ $d_{\infty,\M}^{\phi}$) are larger the `steeper' the gauge function $\phi$ is.

\begin{definition}\label{def Gateaux-Hadamard-Frechet differentiability}
In the setting of Definition \mbox{\ref{def S differentiability}} we refer to `${\cal S}$-differentiability' as
\vspace{-1.5mm}
\begin{enumerate}
	\item[{\rm (a)}] `Gateaux--L\'{e}vy differentiability' if ${\cal S} = {\cal S}_{{\rm f}} := \{{\cal K}\subseteq{\cal P}_\psi : {\cal K} \mbox{ is finite}\}$.
	
	\item[{\rm (b)}] `Hadamard differentiability' if ${\cal S} = {\cal S}_{{\rm rc}} := \{{\cal K}\subseteq{\cal P}_\psi : {\cal K} \mbox{ is relatively compact}\}$.
	
	\item[{\rm (c)}] `Fr\'echet differentiability' if ${\cal S} = {\cal S}_{{\rm b}} := \{{\cal K}\subseteq{\cal P}_\psi : {\cal K} \mbox{ is bounded}\}$.
\end{enumerate}
\end{definition}

Clearly, `Fr\'echet differentiability' (of ${\cal V}$ at $\P$ w.r.t.\ $(\M,\phi)$) implies `Hadamard differentiability' which in turn implies `Gateaux--L\'{e}vy differentiability', each with the same `derivative'.

The last sentence before Definition \mbox{\ref{def Gateaux-Hadamard-Frechet differentiability}} and the second to last sentence in part (iii) of Remark \mbox{\ref{remark S differentiability}} together imply that `Hadamard (resp.\ Fr\'echet) differentiability' w.r.t.\ $(\M,\phi_1)$ implies `Hadamard (resp.\ Fr\'echet) differentiability' w.r.t.\ $(\M,\phi_2)$ when $\phi_1\ge\phi_2$.

The following lemma, whose proof can be found in Subsection \mbox{\ref{Subsec - Proof Lemma hd characterization}} of the supplementary material, provides an equivalent characterization of `Hadamard differentiability'.

\begin{lemma}\label{lemma HD characterization}
Let $\M\subseteq\M_\psi(E)$, $\phi$ be another gauge function, ${\cal V}:{\cal P}_\psi\rightarrow L$ be any map, and fix $\P\in{\cal P}_\psi$. Then the following two assertions hold.

{\rm (i)} If ${\cal V}$ is `Hadamard differentiable' at $\P$ w.r.t.\ $(\M,\phi)$ with `Hadamard derivative' $\dot{\cal V}_{\P}$, then we have for each triplet $(\Q, (\Q_m), (\varepsilon_m))\in{\cal P}_\psi\times{\cal P}_\psi^{\N}\times(0,1]^{\N}$ with $d_{\infty,\M}^{\phi}(\Q_m,\Q)\to 0$ and $\varepsilon_m\to 0$ that
\begin{eqnarray}\label{HD characterization}
    \lim_{m\to\infty}\Big\|\frac{{\cal V}(\P+\varepsilon_m(\Q_m-\P))-{\cal V}(\P)}{\varepsilon_m} - \dot{\cal V}_{\P}(\Q-\P)\Big\|_L = 0.
\end{eqnarray}

{\rm (ii)} If there exists an $(\M,\phi)$-continuous map $\dot{\cal V}_{\P}:{\cal P}_{\psi}^{\P;\pm}\rightarrow L$ such that (\mbox{\ref{HD characterization}}) holds for each triplet $(\Q, (\Q_m), (\varepsilon_m))\in{\cal P}_\psi\times{\cal P}_\psi^{\N}\times(0,1]^{\N}$ with $d_{\infty,\M}^{\phi}(\Q_m,\Q)\to 0$ and $\varepsilon_m\to 0$, then ${\cal V}$ is `Hadamard differentiable' at $\P$ w.r.t.\ $(\M,\phi)$ with `Hada\-mard derivative' $\dot{\cal V}_{\P}$.
\end{lemma}


\subsection{`Differentiability' of the value functional} \label{Subsec - Hadamard differentiability of the optimal value functional}


Recall that $\boldsymbol{A}$, $\Pi$, and $\boldsymbol{r}$ are fixed, and let $V_{n}^{\P;\pi}$ and $V_{n}^{\P}$ be defined as in (\mbox{\ref{exp tot reward}}) and (\mbox{\ref{value function}}), respectively. Moreover let $\psi$ be any gauge function and fix some ${\cal P}_\psi\subseteq\overline{\cal P}_\psi$ being closed under mixtures.

In view of Lemma \mbox{\ref{lemma suff cond for stand cond - pre}} (with ${\cal P}':=\{\P\}$), condition (a) of Theorem \mbox{\ref{Thm - hadamard cal v}} below ensures that Assumption {\bf (A)} is satisfied for {\em any} $\P\in{\cal P}_\psi$. Then for any $x_n\in E$, $\pi\in\Pi$, and $n=0,\ldots,N$ we may define under condition (a) of Theorem \mbox{\ref{Thm - hadamard cal v}} functionals ${\cal V}_{n}^{x_n;\pi}:{\cal P}_\psi\rightarrow\R$ and ${\cal V}_{n}^{x_n}:{\cal P}_\psi\rightarrow\R$ by
\begin{equation}\label{def cal v}
    {\cal V}_{n}^{x_n;\pi}(\P) := V_{n}^{\P;\pi}(x_n)\quad\mbox{ and }\quad{\cal V}_{n}^{x_n}(\P) := V_{n}^{\P}(x_n),
\end{equation}
respectively. Note that ${\cal V}_{n}^{x_n}(\P)$ specifies the maximal value for the expected total reward in the MDM (given state $x_n$ at time $n$) when the underlying transition function is $\P$. By analogy with the name `value function' we refer to ${\cal V}_{n}^{x_n}$ as {\em value functional given state $x_n$ at time $n$}. Part (ii) of Theorem \mbox{\ref{Thm - hadamard cal v}} provides (under some assumptions) an `Hadamard derivative' of the value functional ${\cal V}_{n}^{x_n}$ in the sense of Definition \mbox{\ref{def Gateaux-Hadamard-Frechet differentiability}}.

Conditions (b) and (c) of Theorem \mbox{\ref{Thm - hadamard cal v}} involve the so-called {\em Minkowski (or gauge) functional} $\rho_{\M}:\M_\psi(E)\rightarrow\overline{\R}_+$ (see, e.g., \cite[p.\,25]{Rudin1991}) defined by
\begin{equation}\label{def minkowski functional}
	\rho_{\M}(h) := \inf\big\{\lambda\in\R_{>0}:\, h/\lambda\in\M\big\},
\end{equation}
where we use the convention $\inf\emptyset :=\infty$, $\M$ is any subset of $\M_\psi(E)$, and we set $\R_{>0}:=(0,\infty)$. We note that M\"uller \cite{Mueller1997} also used the Minkowski functional to formulate his assumptions.

\begin{examplenorm}\label{example minkowski}
For the sets $\M$ (and the corresponding gauge functions $\psi$) from Examples \mbox{\ref{examples prob metric - tv}}--\mbox{\ref{examples prob metric - hoelder}} we have $\rho_{\overline{\M}_{\rm{\scriptsize{TV}}}}(h) = {\rm sp}(h)$, $\rho_{\overline{\M}_{\rm{\scriptsize{Kolm}}}}(h) = \mathbb{V}(h)$, $\rho_{\M_{\rm{\scriptsize{BL}}}}(h) = \|h\|_{\rm BL}$, $\rho_{\M_{\rm{\scriptsize{Kant}}}}(h) = \|h\|_{\rm Lip}$, and $\rho_{\M_{\rm{\scriptsize{H\ddot{o}l}},\alpha}}(h) = \|h\|_{\rm{\scriptsize{H\ddot{o}l}},\alpha}$, where as before $\overline{\M}_{\rm{\scriptsize{TV}}}$ and $\overline{\M}_{\rm{\scriptsize{Kolm}}}$ are used to denote the maximal generator of $d_{\rm{\scriptsize{TV}}}$ and $d_{\rm{\scriptsize{Kolm}}}$, respectively. The latter three equations are trivial, for the former two equations see \cite[p.\,880]{Mueller1997}.
{\hspace*{\fill}$\Diamond$\par\bigskip}
\end{examplenorm}

Recall from Definition \mbox{\ref{def opt strategy}} that for given $\P\in{\cal P}_\psi$ and $\delta>0$ the sets $\Pi(\P;\delta)$ and $\Pi(\P)$ consist of all $\delta$-optimal strategies w.r.t.\ $\P$ and of all optimal strategies w.r.t.\ $\P$, respectively. Generators $\M'$ of $d_\M$ were introduced subsequent to (\mbox{\ref{def integral prob metric}}).

\begin{theorem}[`Differentiability' of ${\cal V}_{n}^{x_n;\pi}$ and ${\cal V}_{n}^{x_n}$]\label{Thm - hadamard cal v}
Let $\M\subseteq\M_\psi(E)$ and $\M'$ be any generator of $d_\M$. Fix $\P=(P_n)_{n=0}^{N-1}\in{\cal P}_\psi$, and assume that the following three conditions hold.
\begin{enumerate}
    \item[{\rm (a)}] $\psi$ is a bounding function for the MDM $(\boldsymbol{X},\boldsymbol{A},\Q,\Pi,\boldsymbol{r})$ for any $\Q\in{\cal P}_\psi$.

    \item[{\rm (b)}] $\sup_{\pi\in\Pi}\rho_{\M'}(V_n^{\P;\pi}) < \infty$ for any $n=1,\ldots,N$.

    \item[{\rm (c)}] $\rho_{\M'}(\psi) < \infty$.
\end{enumerate}
Then the following two assertions hold.
\begin{enumerate}
    \item[{\rm (i)}] For any $x_n\in E$, $\pi=(f_n)_{n=0}^{N-1}\in\Pi$, $n=0,\ldots,N$, the map ${\cal V}_{n}^{x_n;\pi}:{\cal P}_\psi\rightarrow\R$ defined by (\mbox{\ref{def cal v}}) is `Fr\'echet differentiable' at $\P$ w.r.t.\ $(\M,\psi)$ with `Fr\'echet derivative' $\dot{\cal V}_{n;\P}^{x_n;\pi}:{\cal P}_{\psi}^{\P;\pm}\rightarrow\R$ given by
        \begin{eqnarray} \label{hd ex tot costs}
           \lefteqn{\dot{\cal V}_{n;\P}^{x_n;\pi}(\Q-\P)} \\
            & \hspace{-2mm}:=\hspace{-2mm} &  \sum_{k=n+1}^{N-1}\sum_{j=n}^{k-1}\int_E\cdots\int_E r_k(y_k,f_k(y_k))\,P_{k-1}\big((y_{k-1},f_{k-1}(y_{k-1})),dy_k\big) \nonumber \\
            & & \quad \cdots (Q_j-P_j)\big((y_j,f_{j}(y_j)),dy_{j+1}\big)\cdots P_n\big((x_n,f_n(x_n)),dy_{n+1}\big)\nonumber \\
            & & +~ \sum_{j=n}^{N-1}\int_E\cdots\int_E r_N(y_N)\,P_{N-1}\big((y_{N-1},f_{N-1}(y_{N-1})),dy_N\big)\nonumber \\
            & & \quad \cdots (Q_j-P_j)\big((y_j,f_{j}(y_j)),dy_{j+1}\big)\cdots P_n\big((x_n,f_n(x_n)),dy_{n+1}\big). \nonumber
        \end{eqnarray}

    \item[{\rm (ii)}] For any $x_n\in E$ and $n=0,\ldots,N$, the map ${\cal V}_{n}^{x_n}:{\cal P}_\psi\rightarrow\R$ defined by (\mbox{\ref{def cal v}}) is `Hadamard differentiable' at $\P$ w.r.t.\ $(\M,\psi)$ with `Hadamard derivative' $\dot{\cal V}_{n;\P}^{x_n}:{\cal P}_{\psi}^{\P;\pm}\rightarrow\R$ given by
        \begin{eqnarray}\label{hd value function - 0}
            \dot{\cal V}_{n;\P}^{x_n}(\Q-\P) := \lim_{\delta\searrow 0}\,\sup_{\pi\in\Pi(\P;\delta)}\dot{\cal V}_{n;\P}^{x_n;\pi}(\Q-\P).
        \end{eqnarray}
        If the set of optimal strategies $\Pi(\P)$ is non-empty, then the `Hadamard derivative' admits the representation
        \begin{eqnarray}\label{hd value function}
            \dot{\cal V}_{n;\P}^{x_n}(\Q-\P)
            = \sup_{\pi\in\Pi(\P)}\dot{\cal V}_{n;\P}^{x_n;\pi}(\Q-\P).
        \end{eqnarray}
\end{enumerate}
\end{theorem}

The proof of Theorem \mbox{\ref{Thm - hadamard cal v}} can be found in Section \mbox{\ref{Sec - Proof of Theorem HD}} of the supplementary material. Note that the set $\Pi(\P;\delta)$ shrinks as $\delta$ decreases. Therefore the right-hand side of (\mbox{\ref{hd value function - 0}}) is well defined. The supremum in (\mbox{\ref{hd value function}}) ranges over all optimal strategies w.r.t.\ $\P$. If, for example, the MDM $(\boldsymbol{X},\boldsymbol{A},\P,\Pi,\boldsymbol{r})$ satisfies conditions (a)--(c) of Theorem \mbox{\ref{Thm - existence opt strategy}} in the supplementary material, then by part (iii) of this theorem an optimal strategy can be found, i.e.\ $\Pi(\P)$ is non-empty. The existence of an optimal strategy is also ensured if the sets $F_0,\ldots,F_{N-1}$ are finite (a situation one often faces in applications). In the latter case the `Hadamard derivative' $\dot{\cal V}_{n;\P}^{x_n}(\Q-\P)$ can easily be determined by computing the finitely many values $\dot{\cal V}_{n;\P}^{x_n;\pi}(\Q-\P)$, $\pi\in\Pi(\P)$, and taking their maximum. The discrete case will be discussed in more detail in Subsection \mbox{\ref{Subsec - Hadamard differentiability optimal value functional - discrete case}} of the supplementary material.

If there exists a unique optimal strategy $\pi^{\P}\in\Pi$ w.r.t.\ $\P$, then $\Pi(\P)$ is nothing but the singleton $\{\pi^{\P}\}$, and in this case the `Hadamard derivative' $\dot{\cal V}_{0;\boldsymbol{P}}^{x_0}$ of the optimal value (functional) ${\cal V}_0^{x_0}$ 
at $\P$ coincides with $\dot{\cal V}_{0;\P}^{x_0;\pi^{\P}}$.

\begin{remarknorm}\label{remark thm hd cal v}
(i) The `Fr\'echet differentiability' in part (i) of Theorem \mbox{\ref{Thm - hadamard cal v}} holds even uniformly in $\pi\in\Pi$; see Theorem \mbox{\ref{Thm - hadamard upsilon}} in the supplementary material for the precise meaning.

(ii) We do not know if it is possible to replace `Hadamard differentiability' by `Fr\'echet differentiability' in part (ii) of Theorem \mbox{\ref{Thm - hadamard cal v}}. The following arguments rather cast doubt on this possibility. The proof of part (ii) is based on the decomposition of the value functional ${\cal V}_{n}^{x_n}$ in display (\mbox{\ref{representation cal v}}) of the supplementary material and a suitable chain rule, where the decomposition (\mbox{\ref{representation cal v}})
involves the sup-functional $\Psi$ introduced in display (\mbox{\ref{def of Upsilon}}) of the supplementary material. However, Corollary 1 in \cite{CoxNadler1971} (see also Proposition 4.6.5 in \cite{Schirotzek2007}) shows that in normed vector spaces sup-functionals are in general {\em not} Fr\'echet differentiable. This could be an indication that `Fr\'echet differentiable' of the value functional indeed fails. We can not make a reliable statement in this regard.

(iii) Recall that `Hadamard (resp.\ Fr\'echet) differentiability' w.r.t.\ $(\M,\psi)$ implies `Hadamard (resp.\ Fr\'echet) differentiability' w.r.t.\ $(\M,\phi)$ for any gauge function $\phi\le\psi$. However, for any such $\phi$ `Hadamard (resp.\ Fr\'echet) differentiability' w.r.t.\ $(\M,\phi)$ is less meaningful than w.r.t.\ $(\M,\psi)$. Indeed, when using $d_{\infty,\M}^{\phi}$ with $\phi\le\psi$ instead of $d_{\infty,\M}^{\psi}$, the sets ${\cal K}$ for whose elements the first-order sensitivities can be compared with each other with clear conscience are smaller and the `derivative' is less robust.

(iv) In the case where we are interested in minimizing expected total costs in the MDM $(\boldsymbol{X},\boldsymbol{A},\P,\Pi,\boldsymbol{r})$ (see Remark \mbox{\ref{remark non Markovian strat}}(ii)), we obtain under the assumptions (and with the same arguments as in the proof of part (ii)) of Theorem \mbox{\ref{Thm - hadamard cal v}} that the `Hadamard derivative' of the corresponding value functional is given by (\mbox{\ref{hd value function - 0}}) (resp.\ (\mbox{\ref{hd value function}})) with ``$\sup$'' replaced by ``$\inf$''.
{\hspace*{\fill}$\Diamond$\par\bigskip}
\end{remarknorm}

\begin{remarknorm}\label{remark verification assumptions}
(i) Condition (a) of Theorem \mbox{\ref{Thm - hadamard cal v}} is in line with the existing literature. In fact, similar conditions as in Definition \mbox{\ref{def bounding function}} (with ${\cal P}':=\{\Q\}$) have been imposed many times before; see, for instance, \cite[Definition 2.4.1]{BaeuerleRieder2011}, \cite[Definition 2.4]{Mueller1997}, \cite[p.\,231\,ff]{Puterman1994}, and \cite{Wessels1977}.

(ii) In some situations, condition (a) implies condition (b) in Theorem \mbox{\ref{Thm - hadamard cal v}}. This is the case, for instance, in the following four settings (the involved sets $\M'$ were introduced in Examples \mbox{\ref{examples prob metric - tv}}--\mbox{\ref{examples prob metric - hoelder}}).

1) $\M':=\overline{\M}_{\rm{\scriptsize{TV}}}$ and $\psi:\equiv 1$.

2) $\M':=\overline{\M}_{\rm{\scriptsize{Kolm}}}$ and $\psi:\equiv 1$, as well as for $n=1,\ldots,N-1$\\
\noindent\hspace*{8.2mm} -\ $\int_{\R}V_{n+1}^{\P;\pi}(y)\,P_n((\,\cdot\,,f_n(\,\cdot\,)),dy)$, $\pi=(f_n)_{n=0}^{N-1}\in\Pi$, are increasing,\\	
\noindent\hspace*{8.2mm} -\ $r_n(\,\cdot\,,f_n(\,\cdot\,))$, $\pi=(f_n)_{n=0}^{N-1}\in\Pi$, and $r_N(\cdot)$ are increasing.
			
3) $\M':=\M_{\rm{\scriptsize{BL}}}$ and $\psi:\equiv 1$, as well as for $n=1,\ldots,N-1$\\
\noindent\hspace*{8.2mm} -\ $\sup_{\pi=(f_n)_{n=0}^{N-1}\in\Pi}\sup_{x\not=y}d_{\rm{\scriptsize{BL}}}(P_n((x,f_n(x)),\,\bullet\,),P_n((y,f_n(y)),\,\bullet\,))/d_E(x,y)<\infty$,\\	
\noindent\hspace*{8.2mm} -\ $\sup_{\pi=(f_n)_{n=0}^{N-1}\in\Pi}\|r_n(\,\cdot\,,f_n(\,\cdot\,))\|_{\rm{\scriptsize{Lip}}}<\infty$ and $\|r_N\|_{\rm{\scriptsize{Lip}}}<\infty$.

4) $\M':=\M_{\rm{\scriptsize{H\ddot{o}l}},\alpha}$ and $\psi(x):=1+d_E(x,x')^\alpha$, as well as for $n=1,\ldots,N-1$\\
\noindent\hspace*{8.2mm} -\ $\sup_{\pi=(f_n)_{n=0}^{N-1}\in\Pi}\sup_{x\not=y}d_{\rm{\scriptsize{H\ddot{o}l}},\alpha}(P_n((x,f_n(x)),\,\bullet\,),P_n((y,f_n(y)),\,\bullet\,))/d_E(x,y)^\alpha<\infty$,\\
\noindent\hspace*{8.2mm} -\ $\sup_{\pi=(f_n)_{n=0}^{N-1}\in\Pi}\|r_n(\,\cdot\,,f_n(\,\cdot\,))\|_{\rm{\scriptsize{H\ddot{o}l}},\alpha}<\infty$ and $\|r_N\|_{\rm{\scriptsize{H\ddot{o}l}},\alpha}<\infty$ for some $x'\in E$ and\\
\noindent\hspace*{8mm} $\alpha\in(0,1]$. Recall that $\M_{\rm{\scriptsize{H\ddot{o}l}},\alpha}=\M_{\rm{\scriptsize{Kant}}}$ for $\alpha=1$.

\noindent The proof of (a)$\Rightarrow$(b) relies in setting 1) on Lemma \mbox{\ref{lemma suff cond for stand cond - pre}} (with ${\cal P}':=\{\P\}$) and in settings 2)--4) on Lemma \mbox{\ref{lemma suff cond for stand cond - pre}} (with ${\cal P}':=\{\P\}$) along with Proposition \mbox{\ref{proposition reward iteration}} of the supplementary material. The conditions in setting 2) are similar to those in parts (ii)--(iv) of Theorem 2.4.14 in \cite{BaeuerleRieder2011}, and the conditions in settings 3) and 4) are motivated by the statements in \cite[p.\,11f]{Hinderer2005}.

(iii) In many situations, condition (c) of Theorem \mbox{\ref{Thm - hadamard cal v}} holds trivially. This is the case, for instance, if $\M'\in\{\overline{\M}_{\rm{\scriptsize{TV}}},\overline{\M}_{\rm{\scriptsize{Kolm}}},\M_{\rm{\scriptsize{BL}}}\}$ and $\psi:\equiv 1$, or if $\M':=\M_{\rm{\scriptsize{H\ddot{o}l}},\alpha}$ and $\psi(x):=1+d_E(x,x')^\alpha$ for some fixed $x'\in E$ and $\alpha\in(0,1]$.

(iv) The conditions (b) and (c) of Theorem \mbox{\ref{Thm - hadamard cal v}} can also be verified directly in some cases; see, for instance, the proof of Lemma \mbox{\ref{lemma Ex-fin verification assumptions Thm hd}} in Subsection \mbox{\ref{Subsec - Finance Example - auxiliary lemmas proof Thm hd}} of the supplementary material.
{\hspace*{\fill}$\Diamond$\par\bigskip}
\end{remarknorm}

In applications it is not necessarily easy to specify the set $\Pi(\P)$ of all optimal strategies w.r.t.\ $\P$. While in most cases an optimal strategy can be found with little effort (one can use the Bellman equation; see part (i) of Theorem \mbox{\ref{Thm - existence opt strategy}} in Section \mbox{\ref{Sec - Existence of optimal strategies}} of the supplementary material), it is typically more involved to specify {\em all} optimal strategies or to show that the optimal strategy is unique. The following remark may help in some situations; for an application see Subsection \mbox{\ref{Subsec - Finance Example - Hadamard derivative  optimal value}}.

\begin{remarknorm}\label{remark hd value function subset}
In some situations it turns out that for {\em every} $\P\in{\cal P}_\psi$ the solution of the optimization problem (\mbox{\ref{maximization problem}}) does not change if $\Pi$ is replaced by a subset $\Pi'\subseteq\Pi$ (being independent of $\P$). Then in the definition (\mbox{\ref{value function}}) of the value function (at time $0$) the set $\Pi$ can be replaced by the subset $\Pi'$, and it follows (under the assumptions of Theorem \mbox{\ref{Thm - hadamard cal v}}) that in the representation (\mbox{\ref{hd value function}}) of the `Hadamard derivative' $\dot{\cal V}_{0;\P}^{x_0}$ of ${\cal V}_0^{x_0}$ at $\P$ the set $\Pi(\P)$ can be replaced by the set $\Pi'(\P)$ of all optimal strategies w.r.t.\ $\P$ from the subset $\Pi'$. Of course, in this case it suffices to ensure that conditions (a)--(b) of Theorem \mbox{\ref{Thm - hadamard cal v}} are satisfied for the subset $\Pi'$ instead of $\Pi$.
{\hspace*{\fill}$\Diamond$\par\bigskip}
\end{remarknorm}


\subsection{Two alternative representations of $\dot{\cal V}_{n;\P}^{x_n;\pi}$} \label{Subsec - alternative rep Frechet derivative}


In this subsection we present two alternative representations (see (\mbox{\ref{hadamard derivative ex tot costs - II}}) and (\mbox{\ref{remark hd - eq10}})) of the `Fr\'echet derivative' $\dot{\cal V}_{n;\P}^{x_n;\pi}$ in (\mbox{\ref{hd ex tot costs}}). The representation (\mbox{\ref{hadamard derivative ex tot costs - II}}) will be beneficial for the proof
of Theorem \mbox{\ref{Thm - hadamard cal v}} (see Lemma \mbox{\ref{lemma continuity hd}} in Subsection \mbox{\ref{Subsec - differentiablility of upsilon}} of the supplementary material) and the representation (\mbox{\ref{remark hd - eq10}}) will be used to derive the `Hadamard derivative' of the optimal value of the terminal wealth problem in (\mbox{\ref{Ex fin - term wealth prob}}) below (see the proof of Theorem \mbox{\ref{Ex fin - Thm Hadamard}} in Subsection \mbox{\ref{Subsec - Finance Example - proof Thm hd}} of the supplementary material).

\begin{remarknorm}[Representation I]\label{remark hd repr 1}
By rearranging the sums in (\mbox{\ref{hd ex tot costs}}), we obtain under the assumptions of Theorem \mbox{\ref{Thm - hadamard cal v}} that for every fixed $\P=(P_n)_{n=0}^{N-1}\in{\cal P}_\psi$ the `Fr\'echet derivative' $\dot{\cal V}_{n;\P}^{x_n;\pi}$ of ${\cal V}_{n}^{x_n;\pi}$ at $\P$ can be represented as
\begin{equation} \label{hadamard derivative ex tot costs - II}
    \begin{aligned}
    \dot{\cal V}_{n;\P}^{x_n;\pi}(\Q-\P)
    & \,=\, \sum_{k=n}^{N-1}\int_E\int_E\cdots\int_E V_{k+1}^{\P;\pi}(y_{k+1})\,(Q_k - P_k)\big((y_k,f_k(y_k)),dy_{k+1}\big) \\
    &  \qquad P_{k-1}\big((y_{k-1},f_{k-1}(y_{k-1})),dy_{k}\big) \cdots P_n\big((x_n,f_n(x_n)),dy_{n+1}\big)
    \end{aligned}
\end{equation}
for every $x_n\in E$, $\Q=(Q_n)_{n=0}^{N-1}\in{\cal P}_\psi$, $\pi=(f_n)_{n=0}^{N-1}\in\Pi$, and $n=0,\ldots,N$.
{\hspace*{\fill}$\Diamond$\par\bigskip}
\end{remarknorm}

\begin{remarknorm}[Representation II]\label{remark hd - iteration scheme}
For every fixed $\P=(P_n)_{n=0}^{N-1}\in{\cal P}_\psi$, and under the assumptions of Theorem \mbox{\ref{Thm - hadamard cal v}}, the `Fr\'echet derivative' $\dot{\cal V}_{n;\P}^{x_n;\pi}$ of ${\cal V}_{n}^{x_n;\pi}$ at $\P$ admits the representation
\begin{equation}\label{remark hd - eq10}
    \dot{\cal V}_{n;\P}^{x_n;\pi}(\Q-\P)=\dot V_n^{\P,\Q;\pi}(x_n)
\end{equation}
for every $x_n\in E$, $\Q=(Q_n)_{n=0}^{N-1}\in{\cal P}_\psi$, $\pi=(f_n)_{n=0}^{N-1}\in\Pi$, and $n=0,\ldots,N$, where $(\dot V_k^{\P,\Q;\pi})_{k=0}^N$ is the solution of the following {\em backward iteration scheme}
\begin{equation}\label{Hadamard pi backward iteration scheme}
	\begin{aligned}
	   \dot V_N^{\P,\Q;\pi}(\cdot) & \,:=\, 0 \\
	   \dot V_k^{\P,\Q;\pi}(\cdot) & \,:=\, \int_E \dot V_{k+1}^{\P,\Q;\pi}(y)\,P_k\big((\,\cdot\,,f_k(\cdot)),dy\big) \\
       & \qquad + \int_E V_{k+1}^{\P;\pi}(y)\,(Q_k - P_k)\big((\,\cdot\,,f_k(\cdot)),dy\big),\qquad k=0,\ldots,N-1. \
	\end{aligned}
\end{equation}

Indeed, it is easily seen that $\dot V_n^{\P,\Q;\pi}(x_n)$ coincides with the right-hand side of (\mbox{\ref{hadamard derivative ex tot costs - II}}). Note that it can be verified iteratively by means of condition (a) of Theorem \mbox{\ref{Thm - hadamard cal v}} and Lemma \mbox{\ref{lemma suff cond for stand cond - pre}} (with ${\cal P}':=\{\Q\}$) that $\dot V_n^{\P,\Q;\pi}(\cdot)\in\M_\psi(E)$ for every $\Q\in{\cal P}_\psi$, $\pi\in\Pi$, and $n=0,\ldots,N$. In particular, this implies that the integrals on the right-hand side of (\mbox{\ref{Hadamard pi backward iteration scheme}}) exist and are finite. Also note that the iteration scheme (\mbox{\ref{Hadamard pi backward iteration scheme}}) involves the family $(V^{\P;\pi}_k)_{k=1}^N$ which itself can be seen as the solution of a backward iteration scheme:
\begin{eqnarray*}
    V_N^{\P;\pi}(\cdot) & := & r_N(\cdot)\\
	V_k^{\P;\pi}(\cdot) & := & r_k(\,\cdot\,,f_k(\cdot)) + \int_E V_{k+1}^{\P;\pi}(y)\,P_k\big((\,\cdot\,,f_k(\cdot)),dy\big),\qquad k=1,\ldots,N-1;
\end{eqnarray*}
see Proposition \mbox{\ref{proposition reward iteration}} of the supplementary material.
{\hspace*{\fill}$\Diamond$\par\bigskip}
\end{remarknorm}


\section{Application to a terminal wealth optimization problem in mathematical finance}\label{Sec - Finance Example}


In this section we will apply the theory of Sections \mbox{\ref{Sec - Formal definition of MDM}}--\mbox{\ref{Sec - Hadamard differentiability chapter}} to a particular optimization problem in mathematical finance. At first, we introduce in Subsection \mbox{\ref{Subsec - Finance Example - Basic financial market}} the basic financial market model and formulate subsequently the terminal wealth problem as a classical optimization problem in mathematical finance. The market model is in line with standard literature as \cite[Chapter 4]{BaeuerleRieder2011} or \cite[Chapter 5]{FoellmerSchied2011}. To keep the presentation as clear as possible we restrict ourselves to a simple variant of the market model (only one risky asset). In Subsection \mbox{\ref{Subsec - Finance Example - Embedding FM into MDM}} we will see that the market model can be embedded into the MDM of Section \mbox{\ref{Sec - Formal definition of MDM}}. It turns out that the existence (and computation) of an optimal (trading) strategy can be obtained by solving iteratively $N$ one-stage investment problems; see Subsection \mbox{\ref{Subsec - Finance Example - Comp opt trading strat}}. In Subsection \mbox{\ref{Subsec - Finance Example - Hadamard derivative  optimal value}} we will specify the `Hadamard derivative' of the optimal value functional of the terminal wealth problem, and Subsection \mbox{\ref{Subsec - Finance Example - Numerical example}} provides some numerical examples.


\subsection{Basic financial market model, and the target}\label{Subsec - Finance Example - Basic financial market}


Consider an $N$-period financial market consisting of one riskless bond $B=(B_0,\break\ldots,B_N)$ and one risky asset $S=(S_0,\ldots,S_N)$. Further assume that the value of the bond evolves deterministically according to
$$
    B_0= 1,\qquad		B_{n+1} = \mathfrak{r}_{n+1} B_n,	\qquad n=0,\ldots,N-1
$$
for some fixed constants $\mathfrak{r}_1,\ldots,\mathfrak{r}_N\in\R_{\ge 1}$, and that the value of the asset evolves stochastically according to
$$
    S_0>0,\qquad 	S_{n+1} = \mathfrak{R}_{n+1} S_n,	\qquad n=0,\ldots,N-1
$$
for some independent $\R_{\ge 0}$-valued random variables $\mathfrak{R}_1,\ldots,\mathfrak{R}_N$ on some probability space $\OFP$ with distributions $\mathfrak{m}_1,\ldots,\mathfrak{m}_N$, respectively.

Throughout Section \mbox{\ref{Sec - Finance Example}} we will assume that the financial market satisfies the following Assumption {\bf (FM)}, where $\alpha\in(0,1)$ is fixed and chosen as in (\mbox{\ref{Ex fin - def power utility}}) below. In Examples \mbox{\ref{Ex fin - example CRR model}} and \mbox{\ref{Ex fin - example BSM model}} we will discuss specific financial market models which satisfy Assumption {\bf (FM)}.

\par\medskip

\noindent{\bf Assumption (FM)}:
The following three assertions hold for any $n=0,\ldots,N-1$.
\vspace{-2mm}
\begin{enumerate}
	\item[{\rm (a)}] $\int_{\R_{\ge 0}}y^\alpha\,\mathfrak{m}_{n+1}(dy)<\infty$.

	\item[{\rm (b)}] $\mathfrak{R}_{n+1}>0$ $\pr$-a.s.
	
    \item[{\rm (c)}] $\pr[\mathfrak{R}_{n+1} \neq \mathfrak{r}_{n+1}]=1$.
\end{enumerate}

\par\medskip

Note that for any $n=0,\ldots,N-1$ the value $\mathfrak{r}_{n+1}$ (resp.\ $\mathfrak{R}_{n+1}$) corresponds to the relative price change $B_{n+1}/B_n$ (resp.\ $S_{n+1}/S_n$) of the bond (resp.\ asset) between time $n$ and $n+1$. Let ${\cal F}_0$ be the trivial $\sigma$-algebra, and set ${\cal F}_n:=\sigma(S_0,\ldots,S_n)=\sigma(\mathfrak{R}_1,\ldots,\mathfrak{R}_n)$ for any $n=1,\ldots,N$.

Now, an agent invests a given amount of capital $x_0\in\R_{\ge 0}$ in the bond and the asset according to some self-financing trading strategy. By {\em trading strategy} we mean an $({\cal F}_n)$-adapted $\R_{\ge 0}^2$-valued stochastic process $\varphi=(\varphi_n^0,\varphi_n)_{n=0}^{N-1}$, where $\varphi_n^0$ (resp.\ $\varphi_n$) specifies the amount of capital that is invested in the bond (resp.\ asset) during the time interval $[n,n+1)$. Here we require that both $\varphi_n^0$ and $\varphi_n$ are nonnegative for any $n$, which means that taking loans and short sellings of the asset are excluded. The corresponding {\em portfolio process} $X^{\varphi}=(X_0^{\varphi},\ldots,X_N^{\varphi})$ associated with $\varphi=(\varphi_n^0,\varphi_n)_{n=0}^{N-1}$ is given by
$$
    X_0^{\varphi} := \varphi_0^0+\varphi_0 \quad\mbox{ and }\quad X_{n+1}^{\varphi} := \varphi_n^0\mathfrak{r}_{n+1} + \varphi_n\mathfrak{R}_{n+1},\qquad n=0,\ldots,N-1.
$$
A trading strategy $\varphi=(\varphi_n^0,\varphi_n)_{n=0}^{N-1}$ is said to be {\em self-financing w.r.t.\ the initial capital $x_0$} if $x_0=\varphi_0^0+\varphi_0$ and $X_n^{\varphi}=\varphi_n^0+\varphi_n$ for all $n=1,\ldots,N$. It is easily seen that for any self-financing trading strategy $\varphi=(\varphi_n^0,\varphi_n)_{n=0}^{N-1}$ w.r.t.\ $x_0$ the corresponding portfolio process admits the representation
\begin{equation}\label{Ex fin - rep portfolio process}
    X_0^{\varphi} = x_0 \quad\mbox{ and }\quad X_{n+1}^{\varphi} = \mathfrak{r}_{n+1} X_n^{\varphi} + \varphi_n(\mathfrak{R}_{n+1}-\mathfrak{r}_{n+1}) \quad\mbox{for }n=0,\ldots,N-1.
\end{equation}
Note that $X_n^{\varphi}-\varphi_n$ corresponds to the amount of capital which is invested in the bond between time $n$ and $n+1$. Also note that it can be verified easily by means of Remark 3.1.6 in \cite{BaeuerleRieder2011} that under condition (c) of Assumption {\bf (FM)} the financial market introduced above is free of arbitrage opportunities.

In view of (\mbox{\ref{Ex fin - rep portfolio process}}), we may and do identify a self-financing trading strategy w.r.t.\ $x_0$ with an $({\cal F}_n)$-adapted $\R_{\ge 0}$-valued stochastic process $\varphi=(\varphi_n)_{n=0}^{N-1}$ satisfying $\varphi_0\in[0,x_0]$ and $\varphi_n\in[0,X_n^{\varphi}]$ for all $n=1,\ldots,N-1$. We restrict ourselves to {\em Markovian} self-financing trading strategies $\varphi=(\varphi_n)_{n=0}^{N-1}$ w.r.t.\ $x_0$ which means that $\varphi_n$ only depends on $n$ and $X_n^{\varphi}$. To put it another way, we assume that for any $n=0,\ldots,N-1$ there exists some Borel measurable map $f_n:\R_{\ge 0}\to\R_{\ge 0}$ such that
$
    \varphi_n = f_n(X_n^{\varphi}).
$
Then, in particular, $X^{\varphi}$ is an $\R_{\ge 0}$-valued $({\cal F}_n)$-Markov process whose one-step transition probability at time $n\in\{0,\ldots,N-1\}$ given state $x_n\in\R_{\ge 0}$ and strategy $\varphi=(\varphi_n)_{n=0}^{N-1}$ (resp.\ $\pi=(f_n)_{n=0}^{N-1}$) is given by
$
    \mathfrak{m}_{n+1}\circ\eta_{n,(x,f_n(x))}^{-1}
$
with
\begin{equation}\label{Ex fin - def mapping eta}
	 \eta_{n,(x,f_n(x))}(y) := \mathfrak{r}_{n+1}x + f_n(x)(y - \mathfrak{r}_{n+1}), \qquad y\in\R_{\ge 0}.
\end{equation}

The agent's aim is to find a self-financing trading strategy $\varphi=(\varphi_n)_{n=0}^{N-1}$ (resp.\ $\pi=(f_n)_{n=0}^{N-1}$) w.r.t.\ $x_0$ for which her expected utility of the discounted 
terminal wealth is maximized. We assume that the agent is risk averse and that her attitude towards risk is set via the {\em power utility} function $u_\alpha:\R_{\ge 0}\rightarrow\R_{\ge 0}$ defined by
\begin{equation}\label{Ex fin - def power utility}
   u_\alpha(y) := y^\alpha
\end{equation}
for some fixed $\alpha\in(0,1)$ (as in Assumption {\bf (FM)}). The coefficient $\alpha$ determines the degree of risk aversion of the agent: the smaller the coefficient $\alpha$, the greater her risk aversion. Hence the agent is interested in those self-financing trading strategies $\varphi=(\varphi_n)_{n=0}^{N-1}$ (resp.\ $\pi=(f_n)_{n=0}^{N-1}$) w.r.t.\ $x_0$ for which the expectation of $u_\alpha(X_N^{\varphi}/B_N)$ under $\pr$ is maximized.

In the following subsections we will assume for notational simplicity that $\mathfrak{r}_{1},\ldots,\mathfrak{r}_{N}$ are fixed and that $\mathfrak{m}_1,\ldots,\mathfrak{m}_N$ are a sort of model parameters. In this case the factor $1/B_N$ in $u_\alpha(X_N^{\varphi}/B_N)$ in display (\mbox{\ref{Ex fin - terminal reward function}}) is superfluous; it indeed does not influence the maximization problem or any `derivative' of the optimal value. On the other hand, if also the (Dirac-) distributions of $\mathfrak{r}_{1},\ldots,\mathfrak{r}_{N}$ would be allowed to be variable, then this factor could matter for the derivative of the optimal value w.r.t.\ changes in the (deterministic) dynamics of $B_N$.


\subsection{Embedding into MDM, and optimal trading strategies} \label{Subsec - Finance Example - Embedding FM into MDM}


The setting introduced in Subsection \mbox{\ref{Subsec - Finance Example - Basic financial market}} can be embedded into the setting of Sections \mbox{\ref{Sec - Formal definition of MDM}}--\mbox{\ref{Sec - Hadamard differentiability chapter}} as follows. Let $\mathfrak{r}_1,\ldots,\mathfrak{r}_N\in\R_{\ge 1}$ be a priori fixed constants. Let
$
	 (E,{\cal E}) := (\R_{\ge 0},{\cal B}(\R_{\ge 0})),
$
and
$	
	  A_n(x) := [0,x] 
$
for any $x\in\R_{\ge 0}$ and $n=0,\ldots,N-1$. Then 
$
   A_n = \R_{\ge 0}
$
and
$
	D_n = D := \{(x,a)\in\R_{\ge 0}^2:\,a\in[0,x]\}.
$
Let ${\cal A}_n := {\cal B}(\R_{\ge 0})$. In particular, ${\cal D}_n = {\cal B}(\R_{\ge 0}^2)\cap D$ and the set $\mathbb{F}_n$ of all decision rules at time $n$ consists of all those Borel measurable functions $f_n:\R_{\ge 0}\rightarrow\R_{\ge 0}$ which satisfy $f_n(x)\in[0,x]$ for all $x\in\R_{\ge 0}$ (in particular $\mathbb{F}_n$ is independent of $n$). For any $n=0,\ldots,N-1$, let the set $F_n$ of all admissible decision rules at time $n$ be equal to $\mathbb{F}_n$.
Let as before $\Pi:=F_0\times\cdots\times F_{N-1}$.

Moreover let $r_n:\equiv 0$ for any $n=0,\ldots,N-1$, and
\begin{equation}\label{Ex fin - terminal reward function}
    r_N(x) := u_\alpha(x/B_N), \qquad x\in\R_{\ge 0}.
\end{equation}
Consider the gauge function $\psi:\R_{\ge 0}\rightarrow\R_{\ge 1}$ defined by
\begin{equation}\label{Ex fin - gauge function}
	\psi(x) := 1 + u_\alpha(x). 
\end{equation}
Let ${\cal P}_\psi$ be the set of all transition functions $\P=(P_n)_{n=0}^{N-1}\in{\cal P}$ consisting of transition kernels of the shape
\begin{equation}\label{Ex fin - set trans func psi}
    P_n\big((x,a),\,\bullet\,\big) := \mathfrak{m}_{n+1}\circ\eta_{n,(x,a)}^{-1}\,[\,\bullet\,],		\qquad (x,a)\in D_n,\,n=0,\ldots,N-1
\end{equation}
for some $\mathfrak{m}_{n+1}\in{\cal M}_1^{\alpha}(\R_{\ge 0})$, where ${\cal M}_1^{\alpha}(\R_{\ge 0})$ is the set of all $\mu\in{\cal M}_1(\R_{\ge 0})$ satisfying $\int_{\R_{\ge 0}}u_\alpha\,d\mu<\infty$, and the map $\eta_{n,(x,a)}$ is defined as in (\mbox{\ref{Ex fin - def mapping eta}}).
In particular, ${\cal P}_\psi\subseteq\overline{\cal P}_\psi$ (with $\overline{\cal P}_\psi$ defined as in Subsection \mbox{\ref{Subsec - Metric on set of transition functions}}), and it can be verified easily that $\psi$ given by (\mbox{\ref{Ex fin - gauge function}}) is a bounding function for the MDM $(\boldsymbol{X},\boldsymbol{A},\Q,\Pi,\boldsymbol{r})$ for any $\Q\in{\cal P}_\psi$ (see Lemma \mbox{\ref{lemma Ex-fin verification assumptions Thm hd}}(i) of the supplementary material). Note that $\boldsymbol{X}$ plays the role of the portfolio process $X^{\varphi}$ from Subsection \mbox{\ref{Subsec - Finance Example - Basic financial market}}. Also note that for some fixed $x_0\in\R_{\ge 0}$, any self-financing trading strategy $\varphi=(\varphi_n)_{n=0}^{N-1}$ w.r.t.\ $x_0$ may be identified with some $\pi=(f_n)_{n=0}^{N-1}\in\Pi$ via $\varphi_n=f_n(X_n^{\varphi})$.

Then, for every fixed $x_0\in\R_{\ge 0}$ and $\P\in{\cal P}_\psi$ the terminal wealth problem introduced at the very end of Subsection \mbox{\ref{Subsec - Finance Example - Basic financial market}} reads as
\begin{equation}\label{Ex fin - term wealth prob}
     \ex^{x_0,\P;\pi}[r_N(X_N)] \longrightarrow \max\ (\mbox{in $\pi\in\Pi$)\,!}
\end{equation}
A strategy $\pi^{\P}\in\Pi$ is called an {\em optimal (self-financing) trading strategy w.r.t.\ $\P$ (and $x_0$)} if it solves the maximization problem (\mbox{\ref{Ex fin - term wealth prob}}).

\begin{remarknorm}
In the setting of Subsection \mbox{\ref{Subsec - Finance Example - Basic financial market}} we restrict ourselves to Markovian self-financing trading strategies $\varphi=(\varphi_n)_{n=0}^{N-1}$ w.r.t.\ $x_0$ which
may be identified with some $\pi=(f_n)_{n=0}^{N-1}\in\Pi$ via $\varphi_n=f_n(X_n^{\varphi})$. Of course, one could also assume that the decision rules of a trading
strategy $\pi$ also depend on past actions and past values of the portfolio process $X^{\varphi}$. However, as already discussed in Remark \mbox{\ref{remark non Markovian strat}}(i), the corresponding history-dependent trading strategies do not lead to an improved optimal value for the terminal wealth problem (\mbox{\ref{Ex fin - term wealth prob}}).
{\hspace*{\fill}$\Diamond$\par\bigskip}
\end{remarknorm}


\subsection{Computation of optimal trading strategies}\label{Subsec - Finance Example - Comp opt trading strat}


In this subsection we discuss the existence and computation of solutions to the terminal wealth problem (\mbox{\ref{Ex fin - term wealth prob}}), maintaining the notation of Subsection \mbox{\ref{Subsec - Finance Example - Embedding FM into MDM}}. We will adapt the arguments of Section 4.2 in \cite{BaeuerleRieder2011}. As before
$\mathfrak{r}_1,\ldots,\mathfrak{r}_N\in\R_{\ge 1}$ are fixed constants.

Basically the existence of an optimal trading strategy for the terminal wealth problem (\mbox{\ref{Ex fin - term wealth prob}}) can be ensured with the help of a suitable analogue of Theorem 4.2.2 in \cite{BaeuerleRieder2011}. In order to specify the optimal trading strategy explicitly one has to determine the local maximizers in the Bellman equation; see Theorem \mbox{\ref{Thm - existence opt strategy}}(i) in Section \mbox{\ref{Sec - Existence of optimal strategies}} of the supplementary material. However this is not necessarily easy. On the other hand, part (ii) of Theorem \mbox{\ref{Ex fin - Thm opt trading strat}} ahead (a variant of Theorem 4.2.6 in \cite{BaeuerleRieder2011}) shows that, for our particular choice of the utility function (recall (\mbox{\ref{Ex fin - def power utility}})), the optimal investment in the asset at time $n\in\{0,\ldots,N-1\}$ has a rather simple form insofar as it depends linearly on the wealth. The respective coefficient can be obtained by solving the one-stage optimization problem in (\mbox{\ref{Ex fin - one-stage max problem}}) ahead. That is, instead of finding the optimal amount of capital (possibly depending on the wealth) to be invested in the asset, it suffices to find the optimal fraction of the wealth (being independent of the wealth itself) to be invested in the asset.

For the formulation of the one-stage optimization problem note that every transition function $\P\in{\cal P}_\psi$ is generated through (\mbox{\ref{Ex fin - set trans func psi}}) by some $(\mathfrak{m}_1,\ldots,\mathfrak{m}_N)\in{\cal M}_1^{\alpha}(\R_{\ge 0})^N$. For every $\P\in{\cal P}_\psi$, we use $(\mathfrak{m}_1^{\P},\ldots,\mathfrak{m}_N^{\P})$ to denote any such set of `parameters'. Now, consider for any $\P\in{\cal P}_\psi$ and $n=0,\ldots,N-1$ the optimization problem
\begin{equation}\label{Ex fin - one-stage max problem}
	v_n^{\P;\gamma} := \int_{\R_{\ge 0}}u_\alpha\Big(1+\gamma\Big(\frac{y}{\mathfrak{r}_{n+1}} - 1\Big)\Big)\,\mathfrak{m}_{n+1}^{\P}(dy)\,\longrightarrow\,\max\ (\mbox{in $\gamma\in[0,1]$)\,!}	
\end{equation}
Note that $1+\gamma(y/\mathfrak{r}_{n+1}-1)$ lies in $\R_{\ge 0}$ for any $\gamma\in[0,1]$ and $y\in\R_{\ge 0}$, and that the integral on the left-hand side (exists and) is finite (this follows from displays (\mbox{\ref{proof lem ex solution red opt prob - eq10}})--(\mbox{\ref{proof lem ex solution red opt prob - eq30}}) in Subsection \mbox{\ref{Subsec - Proof of Lem ex solution red opt prob}} of the supplementary material) and should be seen as the expectation of $u_\alpha(1+\gamma(\mathfrak{R}_{n+1}/\mathfrak{r}_{n+1} - 1))$ under $\pr$.

The following lemma, whose proof can be found in Subsection \mbox{\ref{Subsec - Proof of Lem ex solution red opt prob}} of the supplementary material, shows in particular that
$$
	v_n^{\P} := \sup_{\gamma\in[0,1]}v_n^{\P;\gamma}
$$
is the maximal value of the optimization problem (\mbox{\ref{Ex fin - one-stage max problem}}).

\begin{lemma}\label{Ex fin - lemma existence solution reduced opt prob}
For any $\P\in{\cal P}_\psi$ and $n=0,\ldots,N-1$, there exists a unique solution $\gamma_n^{\P}\in[0,1]$ to the optimization problem (\mbox{\ref{Ex fin - one-stage max problem}}).
\end{lemma}

Part (i) of the following Theorem \mbox{\ref{Ex fin - Thm opt trading strat}} involves the value function introduced in (\mbox{\ref{value function}}). In the present setting this function has a comparatively simple form:
\begin{equation}\label{Ex fin - rep value functions}
	V_n^{\P}(x_n) = \sup_{\pi\in\Pi}\ex_{n,x_n}^{x_0,\P;\pi}[r_N(X_N)]
\end{equation}
for any $x_n\in\R_{\ge 0}$, $\P\in{\cal P}_\psi$, and $n=0,\ldots,N$.

Part (ii) involves the subset $\Pi_{\rm lin}$ of $\Pi$ which consists of all {\em linear trading strategies}, i.e.\ of all $\pi\in\Pi$ of the form $\pi=(f_n^{\boldsymbol{\gamma}})_{n=0}^{N-1}$ for some $\boldsymbol{\gamma}=(\gamma_n)_{n=0}^{N-1}\in[0,1]^N$, where
\begin{equation}\label{Ex fin - linear decision rules}
    f_n^{\boldsymbol{\gamma}}(x) := \gamma_n\,x,\qquad x\in\R_{\ge 0},\,n=0,\ldots,N-1.
\end{equation}

In part (i) and elsewhere we use the convention that the product over the empty set is $1$.

\begin{theorem}[Optimal trading strategy]\label{Ex fin - Thm opt trading strat}
For any $\P\in{\cal P}_\psi$ the following two assertions hold.
\begin{itemize}
	\item[{\rm (i)}] The value function $V_n^{\P}$ given by (\mbox{\ref{Ex fin - rep value functions}}) admits the representation
    $$
    	V_n^{\P}(x_n) = \mathfrak{v}_n^{\P} u_\alpha(x_n/B_n)
    $$
    for any $x_n\in\R_{\ge 0}$ and $n=0,\ldots,N-1$, where $\mathfrak{v}_n^{\P} := \prod_{k=n}^{N-1}v_k^{\P}$.

	\item[{\rm (ii)}] For any $n=0,\ldots,N-1$, let $\gamma_n^{\P}\in[0,1]$ be the unique solution to the optimization problem (\mbox{\ref{Ex fin - one-stage max problem}}) and define a decision rule $f_n^{\P}:\R_{\ge 0}\to\R_{\ge 0}$ at time $n$ through
    \begin{equation}\label{Ex fin - optimal trading strategy}
       f_n^{\P}(x) := \gamma_n^{\P}x,\qquad x\in\R_{\ge 0}.
    \end{equation}
    Then $\pi^{\P}:=(f_n^{\P})_{n=0}^{N-1}\in\Pi_{\rm lin}$ forms an optimal trading strategy w.r.t.\ $\P$. Moreover, there is no further optimal trading strategy w.r.t.\ $\P$ which belongs to $\Pi_{\rm lin}$.
\end{itemize}
\end{theorem}

The proof of Theorem \mbox{\ref{Ex fin - Thm opt trading strat}} can be found in Subsection \mbox{\ref{Subsec - Finance Example - proof Thm opt trad strat}} of the supplementary material. The second assertion of part (ii) of Theorem \mbox{\ref{Ex fin - Thm opt trading strat}} will be beneficial for part (ii) of Theorem \mbox{\ref{Ex fin - Thm Hadamard}}; for details see Remark \mbox{\ref{Ex fin - remark opt linear strategies}}. The following two Examples \mbox{\ref{Ex fin - example CRR model}} and \mbox{\ref{Ex fin - example BSM model}} illustrate part (ii) of Theorem \mbox{\ref{Ex fin - Thm opt trading strat}}.

\begin{examplenorm}[Cox--Ross--Rubinstein model]\label{Ex fin - example CRR model}
Let $\mathfrak{r}_1=\cdots=\mathfrak{r}_N=\mathfrak{r}$ for some $\mathfrak{r}\in\R_{\ge 1}$. Moreover let $\P\in{\cal P}$ be any transition function defined as in (\mbox{\ref{Ex fin - set trans func psi}}) with $\mathfrak{m}_1=\cdots=\mathfrak{m}_N=\mathfrak{m}_{\P}$ for some $\mathfrak{m}_{\P} := p_{\P}\delta_{{\sf u}_{\P}} + (1-p_{\P})\delta_{{\sf d}_{\P}}$, where $p_{\P}\in[0,1]$ and ${\sf d}_{\P},{\sf u}_{\P}\in\R_{>0}$ are some given constants (depending on $\P$) satisfying ${\sf d}_{\P}<\mathfrak{r}<{\sf u}_{\P}$. Then $\P\in{\cal P}_\psi$ and conditions (a)--(c) of Assumption {\bf (FM)} are clearly satisfied. In particular, the corresponding financial market is arbitrage-free and the optimization problem (\mbox{\ref{Ex fin - one-stage max problem}}) simplifies to (up to the factor $\mathfrak{r}^{-\alpha}$)
\begin{equation}\label{Ex fin - optimization problem CRR model}
	\big\{p_{\P}\, u_\alpha(\mathfrak{r} + \gamma({\sf u}_{\P}-\mathfrak{r})) + (1-p_{\P})\, u_\alpha(\mathfrak{r} + \gamma({\sf d}_{\P}-\mathfrak{r}))\big\}\,\longrightarrow \max\ (\mbox{in $\gamma\in[0,1]$)\,!}
\end{equation}
Lemma \mbox{\ref{Ex fin - lemma existence solution reduced opt prob}} ensures that (\mbox{\ref{Ex fin - optimization problem CRR model}}) has a unique solution, $\gamma_{{\sf CRR}}^{\P}$, and it can be checked easily (see, e.g., \cite[p.\,86]{BaeuerleRieder2011}) that this solution admits the representation
\begin{equation}\label{Ex fin - solution CRR model}
    \gamma_{{\sf CRR}}^{\P} =
    \left\{
    \begin{array}{lll}
        0 & , & p_{\P}\in[0,p_{\P,0}]\\
        \frac{\mathfrak{r}}{(\mathfrak{r} - {\sf d}_{\P})({\sf u}_{\P} - \mathfrak{r})}\cdot\frac{p_{\P}^{\kappa_\alpha}({\sf u}_{\P} - \mathfrak{r})^{\kappa_\alpha} - (1-p_{\P})^{\kappa_\alpha}(\mathfrak{r} - {\sf d}_{\P})^{\kappa_\alpha}}{p_{\P}^{\kappa_\alpha}({\sf u}_{\P} - \mathfrak{r})^{\kappa_\alpha\alpha} + (1-p_{\P})^{\kappa_\alpha}(\mathfrak{r} - {\sf d}_{\P})^{\kappa_\alpha\alpha}} & , & p_{\P}\in(p_{\P,0},p_{\P,1})\\
        1 & , &  p_{\P}\in[p_{\P,1},1]
    \end{array}
    \right.,
\end{equation}
where $\kappa_\alpha:=(1-\alpha)^{-1}$ and
\begin{equation*}\label{Ex fin - restriction CRR model}
	p_{\P,0} := \frac{\mathfrak{r} - {\sf d}_{\P}}{{\sf u}_{\P} - {\sf d}_{\P}}~(>0) \quad\mbox{ and }\quad p_{\P,1} := \frac{{\sf u}_{\P}^{1-\alpha}(\mathfrak{r} - {\sf d}_{\P})}{{\sf u}_{\P}^{1-\alpha}(\mathfrak{r} - {\sf d}_{\P}) + {\sf d}_{\P}^{1-\alpha}({\sf u}_{\P} - \mathfrak{r})}~(<1).
\end{equation*}
Note that only fractions from the interval $[0,1]$ are admissible, and that the expression in the middle line in (\mbox{\ref{Ex fin - solution CRR model}}) lies in $(0,1)$ when $p_{\P}\in (p_{\P,0},p_{\P,1})$. Thus, part (ii) of Theorem \mbox{\ref{Ex fin - Thm opt trading strat}} shows that the strategy $\pi^{\P}_{\sf CRR}$ defined by (\mbox{\ref{Ex fin - optimal trading strategy}}) (with $\gamma_n^{\P}$ replaced by $\gamma_{{\sf CRR}}^{\P}$) is optimal w.r.t.\ $\P$ and unique among all $\pi\in\Pi_{\rm lin}(\P)$.
{\hspace*{\fill}$\Diamond$\par\bigskip}
\end{examplenorm}

In the following example the bond and the asset evolve according to the ordinary differential equation and the It\^o stochastic differential equation
\begin{equation*}\label{bond DE and stock SDE}
    d\mathfrak{B}_t = \nu\mathfrak{B}_t\,dt \quad\mbox{ and }\quad d\mathfrak{S}_t = \mu \mathfrak{S}_t\,dt+\sigma \mathfrak{S}_t\,d\mathfrak{W}_t,
\end{equation*}
respectively, where $\nu,\mu\in\R_{\ge 0}$ and $\sigma\in\R_{>0}$ are constants and $\mathfrak{W}$ is a one-dimensional standard Brownian motion. We assume that the trading period is (without loss of generality) the unit interval $[0,1]$ and that the bond and the asset can be traded only at $N$ equidistant time points in $[0,1]$, namely at $t_{N,n}:=n/N$, $n=0,\ldots,N-1$. Then, in particular, the relative price changes $\mathfrak{r}_{n+1}:=B_{n+1}/B_n=\mathfrak{B}_{t_{N,n+1}}/\mathfrak{B}_{t_{N,n}}$ and $\mathfrak{R}_{n+1}:=S_{n+1}/S_n=\mathfrak{S}_{t_{N,n+1}}/\mathfrak{S}_{t_{N,n}}$ are given by
$$
	\exp\big\{\nu(t_{N,n+1}-t_{N,n})\big\}
$$
and
$$
	\exp\big\{(\mu - \tfrac{\sigma^2}{2})(t_{N,n+1}-t_{N,n}) + \sigma (\mathfrak{W}_{t_{N,n+1}} - \mathfrak{W}_{t_{N,n}})\big\},
$$
respectively. In particular, $\mathfrak{r}_{n+1}=\exp(\nu/N)$ and $\mathfrak{R}_{n+1}$ is distributed according to the log-normal distribution $\rm{LN}_{(\mu - \sigma^2/2)/N,\sigma^2/N}$ for any $n=0,\ldots,N-1$.

\begin{examplenorm}[Black--Scholes--Merton model]\label{Ex fin - example BSM model}
Let $\mathfrak{r}_1=\cdots=\mathfrak{r}_N=\mathfrak{r}$ for $\mathfrak{r}:=\exp(\nu/N)$, where $\nu\in\R_{\ge 0}$. Moreover let $\P\in{\cal P}$ be any transition function defined as in (\mbox{\ref{Ex fin - set trans func psi}}) with $\mathfrak{m}_1=\cdots=\mathfrak{m}_N=\mathfrak{m}_{\P}$ for $\mathfrak{m}_{\P}:=\rm{LN}_{(\mu_{\P} - \sigma_{\P}^2/2)/N,\sigma_{\P}^2/N}$, where $\mu_{\P}\in\R_{\ge 0}$ and $\sigma_{\P}\in\R_{>0}$ are some given constants (depending on $\P$) satisfying $\mu_{\P}>(1-\alpha)\sigma_{\P}^2$. Then $\P\in{\cal P}_\psi$ and it is easily seen that conditions (a)--(c) of Assumption {\bf (FM)} hold. In particular, the corresponding financial market is arbitrage-free and the optimization problem (\mbox{\ref{Ex fin - one-stage max problem}}) now reads as
\begin{equation}\label{Ex fin - optimization problem BSM model}
	\int_{\R_{\ge 0}}u_\alpha\Big(1+\gamma\Big(\frac{y}{\mathfrak{r}} - 1\Big)\Big)\mathfrak{f}_{(\mu_{\P} - \sigma_{\P}^2/2)/N,\sigma_{\P}^2/N}(y)\,\ell(dy)\longrightarrow \max\ (\mbox{in $\gamma\in[0,1]$)\,!}
\end{equation}
where $\mathfrak{f}_{(\mu_{\P} - \sigma_{\P}^2/2)/N,\sigma_{\P}^2/N}$ is the standard Lebesgue density of the log-normal distribution ${\rm LN}_{(\mu_{\P} - \sigma_{\P}^2/2)/N,\sigma_{\P}^2/N}$. Lemma \mbox{\ref{Ex fin - lemma existence solution reduced opt prob}} ensures that (\mbox{\ref{Ex fin - optimization problem BSM model}}) has a unique solution, $\gamma_{{\sf BSM}}^{\P}$, and it is known (see, e.g., \cite{Merton1969,Pham2009}) that this solution is given by
\begin{equation}\label{Ex fin - solution BSM model}
    \gamma_{{\sf BSM}}^{\P} =
    \left\{
    \begin{array}{lll}
        0 & , & \nu\in[\mu_{\P},\infty) \\
        \frac{1}{1-\alpha}\frac{\mu_{\P} - \nu}{\sigma_{\P}^2} & , & \nu\in(\nu_{\P,\alpha},\mu_{\P}) \\
        1 & , & \nu\in[0,\nu_{\P,\alpha}]
    \end{array}
    \right.,
\end{equation}
where $\nu_{\P,\alpha}:=\mu_{\P} - (1 - \alpha)\sigma_{\P}^2\,(\in(0,\mu_{\P}))$. Note that only fractions from the interval $[0,1]$ are admissible, and that the expression in the middle line in (\mbox{\ref{Ex fin - solution BSM model}}) is called {\em Merton ratio} and lies in $(0,1)$ when $\nu\in(\nu_{\P,\alpha},\mu_{\P})$. Thus, part (ii) of Theorem \mbox{\ref{Ex fin - Thm opt trading strat}} shows that the strategy $\pi^{\P}_{\sf BSM}$ defined by (\mbox{\ref{Ex fin - optimal trading strategy}}) (with $\gamma_n^{\P}$ replaced by $\gamma_{{\sf BSM}}^{\P}$) is optimal w.r.t.\ $\P$ and unique among all $\pi\in\Pi_{\rm lin}(\P)$.
{\hspace*{\fill}$\Diamond$\par\bigskip}
\end{examplenorm}


\subsection{`Hadamard derivative' of the optimal value functional}\label{Subsec - Finance Example - Hadamard derivative  optimal value}


Maintain the notation and terminology introduced in Subsections \mbox{\ref{Subsec - Finance Example - Basic financial market}}--\mbox{\ref{Subsec - Finance Example - Comp opt trading strat}}. In this subsection we will specify the `Hadamard derivative' of the optimal value functional of the terminal wealth problem (\mbox{\ref{Ex fin - term wealth prob}}) at (fixed) $\P$; see part (ii) of Theorem \mbox{\ref{Ex fin - Thm Hadamard}}. Recall that $\alpha\in(0,1)$ introduced in (\mbox{\ref{Ex fin - def power utility}}) is fixed and determines the degree of risk aversion of the agent.

By the choice of the gauge function $\psi$ (see (\mbox{\ref{Ex fin - gauge function}})) we may choose $\M:=\M':=\M_{\rm{\scriptsize{H\ddot{o}l}},\alpha}$ (with $\M_{\rm{\scriptsize{H\ddot{o}l}},\alpha}$ introduced in Example \mbox{\ref{examples prob metric - hoelder}}) in the setting of Subsection \mbox{\ref{Subsec - Hadamard differentiability of the optimal value functional}}. Note that $\psi$ coincides with the corresponding gauge function in Example \mbox{\ref{examples prob metric - hoelder}} with $x':=0$. That is, in the end the metric $d_{\infty,\M_{\rm{\scriptsize{H\ddot{o}l}},\alpha}}^\psi$ (as defined in (\mbox{\ref{def metric transition functions}})) on ${\cal P}_\psi$ is used to measure the distance between transition functions.

For the formulation of Theorem \mbox{\ref{Ex fin - Thm Hadamard}} recall from (\mbox{\ref{def cal v}}) the definition of the functionals ${\cal V}_0^{x_0;\pi}$ and ${\cal V}_0^{x_0}$, where the maps $V_0^{\P;\pi}$ and $V_0^{\P}$ are given by (\mbox{\ref{exp tot reward}}) and (\mbox{\ref{value function}}), respectively. In the specific setting of Subsection \mbox{\ref{Subsec - Finance Example - Embedding FM into MDM}} we know from (\mbox{\ref{Ex fin - rep value functions}}) that
\begin{equation}\label{Ex fin - value functional Pi lin}
	{\cal V}_0^{x_0;\pi}(\P) \,=\,  V_0^{\P;\pi}(x_0)  \,=\,  \ex^{x_0,\P;\pi}[r_N(X_N)] \quad\mbox{ and }\quad {\cal V}_0^{x_0}(\P) \,=\, \sup_{\pi\in\Pi}{\cal V}_0^{x_0;\pi}(\P)
\end{equation}
for any $x_0\in\R_{\ge 0}$, $\P\in{\cal P}_\psi$, and $\pi\in\Pi$.

Further recall that any $\boldsymbol{\gamma}=(\gamma_n)_{n=0}^{N-1}\in[0,1]^N$ induces a linear trading strategy $\pi_{\boldsymbol{\gamma}}:=(f_n^{\boldsymbol{\gamma}})_{n=0}^{N-1}\in\Pi_{\rm lin}$ through (\mbox{\ref{Ex fin - linear decision rules}}). Let $v_n^{\P;\gamma_n}$ be defined as on the left-hand side of (\mbox{\ref{Ex fin - one-stage max problem}}) and set $v_n^{\P;\boldsymbol{\gamma}}:=v_n^{\P;\gamma_n}$ for any $n=0,\ldots,N-1$.
Moreover, for any $n=0,\ldots,N-1$ denote by $\gamma_n^{\P}$ the unique solution to the optimization problem (\mbox{\ref{Ex fin - one-stage max problem}}) (Lemma \mbox{\ref{Ex fin - lemma existence solution reduced opt prob}} ensures the existence of a unique solution). Finally set $\boldsymbol{\gamma}^{\P}:=(\gamma_n^{\P})_{n=0}^{N-1}$.

\begin{theorem}[`Differentiability' of ${\cal V}_0^{x_0;\pi_{\boldsymbol{\gamma}}}$ and ${\cal V}_0^{x_0}$]\label{Ex fin - Thm Hadamard}
In the setting above let $\P\in{\cal P}_\psi$, $\boldsymbol{\gamma}\in[0,1]^N$, and $x_0\in\R_{\ge 0}$. Then the following two assertions hold.
\begin{itemize}
    \item[{\rm (i)}] The map ${\cal V}_0^{x_0;\pi_{\boldsymbol{\gamma}}}:{\cal P}_\psi\rightarrow\R$ defined by (\mbox{\ref{Ex fin - value functional Pi lin}}) is `Fr\'echet differentiable' at $\P$ w.r.t.\ $(\M_{\rm{\scriptsize{H\ddot{o}l}},\alpha},\psi)$ with `Fr\'echet derivative' $\dot{\cal V}_{0;\P}^{x_0;\pi_{\boldsymbol{\gamma}}}:{\cal P}_{\psi}^{\P; \pm}\rightarrow\R$ given by
        \begin{equation}\label{Ex fin - Frechet exp tot reward}
            \dot{\cal V}_{0;\P}^{x_0;\pi_{\boldsymbol{\gamma}}}(\Q-\P) \,=\, \dot{\mathfrak{v}}_0^{\P,\Q;\pi_{\boldsymbol{\gamma}}}\,u_\alpha(x_0),
        \end{equation}
        where $\dot{\mathfrak{v}}_0^{\P,\Q;\pi_{\boldsymbol{\gamma}}}:= \sum_{k=0}^{N-1}v_{N-1}^{\P;\boldsymbol{\gamma}}\cdots(v_k^{\Q;\boldsymbol{\gamma}} - v_k^{\P;\boldsymbol{\gamma}})\cdots v_0^{\P;\boldsymbol{\gamma}}$.

	\item[{\rm (ii)}] The map ${\cal V}_0^{x_0}:{\cal P}_\psi\rightarrow\R$ defined by (\mbox{\ref{Ex fin - value functional Pi lin}}) is `Hadamard differentiable' at $\P$ w.r.t.\ $(\M_{\rm{\scriptsize{H\ddot{o}l}},\alpha},\psi)$ with `Hadamard derivative' $\dot{\cal V}_{0;\P}^{x_0}:{\cal P}_{\psi}^{\P;\pm}\rightarrow\R$ given by
\begin{equation}\label{Ex fin - Hadamard optimal value}
	\dot{\cal V}_{0;\P}^{x_0}(\Q-\P) \,=\, \sup_{\pi\in\Pi_{\rm lin}(\P)}\dot{\cal V}_{0;\P}^{x_0;\pi}(\Q-\P) \,=\, \dot{\cal V}_{0;\P}^{x_0;\pi_{\boldsymbol{\gamma}^{\P}}}(\Q-\P).
\end{equation}
\end{itemize}
\end{theorem}

\begin{remarknorm}\label{Ex fin - remark opt linear strategies}
Basically Theorem \mbox{\ref{Thm - hadamard cal v}} yields the first ``$=$'' in (\mbox{\ref{Ex fin - Hadamard optimal value}}) with $\Pi_{\rm lin}(\P)$ replaced by $\Pi(\P)$. Since part (ii) of Theorem \mbox{\ref{Ex fin - Thm opt trading strat}} ensures that for any $\P\in{\cal P}_\psi$ there exists an optimal trading strategy which belongs to $\Pi_{\rm lin}$, we may replace for {\em any} $\P\in{\cal P}_\psi$ in the representation (\mbox{\ref{Ex fin - rep value functions}}) of the value function $V_0^{\P}(x_0)$ (or, equivalently, in the representation (\mbox{\ref{Ex fin - value functional Pi lin}}) of the value functional ${\cal V}_0^{x_0}(\P)$) the set $\Pi$ by $\Pi_{\rm lin}$ ($\subseteq\Pi$). Therefore one can use Theorem \mbox{\ref{Thm - hadamard cal v}} to derive the first ``$=$'' in (\mbox{\ref{Ex fin - Hadamard optimal value}}). The second ``$=$'' in (\mbox{\ref{Ex fin - Hadamard optimal value}}) is ensured by the second assertion in part (ii) of Theorem \mbox{\ref{Ex fin - Thm opt trading strat}}. For details see the proof which is carried out in Subsection \mbox{\ref{Subsec - Finance Example - proof Thm hd}} of the supplementary material.
\end{remarknorm}


\subsection{Numerical examples for the `Hadamard derivative'}\label{Subsec - Finance Example - Numerical example}


In this subsection we quantify by means of the `Hadamard derivative' (of the optimal value functional ${\cal V}_0^{x_0}$) the effect of incorporating an unlikely but significant jump in the dynamics $S=(S_0,\ldots,S_N)$ of an asset price on the optimal value of the corresponding terminal wealth problem (\mbox{\ref{Ex fin - term wealth prob}}). At the end of this subsection we will also study the effect of incorporating more than one jump.

We specifically focus on the setting of the discretized Black--Scholes--Merton model from Example \mbox{\ref{Ex fin - example BSM model}} with (mainly) $N=12$. That is, we let $\mathfrak{r}_1=\cdots=\mathfrak{r}_N=\mathfrak{r}$ for $\mathfrak{r}:=\exp(\nu/N)$, where $\nu\in\R_{\ge 0}$. Moreover let $\P$ correspond to $\mathfrak{m}_1=\cdots=\mathfrak{m}_N=\mathfrak{m}_{\P}$ for $\mathfrak{m}_{\P}:=\rm{LN}_{(\mu_{\P} - \sigma_{\P}^2/2)/N,\sigma_{\P}^2/N}$, where $\mu_{\P}\in\R_{\ge 0}$ and $\sigma_{\P}\in\R_{>0}$ are chosen such that $\mu_{\P} > (1-\alpha)\sigma_{\P}^2$.
In fact we let specifically $\mu_{\P}=0.05$ and $\sigma_{\P}=0.2$. This set of parameters is often used in numerical examples in the field of mathematical finance; see, e.g., \cite[p.\,898]{Lemoretal2006}. For the initial state we choose $x_0=1$. For the drift $\nu$ of the bond we will consider different values, all of them lying in $\{0.01,0.02,0.03,0.035,0.04\}$. Moreover, we let (mainly) 
$\alpha\in\{0.25,0.5,0.75\}$. Recall that $\alpha$ determines the degree of risk aversion of the agent; a small $\alpha$ corresponds to high risk aversion.

By a price jump at a fixed time $n\in\{0,\ldots,N-1\}$ we mean that the asset's return $\mathfrak{R}_{n+1}$ is not anymore drawn from $\mathfrak{m}_{\P}$ but is given by a deterministic value $\Delta\in\R_{\ge 0}$ esstentially `away' from $1$. As appears from Table \mbox{\ref{table:Ex Fin - Quantiles of m P}}, in the case $N=12$ it seems to be reasonable to speak of a `jump' at least if $\Delta\le 0.8$ or $\Delta\ge 1.25$. The probability under $\mathfrak{m}_{\P}$ for a realized return smaller than $0.8$ (resp.\ larger than $1.25$) is smaller than $0.0001$. A realized return of $\le 0.5$ (resp.\ $\ge 1.5$) is practically impossible; its probability under $\mathfrak{m}_{\P}$ is smaller than $10^{-30}$ (resp.\ $10^{-10}$). That is, the choice $\Delta=0.5$ or $\Delta=1.5$ doubtlessly corresponds to a significant price jump.

\begin{table}[H]
\centering
\setlength{\tabcolsep}{.9mm}
\caption{Some quantiles of the distribution $\mathfrak{m}_{\P}$ of the asset's return in the discretized ($N=12$) Black--Scholes--Merton model ($\mu_{\P}=0.05$, $\sigma_{\P}=0.2$).}
\label{table:Ex Fin - Quantiles of m P}
\begin{tabular}{c c c c c c c c c}
\hline\noalign{\smallskip}
    $t$ & $10^{-30}$ & $10^{-10}$ & $0.0001$ & $0.0005$ & $0.005$ & $0.01$ & $0.025$ & $0.05$ \\
\noalign{\smallskip}\hline\noalign{\smallskip}
    $F_{\mathfrak{m}_{\P}}^{-1}(t)$ & $0.5172$ & $0.6944$ & $0.8088$ & $0.8290$ & $0.8639$ & $0.8765$ & $0.8952$ & $0.9116$\\
    $F_{\mathfrak{m}_{\P}}^{-1}(1-t)$ & $1.9433$ & $1.4474$ & $1.2426$ & $1.2126$ & $1.1632$ & $1.1466$ & $1.1226$ & $1.1024$ \\
\noalign{\smallskip}\hline
\end{tabular}
\end{table}

If at a fixed time $\tau\in\{0,\ldots,N-1\}$ a formerly nearly impossible `jump' $\Delta$ can now occur with probability $\varepsilon$, then instead of $\mathfrak{m}_{\tau+1}=\mathfrak{m}_{\P}$ one has $\mathfrak{m}_{\tau+1}=(1-\varepsilon)\mathfrak{m}_{\P}+\varepsilon\delta_\Delta$. That is, instead of $\P$ the transition function is now given by $(1-\varepsilon)\P+\varepsilon\Q_{\Delta,\tau}$ with $\Q_{\Delta,\tau}$ generated through (\mbox{\ref{Ex fin - set trans func psi}}) by
$\mathfrak{m}_{n+1}=\mathfrak{m}_{\Q_{\Delta,\tau;n}}$, $n=0,\ldots,N-1$, where
\begin{equation}\label{Ex fin - Hadamard BSM - eq10}
	\mathfrak{m}_{\Q_{\Delta,\tau;n}}
	:= \left\{
    \begin{array}{lll}
     \delta_{\Delta} & , & n=\tau \\
     \mathfrak{m}_{\P} & , &  \mbox{otherwise}
    \end{array}
    \right..
\end{equation}
By part (ii) of Theorem \mbox{\ref{Ex fin - Thm Hadamard}} the `Hadamard derivative' $\dot{\cal V}_{0;\P}^{x_0}$ of the optimal value functional ${\cal V}_0^{x_0}$ evaluated at $\Q_{\Delta,\tau}-\P$ can be written as
\begin{eqnarray}\label{Ex fin - Hadamard BSM - eq20}
    \dot{\cal V}_{0;\P}^{x_0}(\Q_{\Delta,\tau}-\P)
    & = & \sum_{k=0}^{N-1}v_{N-1}^{\P;\boldsymbol{\gamma}^{\P}_{\sf BSM}}\cdots(v_k^{\Q_{\Delta,\tau};\boldsymbol{\gamma}^{\P}_{\sf BSM}} - v_k^{\P;\boldsymbol{\gamma}^{\P}_{\sf BSM}})\cdots v_0^{\P;\boldsymbol{\gamma}^{\P}_{\sf BSM}} \nonumber \\
    & = & v_{N-1}^{\P;\boldsymbol{\gamma}^{\P}_{\sf BSM}}\cdots(v_\tau^{\Q_{\Delta,\tau};\boldsymbol{\gamma}^{\P}_{\sf BSM}} - v_\tau^{\P;\boldsymbol{\gamma}^{\P}_{\sf BSM}})\cdots v_0^{\P;\boldsymbol{\gamma}^{\P}_{\sf BSM}}
\end{eqnarray}
with $\boldsymbol{\gamma}^{\P}_{\sf BSM}:=(\gamma^{\P}_{\sf BSM},\ldots,\gamma^{\P}_{\sf BSM})$, where $\gamma^{\P}_{\sf BSM}$ is given by (\mbox{\ref{Ex fin - solution BSM model}}). The involved factors are
\begin{align}\label{Ex fin - Hadamard BSM - eq40}
	& v_n^{\P;\boldsymbol{\gamma}^{\P}_{\sf BSM}}\\
    & = \left\{
    \begin{array}{lll}
       \hspace{-.5mm} 1 & , & \nu\in[\mu_{\P},\infty) \\
       \hspace{-.5mm} \int_{\R_{\ge 0}}u_\alpha\big(1+\frac{1}{1-\alpha}\frac{\mu_{\P} - \nu}{\sigma_{\P}^2}(\frac{y}{\mathfrak{r}} - 1)\big)\mathfrak{f}_{(\mu_{\P} - \sigma_{\P}^2/2)/N,\sigma_{\P}^2/N}(y)\,\ell(dy) & , & \nu\in(\nu_{\P,\alpha},\mu_{\P}) \\
       \hspace{-.5mm} \mathfrak{r}^{-\alpha}\,\exp\big\{\frac{\alpha}{N}(\mu_{\P} - \frac{\sigma_{\P}^2}{2}) + \frac{(\alpha\sigma_{\P})^2}{2N}\big\} & , & \nu\in[0,\nu_{\P,\alpha}],
    \end{array}
    \right.\nonumber
\end{align}
\begin{equation}\label{Ex fin - Hadamard BSM - eq50}
	v_n^{\Q_{\Delta,\tau};\boldsymbol{\gamma}^{\P}_{\sf BSM}}
	 = \left\{
    \begin{array}{lll}
      \hspace{-.5mm} 1 & , & \nu\in[\mu_{\P},\infty) \\
      \hspace{-.5mm} \int_{\R_{\ge 0}}u_\alpha\big(1 + \frac{1}{1-\alpha}\frac{\mu_{\P} - \nu}{\sigma_{\P}^2}(\frac{y}{\mathfrak{r}}-1)\big)\,\mathfrak{m}_{\Q_{\Delta,\tau;n}}(dy)
       & , &  \nu\in(\nu_{\P,\alpha},\mu_{\P}) \\
      \hspace{-.5mm} \mathfrak{r}^{-\alpha}\,\int_{\R_{\ge 0}}u_\alpha(y)\,\mathfrak{m}_{\Q_{\Delta,\tau;n}}(dy)
      & , & \nu\in[0,\nu_{\P,\alpha}]
    \end{array}
    \right.
\end{equation}
for $n=0,\ldots,N-1$, where $\nu_{\P,\alpha}:=\mu_{\P} - (1 - \alpha)\sigma_{\P}^2$ ($\in(0,\mu_{\P})$).

Note that $\dot{\cal V}_{0;\P}^{x_0}(\Q_{\Delta,\tau}-\P)$ is independent of $\tau$, which can be seen from (\mbox{\ref{Ex fin - Hadamard BSM - eq10}})--(\mbox{\ref{Ex fin - Hadamard BSM - eq50}}). That is, the effect of a jump is independent of the time at which the jump takes place. Also note that $\dot{\cal V}_{0;\P}^{x_0}(\Q_{\Delta,\tau}-\P)\equiv 0$ when $\nu\in[\mu_{\P},\infty)$. This is not surprising, because in this case the optimal fraction $\gamma^{\P}_{\sf BSM}$ to be invested into the asset is equal to $0$ (see (\mbox{\ref{Ex fin - solution BSM model}})) and the agent performs a complete investment in the bond at each trading time $n$.

\begin{remarknorm}\label{Ex fin - remark relatively compactness}
As mentioned before, the `Hadamard derivative' $\dot{\cal V}_{0;\P}^{x_0}$ evaluated at $\Q_{\Delta,\tau}-\P$ can be seen as the first-order sensitivity of the optimal value ${\cal V}_0^{x_0}(\P)$ w.r.t.\ a change of $\P$ to $(1-\varepsilon)\P+\varepsilon\Q_{\Delta,\tau}$, with $\varepsilon>0$ small. It is a natural wish to compare these values for different $\Delta\in\R_{>0}$. In Subsection \mbox{\ref{Subsec - Finance Example - proof remark relatively compactness}} of the supplementary material it is proven that the family $\{\Q_{\Delta,\tau}: \Delta\in[0,\delta]\}$ is relatively compact w.r.t.\ $d_{\infty,\M_{\rm{\scriptsize{H\ddot{o}l}},\alpha}}^\psi$ (the proof does {\em not} work if $d_{\infty,\M_{\rm{\scriptsize{H\ddot{o}l}},\alpha}}^\psi$ is replaced by $d_{\infty,\M_{\rm{\scriptsize{H\ddot{o}l}},\alpha}}^\phi$ for any gauge function $\phi$ `flatter' than $\psi$) for any fixed $\delta\in\R_{>0}$. As a consequence the approximation (\mbox{\ref{motivation of measure for f-o sensitivity}}) with $\Q=\Q_{\Delta,\tau}$ holds uniformly in $\Delta\in[0,\delta]$, and therefore the values $\dot{\cal V}_{0;\P}^{x_0}(\Q_{\Delta,\tau}-\P)$, $\Delta\in[0,\delta]$, can be compared with each other with clear conscience.
{\hspace*{\fill}$\Diamond$\par\bigskip}
\end{remarknorm}

By Remark \mbox{\ref{Ex fin - remark relatively compactness}} and (\mbox{\ref{Ex fin - Hadamard BSM - eq20}}) we are able to compare the effect of incorporating different `jumps' $\Delta$ in the dynamics $S=(S_0,\ldots,S_N)$ of an asset price on the optimal value (functional) ${\cal V}_0^{x_0}(\P)$.

\begin{figure}[H]
\centering
\subfigure{\includegraphics[width=0.35\textwidth]{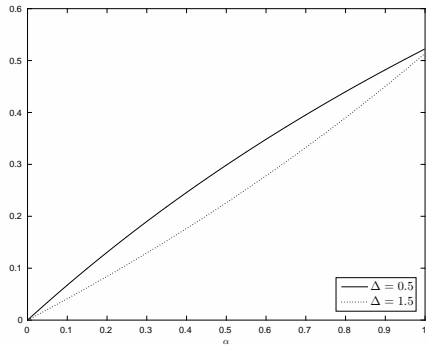}}
\caption{`Hadamard derivative' $\dot{\cal V}_{0;\P}^{x_0}(\Q_{\Delta,\tau}-\P)$ (for $\Delta=1.5)$ and negative `Hadamard derivative' $-\dot{\cal V}_{0;\P}^{x_0}(\Q_{\Delta,\tau}-\P)$ (for $\Delta=0.5)$ for $N=12$, $\nu=0.01$, $\mu_{\P}=0.05$, and $\sigma_{\P}=0.2$ in dependence of the risk aversion parameter $\alpha$.}
\label{fig:Ex Fin - Hadamard comp diff jumps}
\end{figure}

As appears from Figure \mbox{\ref{fig:Ex Fin - Hadamard comp diff jumps}} the negative effect of incorporating a `jump' $\Delta=0.5$ in the dynamics $S=(S_0,\ldots,S_N)$ of an asset price is larger than the positive effect of incorporating a `jump' $\Delta=1.5$ for every choice of the agent's degree of risk aversion. Figure \mbox{\ref{fig:Ex Fin - Hadamard comp diff jumps}} also shows the unsurprising effect that a high risk aversion (small value of $\alpha$) leads to a negligible sensitivity.

Next we compare the values of $\dot{\cal V}_{0;\P}^{x_0}(\Q_{\Delta,\tau}-\P)$ for trading horizons $N\in\{4,12,52\}$ in dependence of the drift $\nu$ of the bond and the `jump' $\Delta$. This choices of $N$ correspond respectively to a quarterly, monthly, and weekly time discretization. We will restrict ourselves to `jumps' $\Delta\le 0.8$. On the one hand, this ensures that the `jumps' are significant; see the discussion above. On the other hand, as just discerned from Figure \mbox{\ref{fig:Ex Fin - Hadamard comp diff jumps}}, 
the effect of jumps `down' are more significant than jumps `up'.

\begin{figure}[H]
\centering
\subfigure{\includegraphics[width=0.35\textwidth]{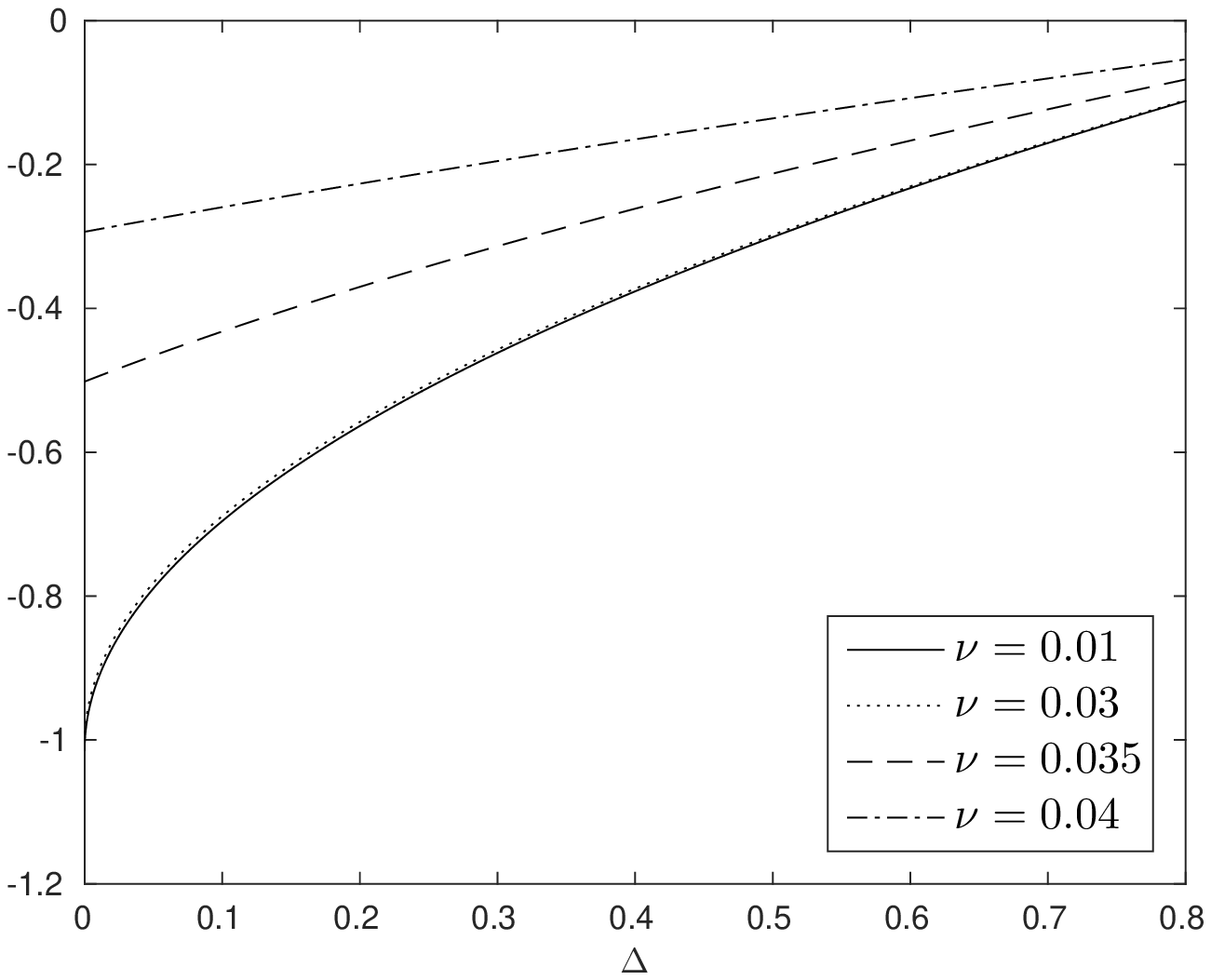}} \hspace{-5mm}
\subfigure{\includegraphics[width=0.35\textwidth]{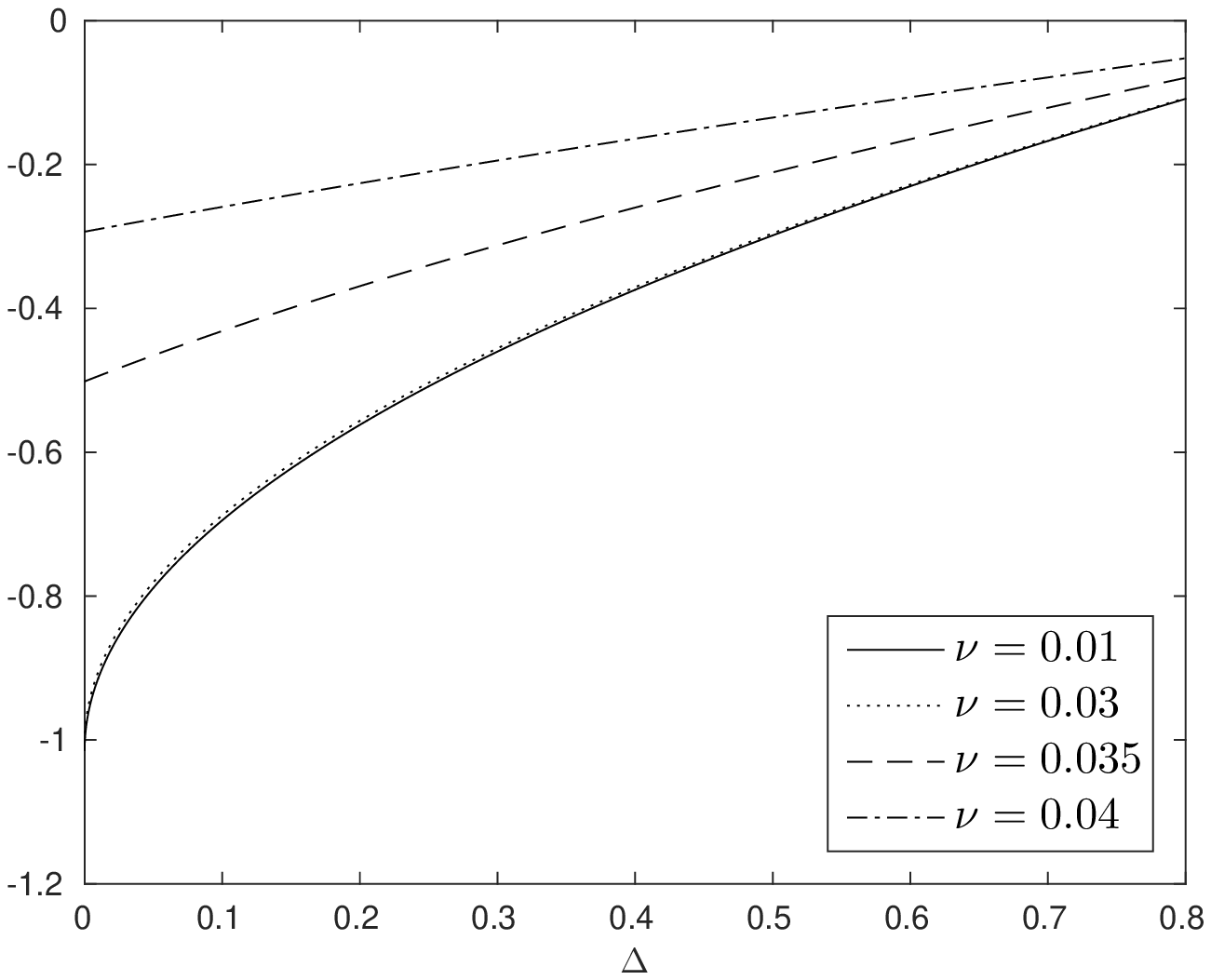}} \hspace{-5mm}
\subfigure{\includegraphics[width=0.35\textwidth]{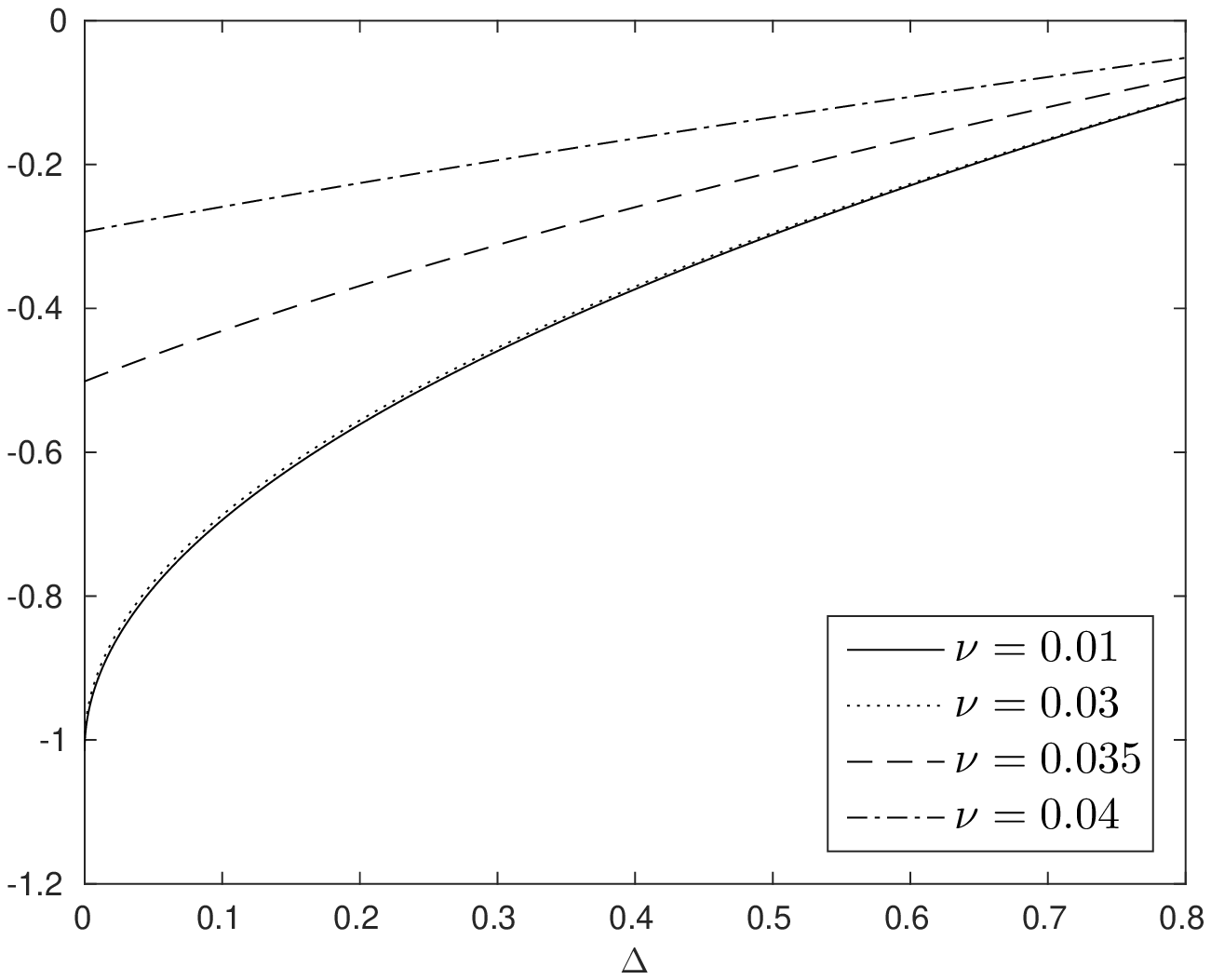}}
\caption{`Hadamard derivative' $\dot{\cal V}_{0;\P}^{x_0}(\Q_{\Delta,\tau}-\P)$ for $\alpha=0.5$, $\mu_{\P}=0.05$, and $\sigma_{\P}=0.2$ in dependence of the `jump' $\Delta$ and the drift $\nu$ of the bond, showing $N=4$ in the first, $N=12$ in the second, and $N=52$ in the third column.}
\label{fig:Ex fin - Hadamard alpha fixed}
\end{figure}

From Figure \mbox{\ref{fig:Ex fin - Hadamard alpha fixed}} one can see that for each trading time $N$ and any $\Delta\in[0,0.8]$ the (negative) effect of incorporating a `jump' $\Delta$ in the dynamics $S=(S_0,\ldots,S_N)$ of an asset price is the smaller the smaller the spread between the drift $\mu_{\P}$ of the asset and the drift $\nu$ of the bond. There is only a tiny (nearly invisible) difference between the `Hadamard derivative' $\dot{\cal V}_{0;\P}^{x_0}(\Q_{\Delta,\tau}-\P)$ for the trading times $N\in\{4,12,52\}$. So the fineness of the discretization seems to play a minor part. Next we compare the values of $\dot{\cal V}_{0;\P}^{x_0}(\Q_{\Delta,\tau}-\P)$ for the drift $\nu\in\{0.02,0.03,0.04\}$ of the bond in dependence of the risk aversion parameter $\alpha$ and the `jump' $\Delta$.

\begin{figure}[H]
\centering
\subfigure{\includegraphics[width=0.35\textwidth]{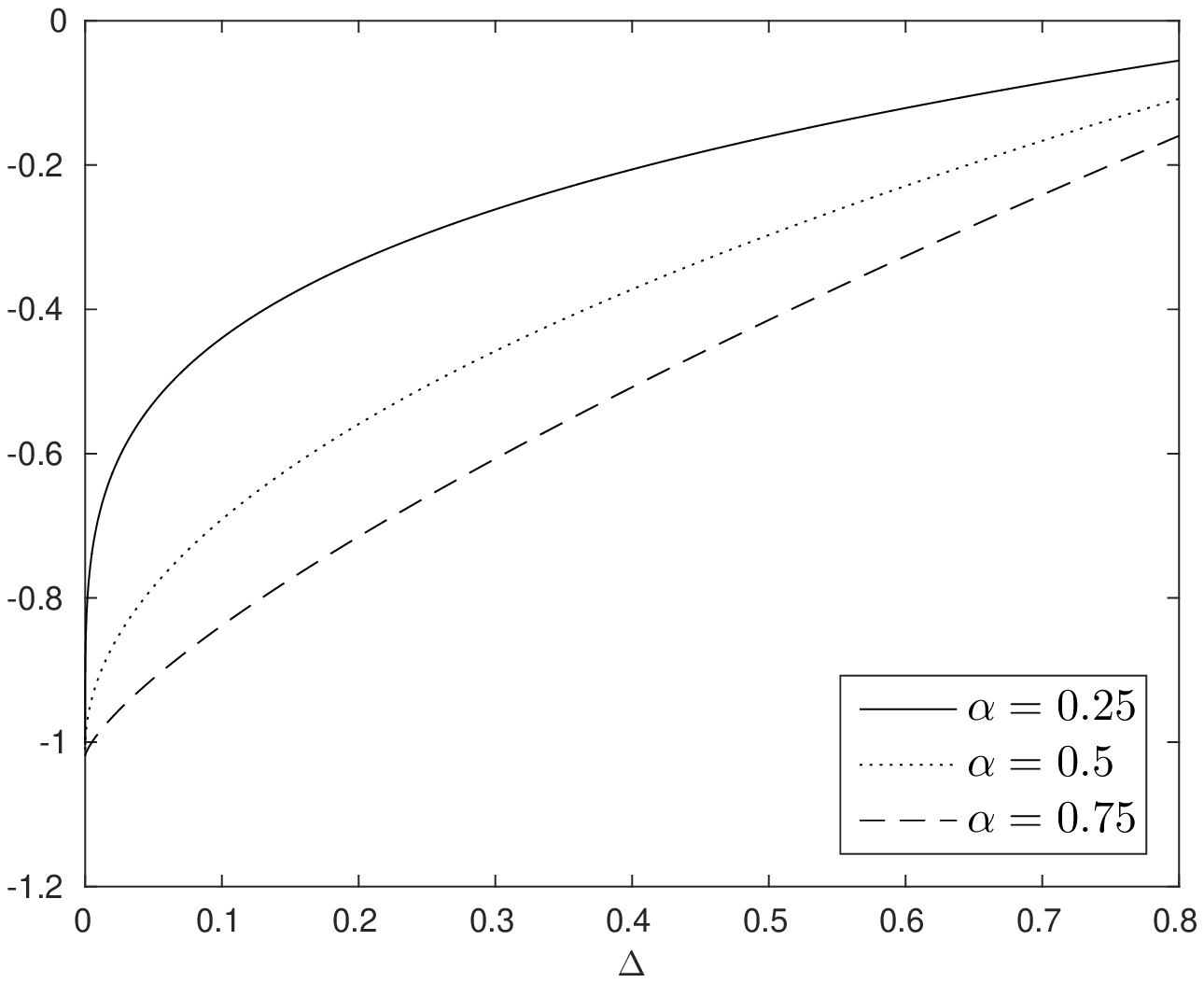}} \hspace{-5mm}
\subfigure{\includegraphics[width=0.35\textwidth]{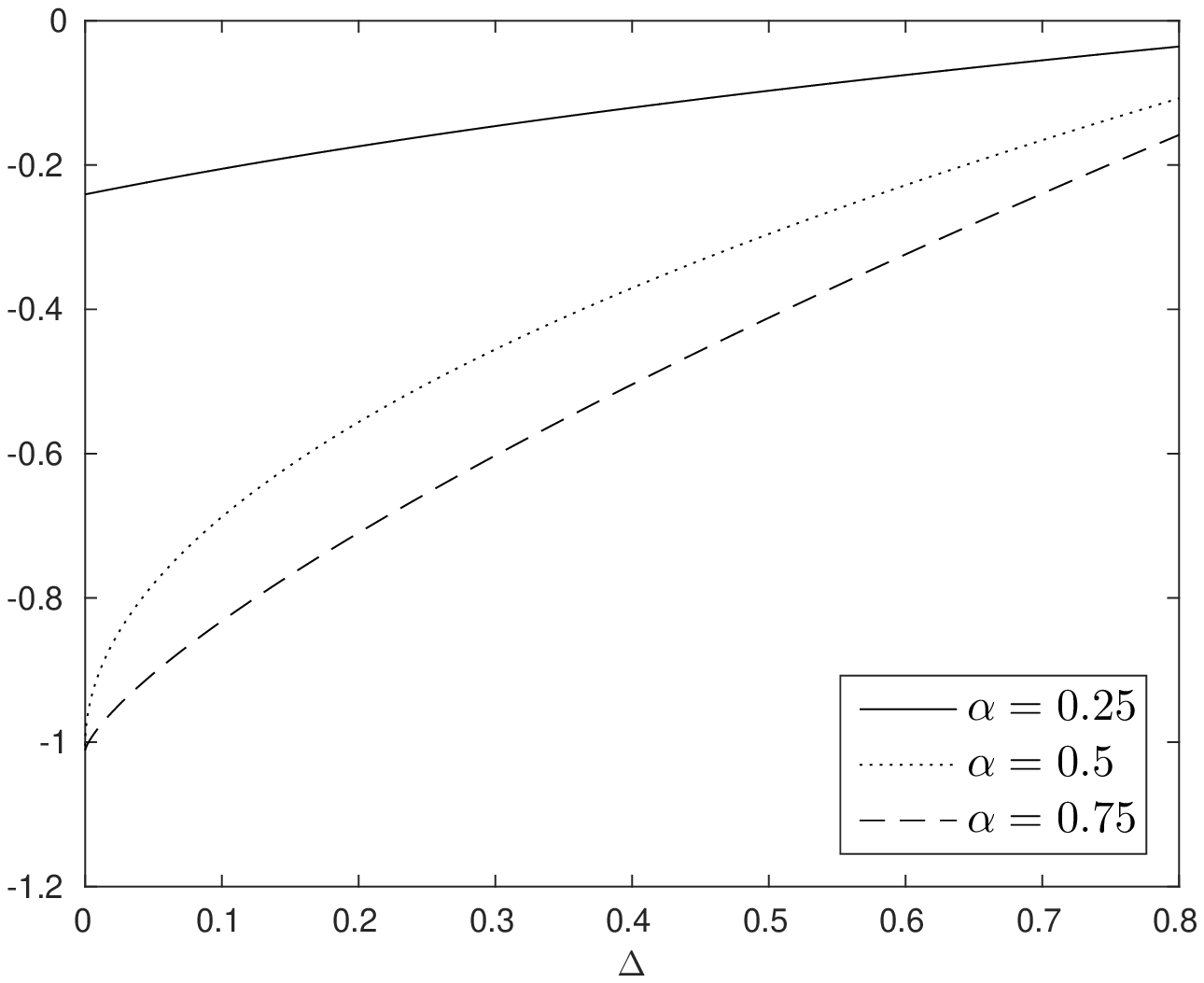}} \hspace{-5mm}
\subfigure{\includegraphics[width=0.35\textwidth]{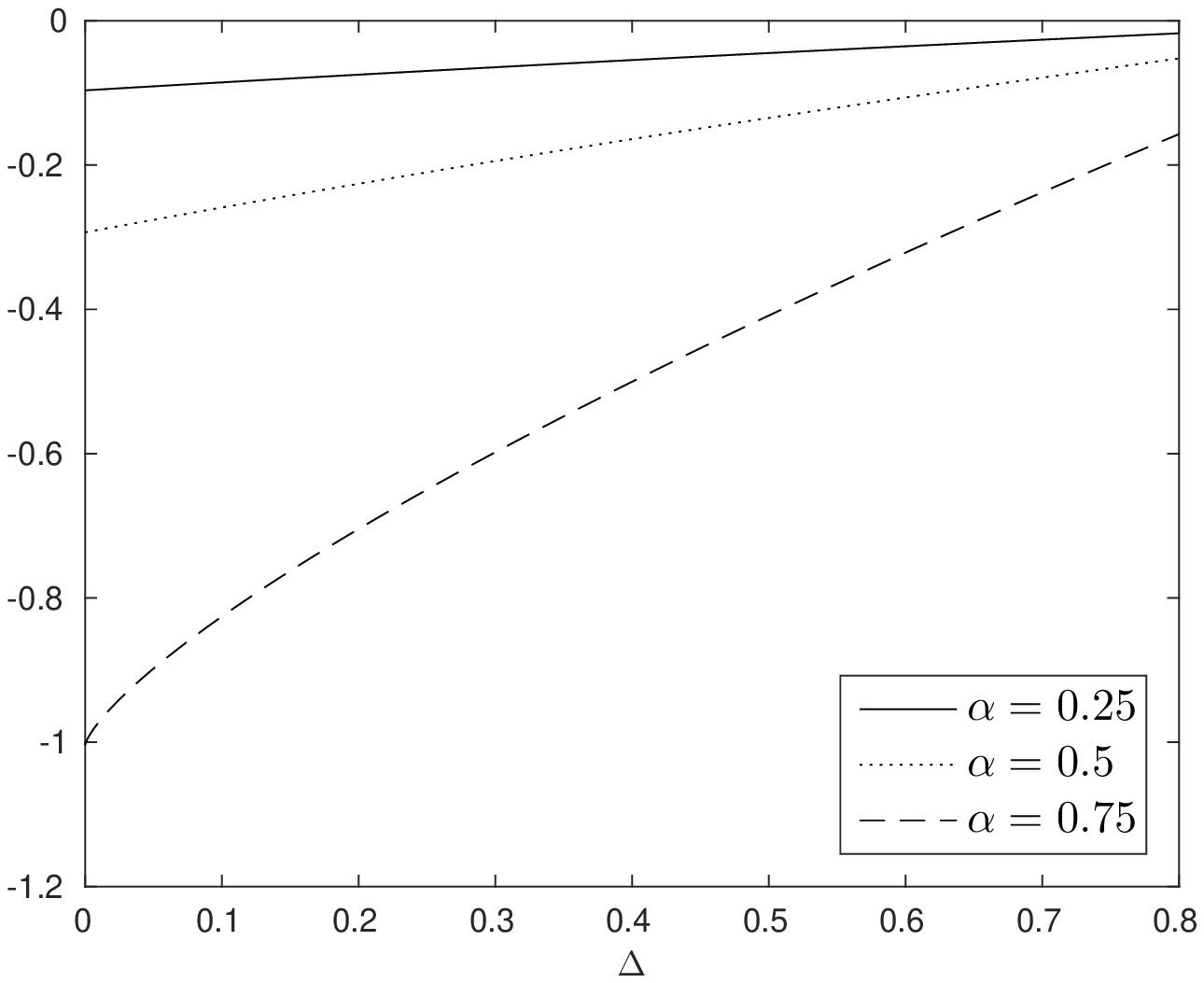}}
\caption{`Hadamard derivative' $\dot{\cal V}_{0;\P}^{x_0}(\Q_{\Delta,\tau}-\P)$ for $N=12$, $\mu_{\P}=0.05$, and $\sigma_{\P}=0.2$ in dependence of the `jump' $\Delta$ and risk aversion parameter $\alpha$, showing $\nu=0.02$ in the first, $\nu=0.03$ in the second, and $\nu=0.04$ in the third column.}
\label{fig:Ex fin - Hadamard nu fixed}
\end{figure}

As appears from Figure \mbox{\ref{fig:Ex fin - Hadamard nu fixed}}, for any $\Delta\in[0,0.8]$ the (negative) effect of incorporating a `jump' $\Delta$ in the dynamics $S=(S_0,\ldots,S_N)$ of an asset price is the smaller the higher the agent's risk aversion, no matter what the drift $\nu\in\{0.02,0.03,0.04\}$ of the bond looks like. Take into account that the extent of this effect is influenced via (\mbox{\ref{Ex fin - Hadamard BSM - eq20}})--(\mbox{\ref{Ex fin - Hadamard BSM - eq50}}) by the optimal fraction $\gamma^{\P}_{\sf BSM}$ to be invested into the asset which in turn depends on the risk aversion parameter $\alpha$ (see (\mbox{\ref{Ex fin - solution BSM model}})).

Finally, let us briefly touch on the case where more than one jump may appear. More precisely, instead of $\Q_{\Delta,\tau}$ (with $\tau\in\{0,\ldots,N-1\}$) consider the transition function $\Q_{\Delta,\boldsymbol{\tau}(\ell)}$ (with $1\le\ell\le N$, $\boldsymbol{\tau}(\ell)=(\tau_1,\ldots,\tau_\ell)$, $\tau_1,\ldots,\tau_\ell\in\{0,\ldots,N-1\}$ pairwise distinct) which is still generated by means of (\mbox{\ref{Ex fin - Hadamard BSM - eq10}}) but with the difference that at the $\ell$ different times $\tau_1,\ldots,\tau_\ell$ the distribution $\mathfrak{m}_{\P}$ is replaced by $\delta_\Delta$. Just as in the case $\ell=1$, it turns out that it does not matter at which times $\tau_1,\ldots,\tau_\ell$ exactly these $\ell$ jumps occur. Figure \mbox{\ref{fig:Ex fin - Hadamard nu fixed multi shock}} shows the value of $\dot{\cal V}_{0;\P}^{x_0}(\Q_{\Delta,\boldsymbol{\tau}(\ell)}-\P)$ in dependence on $\ell$ and $\Delta$. It seems that for any fixed $\Delta\in[0,0.8]$ the first-order sensitivity increases approximately linearly in $\ell$.

\begin{figure}[H]
\centering
\subfigure{\includegraphics[width=0.35\textwidth]{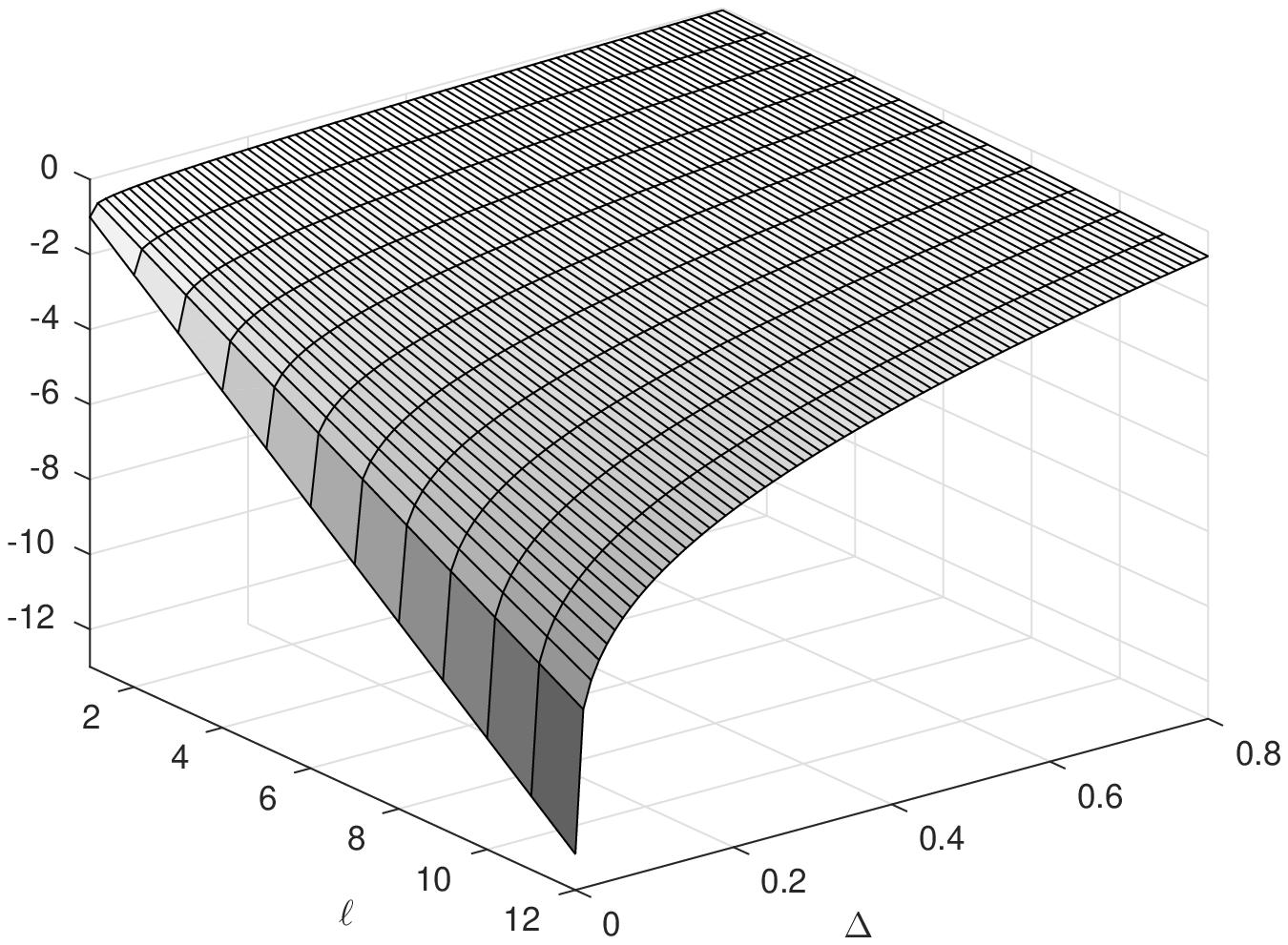}} \hspace{-5mm}
\subfigure{\includegraphics[width=0.35\textwidth]{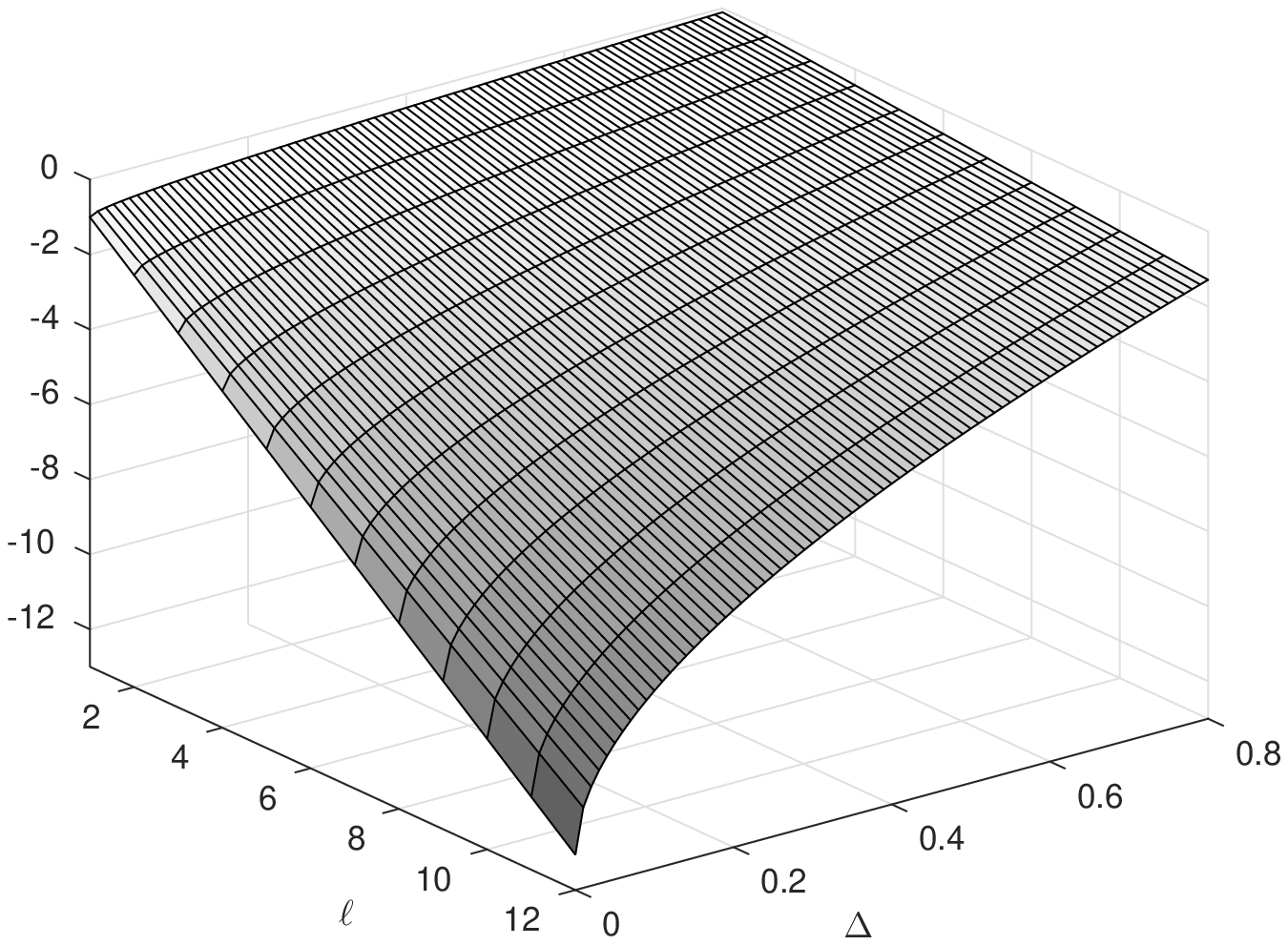}} \hspace{-5mm}
\subfigure{\includegraphics[width=0.35\textwidth]{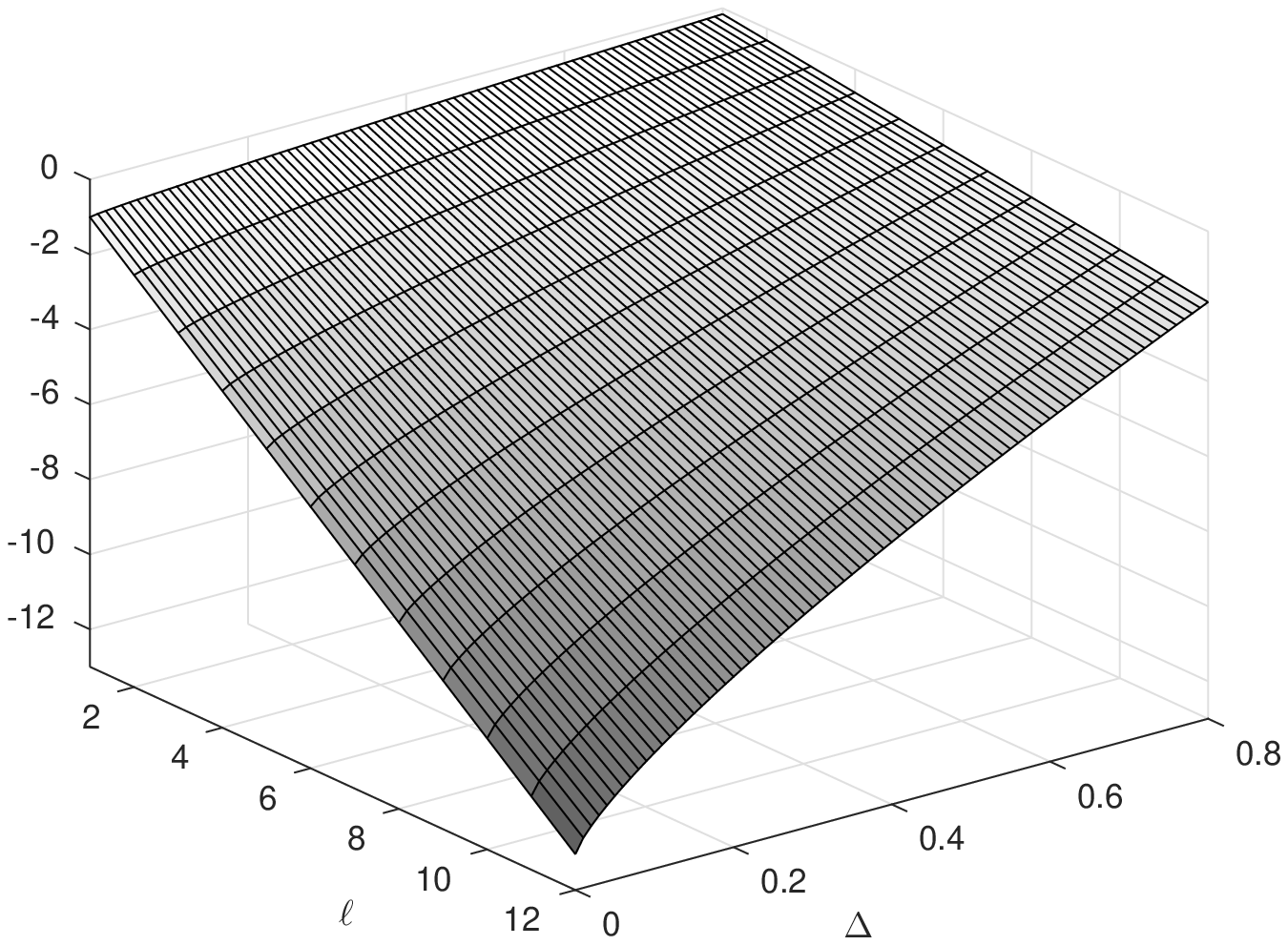}}
\caption{`Hadamard derivative' $\dot{\cal V}_{0;\P}^{x_0}(\Q_{\Delta,\boldsymbol{\tau}(\ell)}-\P)$ for $N=12$ in dependence of $\ell\in\{1,\ldots,N\}$ and $\Delta\in[0,0.8]$ showing $\alpha=0.25$ and $\nu=0.02$ (left), $\alpha=0.5$ and $\nu=0.03$ (middle), and $\alpha=0.75$ and $\nu=0.04$ (right).}
\label{fig:Ex fin - Hadamard nu fixed multi shock}
\end{figure}


\section*{Supplementary material}

The supplementary material illustrates the setting of Sections \mbox{\ref{Sec - Formal definition of MDM}}--\mbox{\ref{Sec - Hadamard differentiability chapter}} in the case of finite state space and action spaces, and contains the proofs of the results from Sections \mbox{\ref{Sec - Hadamard differentiability chapter}}--\mbox{\ref{Sec - Finance Example}}. Moreover, supplemental definitions and results to Section \mbox{\ref{Sec - Formal definition of MDM}} are given and the existence of optimal strategies in general MDMs is discussed. Finally, an supplemental topological result is shown.



\section{Supplement: The discrete case as an illustrating example} \label{Sec - Discrete MDM}


In Sections \mbox{\ref{Sec - Formal definition of MDM}}--\mbox{\ref{Sec - Hadamard differentiability chapter}}  we work in a rather general set-up. This implies that we cannot avoid dealing with `advanced' objects. In the special case where the state space and the action spaces are finite the situation is different. In this case it is possible to present the basic definitions and the main result (Theorem \mbox{\ref{Thm - hadamard cal v}}) in a more comprehensible way.
For the moment we assume that the reader is already familiar with the basic terminology of MDMs. Otherwise we advise the reader to first read Section \mbox{\ref{Sec - Formal definition of MDM}}. In Section \mbox{\ref{Subsec - Hadamard differentiability optimal value functional - discrete case}} it will be discussed how the following elaborations fit to the general set-up of Sections \mbox{\ref{Sec - Formal definition of MDM}}--\mbox{\ref{Sec - Hadamard differentiability chapter}}.


\subsection{Basic model components}\label{Subsec - Model components}


Let $E=\{x_1,\ldots,x_\mathfrak{s}\}$ be a finite state space, $N\in\N$ be a fixed finite time horizon, and $A_{n}(x_i)=\{a_{n,i;1},\ldots,a_{n,i;\mathfrak{t}_{n,i}}\}$ be the finite set of possible actions that can be performed when the MDP is in state $x_i$ at time $n\in\{0,\ldots,N-1\}$. For any $n=0,\ldots,N-1$, $i=1,\ldots,\mathfrak{s}$, and $a\in A_{n}(x_i)$, the (one-step transition) probability measure on $E$ from which the state of the MDP at time $n+1$ is drawn, given that the MDP is in state $x_i$ and action $a$ is selected at time $n$, can be identified with an element $p_{n,i;a}=(p_{n,i;a}(1),\ldots,p_{n,i;a}(\mathfrak{s}))$ of $\R_{\ge 0,1}^{\mathfrak{s}}$. Here $\R_{\ge 0,1}^{\mathfrak{s}}$ is the set of all vectors from $\R^\mathfrak{s}$ whose entries are nonnegative and sum up to $1$, and $p_{n,i;a}(j)$ specifies the probability that the MDP will be in state $x_j$ at time $n+1$, given it is in state $x_i$ and action $a\in A_{n}(x_i)$ is selected at time $n$. In particular, if the initial state $x_0\in E$ is fixed and $i_0$ refers to the corresponding index (i.e.\ $x_0=x_{i_0}$), the vector
\begin{equation}\label{clueing tp}
    \boldsymbol{p}:=\big(\oplus_{k=1}^{\mathfrak{t}_{0,i_0}}p_{0,i_0;a_{0,i_0;k}}\big)\oplus\big(\oplus_{n=1}^{N-1}\oplus_{i=1}^\mathfrak{s}\oplus_{k=1}^{\mathfrak{t}_{n,i}}p_{n,i;a_{n,i;k}}\big)
\end{equation}
in $\R^{\mathfrak{d}}$, with $\mathfrak{d}:=(\mathfrak{t}_{0,i_0}+\sum_{n=1}^{N-1}\sum_{i=1}^\mathfrak{s}\mathfrak{t}_{n,i})\mathfrak{s}$, can be identified with the transition probability function, i.e.\ with the ensemble of all transition probabilities. Here $\oplus$ is the `clueing operator' defined by $(\alpha_1,\ldots,\alpha_s)\oplus(\beta_1,\ldots,\beta_t):=(\alpha_1,\ldots,\alpha_s,\beta_1,\ldots,\beta_t)$. In fact $\boldsymbol{p}$ is even an element of the following subset of $\R^{\mathfrak{d}}$:
\begin{equation}\label{def widetilde cal P}
    \widetilde{\cal P}:=\big(\R_{\ge 0,1}^\mathfrak{s}\big)^{\times(\mathfrak{d}/\mathfrak{s})}.
\end{equation}

If ${\cal V}_0(\boldsymbol{p})$ denotes the optimal value of the MDM based on transition probability function $\boldsymbol{p}$, then ${\cal V}_0={\cal V}_0(\cdot)$ can be seen as a map from $\widetilde{\cal P}$ ($\subseteq\R^{\mathfrak{d}}$) to $\R$.


\subsection{Definition of first-order sensitivity in the discrete case}\label{Subsec - Discrete case definition f-o sensitivity}


It is tempting to consider the classical Fr\'echet (or total) derivative $\dot{\cal V}_{0;\boldsymbol{p}}$ of ${\cal V}_0$ at $\boldsymbol{p}$ in order to obtain a tool for measuring the first-order sensitivity of the optimal value w.r.t.\ a change from $\boldsymbol{p}$ to $(1-\varepsilon)\boldsymbol{p}+\varepsilon\boldsymbol{q}$:
\begin{equation}\label{def total derivative}
    \dot{\cal V}_{0;\boldsymbol{p}}(\boldsymbol{q}-\boldsymbol{p})=\lim_{m\to\infty}\frac{{\cal V}_0(\boldsymbol{p}+h_m(\boldsymbol{q}-\boldsymbol{p}))-{\cal V}_0(\boldsymbol{p})}{h_m}\quad\mbox{ uniformly in }\boldsymbol{q}\in\overline{B}_1(\boldsymbol{p})
\end{equation}
for any $(h_m)\in\R_0^\N$ with $h_m\to 0$, where $\overline{B}_1(\boldsymbol{p})$ is the closed ball in $\R^{\mathfrak{d}}$ around $\boldsymbol{p}$ with radius $1$ and $\R_0:=\R\setminus\{0\}$. This approach is indeed expedient to some extent. However, one has to note that $\boldsymbol{p}+h_m(\boldsymbol{q}-\boldsymbol{p})$ may lie outside ${\cal V}_0$'s domain $\widetilde{\cal P}$. To avoid this problem, we replace condition (\mbox{\ref{def total derivative}}) by the following variant of (\mbox{\ref{def total derivative}}):
\begin{equation}\label{def total-like derivative}
    \dot{\cal V}_{0;\boldsymbol{p}}(\boldsymbol{q}-\boldsymbol{p})=\lim_{m\to\infty}\frac{{\cal V}_0(\boldsymbol{p}+\varepsilon_m(\boldsymbol{q}-\boldsymbol{p}))-{\cal V}_0(\boldsymbol{p})}{\varepsilon_m}\quad\mbox{ uniformly in }\boldsymbol{q}\in\widetilde{\cal P}
\end{equation}
for any $(\varepsilon_m)\in(0,1]^\N$ with $\varepsilon_m\to 0$. Take into account that $\boldsymbol{p}+\varepsilon(\boldsymbol{q}-\boldsymbol{p})$ lies in $\widetilde{\cal P}$ for any $\boldsymbol{p},\boldsymbol{q}\in\widetilde{\cal P}$ and $\varepsilon\in(0,1]$. Also note that, if $\R^{\mathfrak{d}}$ is equipped with the max-norm,
$\widetilde{\cal P}$ is contained in $\overline{B}_1(\boldsymbol{p})$ for any $\boldsymbol{p}\in\widetilde{\cal P}$.

For classical Fr\'echet (or total) differentiability the derivative $\dot{\cal V}_{0;\boldsymbol{p}}$ is required to be linear and continuous. On the one hand, for `Fr\'echet differentiability' (see Definition \mbox{\ref{def of total diff}}) we will also require a sort of continuity, namely that the mapping $\boldsymbol{q}\mapsto\dot{\cal V}_{0;\boldsymbol{p}}(\boldsymbol{q}-\boldsymbol{p})$ from $\widetilde{\cal P}$ to $\R$ is continuous, where $\widetilde{\cal P}$ is equipped with the relative topology of $\R^{\mathfrak{d}}$. On the other hand, the domain of $\dot{\cal V}_{0;\boldsymbol{p}}$ is given by $\widetilde{\cal P}^{\boldsymbol{p};\pm}:=\{\boldsymbol{q}-\boldsymbol{p}:\boldsymbol{q}\in\widetilde{\cal P}\}$ and thus not a linear space. Therefore linearity of $\dot{\cal V}_{0;\boldsymbol{p}}$ is an indefinite property.

In view of (\mbox{\ref{motivation of measure for f-o sensitivity}}), the quantity $\dot{\cal V}_{0;\boldsymbol{p}}(\boldsymbol{q}-\boldsymbol{p})$ can be seen as a measure for the first-order sensitivity of the optimal value ${\cal V}_0(\boldsymbol{p})$ under transition probability function $\boldsymbol{p}$ w.r.t.\ a change from $\boldsymbol{p}$ to $(1-\varepsilon)\boldsymbol{p}+\varepsilon\boldsymbol{q}$, with $\varepsilon>0$ small. For this interpretation it is actually not necessary to require that $\dot{\cal V}_{0;\boldsymbol{p}}(\,\cdot\,-\boldsymbol{p})$ is continuous or that the convergence in (\mbox{\ref{def total-like derivative}}) holds uniformly in $\boldsymbol{q}\in\widetilde{\cal P}$. One can indeed be content with the directional derivative, i.e.\ with the convergence in (\mbox{\ref{def total-like derivative}}) for fixed $\boldsymbol{q}$. Nevertheless continuity and uniformity are natural wishes in this context, because they ensure stability of the first-order sensitivity w.r.t.\ small modifications of $\boldsymbol{q}$ as well as comparability of the first-order sensitivity of (infinitely) many different $\boldsymbol{q}$. We refer to the discussion subsequent to (\mbox{\ref{motivation of measure for f-o sensitivity}}).

\begin{definition}\label{def of total diff}
A map ${\cal V}:\widetilde{\cal P}\rightarrow\R$ is said to be `Fr\'echet differentiable' at $\boldsymbol{p}\in\widetilde{\cal P}$ if there exists a map $\dot{\cal V}_{\boldsymbol{p}}:\widetilde{\cal P}^{\boldsymbol{p};\pm}\to\R$ for which (\mbox{\ref{def total-like derivative}}) (with ${\cal V}_0,\dot{\cal V}_{0;\boldsymbol{p}}$ replaced by ${\cal V},\dot{\cal V}_{\boldsymbol{p}}$ respectively) holds and for which the mapping $\boldsymbol{q}\mapsto\dot{\cal V}_{\boldsymbol{p}}(\boldsymbol{q}-\boldsymbol{p})$ from $\widetilde{\cal P}$ to $\R$ is continuous. In this case $\dot{\cal V}_{\boldsymbol{p}}$ is called `Fr\'echet derivative' of ${\cal V}$ at $\boldsymbol{p}$.
\end{definition}


\subsection{Computation of first-order sensitivity in the discrete case}\label{Subsec - Discrete case computation f-o sensitivity}


To specify the `Fr\'echet derivative' of ${\cal V}_0$ at $\boldsymbol{p}$ we need some further notation. For any strategy $\pi=(f_n)_{n=0}^{N-1}$, we use ${\cal V}_0^\pi(\boldsymbol{p})$ to denote the expected total reward (from time $0$ to $N$) when $\boldsymbol{p}$ is the underlying transition probability function and the decisions are performed according to $\pi$. In the finite setting there exists under $\boldsymbol{p}$ at least one optimal strategy $\pi^{\boldsymbol{p}}$, i.e.\ a strategy $\pi^{\boldsymbol{p}}$ with ${\cal V}_0^{\pi^{\boldsymbol{p}}}(\boldsymbol{p})=\max_\pi{\cal V}_0^{\pi}(\boldsymbol{p})$. We will write $\Pi(\boldsymbol{p})$ for the (finite) set of all optimal strategies w.r.t.\ $\boldsymbol{p}$. Then the results of Subsection \mbox{\ref{Subsec - Hadamard differentiability of the optimal value functional}}  show that the `Fr\'echet derivative' of ${\cal V}_0$ at $\boldsymbol{p}$ is given by
\begin{equation}\label{Frechet opt value discrete case}
    \dot{\cal V}_{0;\boldsymbol{p}}(\boldsymbol{q}-\boldsymbol{p})=\max_{\pi\in\Pi(\boldsymbol{p})}\dot{\cal V}_{0;\boldsymbol{p}}^\pi(\boldsymbol{q}-\boldsymbol{p})\quad\mbox{ for all }\boldsymbol{q}\in\widetilde{\cal P},
\end{equation}
where $\dot{\cal V}_{0;\boldsymbol{p}}^\pi$ refers to the `Fr\'echet derivative' of ${\cal V}_0^\pi$ at $\boldsymbol{p}$. The latter can be obtained from a suitable iteration scheme. According to Remark \mbox{\ref{remark hd - iteration scheme}}  we indeed have
$$
    \dot{\cal V}_{0;\boldsymbol{p}}^\pi(\boldsymbol{q}-\boldsymbol{p})=\dot V_{0}^{\boldsymbol{p},\boldsymbol{q};\pi}(x_{i_0})
$$
(recall that $x_{i_0}\in E$ is the initial state and that $\pi=(f_n)_{n=0}^{N-1}$ refers to a strategy) for
\begin{equation}\label{Frechet pi backward iteration scheme - finite}
	\begin{aligned}
        \dot V_N^{\boldsymbol{p},\boldsymbol{q};\pi}(x_i) & \,:=\, 0 \\
	    \dot V_n^{\boldsymbol{p},\boldsymbol{q};\pi}(x_i) & \,:=\, \sum_{j=1}^\mathfrak{s} \dot V_{n+1}^{\boldsymbol{p},\boldsymbol{q};\pi}(x_j)\,p_{n,i;f_n(x_i)}(j) \\
        & \qquad + \sum_{j=1}^\mathfrak{s} V^{\boldsymbol{p};\pi}_{n+1}(x_j)\,\big(q_{n,i;f_n(x_i)}(j) - p_{n,i;f_n(x_i)}(j)\big),\quad n=0,\ldots,N-1,
	\end{aligned}
\end{equation}
$i=1,\ldots,\mathfrak{s}$, where the $V^{\boldsymbol{p};\pi}_{n}(\cdot)$ are given by the usual backward iteration scheme (see, e.g., Lemma 3.5 in \cite{Hinderer1970} or p.\,80 in \cite{Puterman1994}) for the computation of ${\cal V}^\pi_0(\boldsymbol{p})$:
\begin{equation}\label{reward iteration - finite}
    \begin{aligned}
    V_N^{\boldsymbol{p};\pi}(x_i) & \, := \, r_N(x_i)\\
	V_n^{\boldsymbol{p};\pi}(x_i) & \, := \, r_n(x_i,f_n(x_i)) + \sum_{j=1}^\mathfrak{s} V_{n+1}^{\boldsymbol{p};\pi}(x_j)\,p_{n,i;f_n(x_i)}(j),\quad n=0,\ldots,N-1,
    \end{aligned}
\end{equation}
$i=1,\ldots,\mathfrak{s}$. Here $r_n(x_i,a)$ specifies the (one-stage) reward when making decision $a$ at time $n$ in state $x_i$, and $r_N(x_i)$ specifies the reward of being in state $x_i$ at terminal time $N$. Also note that $f_n(x_i)$ determines the action which is taken at time $n$ in state $x_i$ under strategy $\pi=(f_n)_{n=0}^{N-1}$.


\subsection{An example: stochastic inventory control}\label{Sec - Inventory Example}


In this subsection we will consider an inventory control problem, which is a classical example in discrete dynamic optimization; see, e.g.,  \cite{Bertsekas1995,HernandezLasserre1996,Puterman1994} and references cited therein. At first, we introduce in Subsection \mbox{\ref{Subsec - Inventory Example - Basic inventory control model}} a (simple) inventory control model and formulate the corresponding inventory control problem. Thereafter, in Subsection \mbox{\ref{Subsec - Inventory Example - Embedding inventory control into MDM}}, we will explain how the inventory control model can be embedded into the setting of Subsection \mbox{\ref{Subsec - Model components}}. Finally, in Subsection \mbox{\ref{Subsec - Inventory Example - Numerical example}} we will present some numerical examples for the `Fr\'echet derivative' established in Subsection \mbox{\ref{Subsec - Discrete case computation f-o sensitivity}}.


\subsubsection{Basic inventory control model, and the target}\label{Subsec - Inventory Example - Basic inventory control model}


Consider an $N$-period inventory control system where a supplier of a single product seeks optimal inventory management to meet random commodity demand in such a way that a measure of profit over a time horizon of $N$ periods is maximized. For the formulation of the model, let $I_1,\ldots,I_N$ be $\N_0$-valued independent random variables on some probability space $\OFP$, where $I_{n+1}$ can be seen as the random demand of the single product in the period between time $n$ and time $n+1$. We denote by $\mathfrak{p}_{n+1}=(\mathfrak{p}_{n+1;k})_{k\in\N_0}$ the counting density of $I_{n+1}$ (i.e.\ $\mathfrak{p}_{n+1;k}:=\pr[I_{n+1}=k]$) and assume that $\mathfrak{p}_{n+1}$ is known for any $n=0,\ldots,N-1$. Note that $\mathfrak{p}_{n+1}\in\R_{\ge 0,1}^{\N_0}$, where $\R_{\ge 0,1}^{\N_0}$ denotes the space of all real-valued sequences whose entries are nonnegative and sum up to $1$. Let ${\cal F}_0$ be the trivial $\sigma$-algebra, and set ${\cal F}_n:=\sigma(I_1,\ldots,I_n)$ for any $n=1,\ldots,N$.

We suppose that within each period of time the available inventory level of the single product is restricted to $K$ units (for some fixed $K\in\N$) and that there is no backlogging of unsatisfied demand at the end of each period. The latter means that if at the end of a period the demand exceeds the current inventory, then the whole inventory is sold and the surplus demand gets lost.

Given an initial inventory level $y_0\in\{0,\ldots,K\}$, the supplier intends to find optimal order quantities according to an order strategy to maximize some measure of profit. By {\em order strategy} we mean an (${\cal F}_n$)-adapted $\{0,\ldots,K\}$-valued stochastic process $\varphi=(\varphi_n)_{n=0}^{N-1}$, where $\varphi_n$ specifies the amount of ordered units of the single product at the beginning of period $n$. Here we suppose that the delivery of any order occurs instantaneously.
Since excess demand is lost by assumption,
the corresponding {\em inventory (level) process} $Y^{\varphi}=(Y^{\varphi}_0,\ldots,Y^{\varphi}_N)$ associated with $\varphi=(\varphi_n)_{n=0}^{N-1}$ is given by
\begin{equation}\label{Ex inv - inventory process}
	Y^{\varphi}_0 := y_0 \quad\mbox{ and }\quad Y^{\varphi}_{n+1} := Y^{\varphi}_n + \varphi_n - \min\{I_{n+1},Y^{\varphi}_n + \varphi_n\}, \qquad n=0,\ldots,N-1.
\end{equation}
Note that $\min\{I_{n+1},Y^{\varphi}_n + \varphi_n\}$ corresponds to the amount of units of the single product sold in the period between time $n$ and time $n+1$. Hence we refer to the process $Z^{\varphi}:=(Z^{\varphi}_0,\ldots,Z^{\varphi}_N)$ defined by
\begin{equation}\label{Ex inv - sales process}
	Z^{\varphi}_0 := 0 \quad\mbox{ and }\quad Z^{\varphi}_{n+1} := \min\{I_{n+1},Y^{\varphi}_n + \varphi_n\}, \qquad n=0,\ldots,N-1
\end{equation}
as {\em sales process} associated with $\varphi=(\varphi_n)_{n=0}^{N-1}$.

In view of (\mbox{\ref{Ex inv - inventory process}}) and since the inventory capacity is restricted to $K$ units, we may and do identify any order strategy with an (${\cal F}_n$)-adapted $\{0,\ldots,K\}$-valued stochastic process $\varphi=(\varphi_n)_{n=0}^{N-1}$ satisfying $\varphi_0\in\{0,\ldots,K-y_0\}$ and $\varphi_n\in\{0,\ldots,K-Y^{\varphi}_n\}$ for all $n=1,\ldots,N-1$. We restrict ourselves to {\em Markovian} order strategies $\varphi=(\varphi_n)_{n=0}^{N-1}$ which means that $\varphi_n$ only depends on $n$ and $(Y^{\varphi}_n,Z^{\varphi}_n)$. To put it another way, we suppose that for any $n=0,\ldots,N-1$ there is some map $f_n:\{0,\ldots,K\}^2\to\{0,\ldots,K\}$ such that $\varphi_n = f_n(Y^{\varphi}_n,Z^{\varphi}_n)$. Hence, for given strategy $\varphi=(\varphi_n)_{n=0}^{N-1}$ (resp.\ $\pi=(f_n)_{n=0}^{N-1}$) the process $X^{\varphi}:=(Y^{\varphi},Z^{\varphi})$ is a $\{0,\ldots,K\}^2$-valued $({\cal F}_n)$-Markov process whose one-step transition probability for the transition from state $x=(y,z)\in\{0,\ldots,K\}^2$ at time $n\in\{0,\ldots,N-1\}$ to state $x'=(y',z')\in\{0,\ldots,K\}^2$ at time $n+1$ is given by
$\eta^{\mathfrak{p}_{n+1}}_{(y,f_n(y,z))}(z')\eins_{\{y'=y+f_n(y,z)-z'\}}$ with 
\begin{equation}\label{Ex inv - def mapping eta}
    \eta^{\mathfrak{p}_{n+1}}_{(y,a)}(z') :=
    \left\{
    \begin{array}{lll}
        0 & , & z'>y+a \\
        \mathfrak{p}_{n+1;z'} & , & z'<y+a \\
        \sum_{\ell=z'}^{\infty}\mathfrak{p}_{n+1;\ell} & , & z'=y+a
    \end{array}
    \right..
\end{equation}

The supplier's aim is to find an order strategy $\varphi=(\varphi_n)_{n=0}^{N-1}$ (resp.\ $\pi=(f_n)_{n=0}^{N-1}$) for which the expected total profit
is maximized. Here the profit can be seen as the difference between the sales revenue and the costs for ordering and holding the single product. For the sake of simplicity, we suppose that the sales revenue as well as the ordering and holding costs are known and linear in each period. Hence, we are interested in those order strategies $\varphi=(\varphi_n)_{n=0}^{N-1}$ (resp.\ $\pi=(f_n)_{n=0}^{N-1}$) for which the expectation of
$$
	   \sum_{k=0}^{N-1}\big\{u^{\sf rev}(Z_k^{\varphi}) - u^{\sf ord}(f_k(Y_k^{\varphi},Z_k^{\varphi})) - u^{\sf hol}(Y_k^{\varphi},f_k(Y_k^{\varphi},Z_k^{\varphi}))\big\}
       + \big\{u^{\sf rev}(Z_N^{\varphi}) - u^{\sf hol}(Y_N^{\varphi},0)\big\}
$$
is maximized, where $u^{\sf rev},u^{\sf ord}:\{0,\ldots,K\}\to\N_0$ and $u^{\sf hol}:\{0,\ldots,K\}^2\to\N_0$ are for some fixed $s_{\sf rev}, c_{\sf ord}, c_{\sf fix}, c_{\sf hol}\in\N$ defined by
$$
	u^{\sf rev}(z) := s_{\sf rev}\cdot z, \quad u^{\sf ord}(a) := (c_{\sf fix} + c_{\sf ord}\cdot a)\eins_{\{a>0\}},\quad u^{\sf hol}(y,a) := c_{\sf hol}\cdot(y + a).
$$
Note here that $s_{\sf rev}, c_{\sf ord}, c_{\sf fix}$, and $c_{\sf hol}$ denote the sales revenue, the ordering costs, the fixed ordering costs, and the holding costs per unit of the single product, respectively.


\subsubsection{Embedding into MDM, and optimal order strategies} \label{Subsec - Inventory Example - Embedding inventory control into MDM}


The setting introduced in Subsection \mbox{\ref{Subsec - Inventory Example - Basic inventory control model}} can be embedded into the setting of Subsections \mbox{\ref{Subsec - Model components}} and \mbox{\ref{Subsec - Discrete case computation f-o sensitivity}} as follows. Let $E:=\{x_1,\ldots,x_\mathfrak{s}\}$ for the enumeration $x_1,\ldots,x_\mathfrak{s}$ (with $\mathfrak{s}=(K+1)^2$) of $\{0,\ldots,K\}^2$ given by $x_i=(y_i,z_i)$ with $y_i:=\lceil i/(K+1)\rceil-1$ and $z_i:=i-(K+1)\lceil i/(K+1)\rceil+K$ (here $\lceil\cdot\rceil$ is the ceiling function), $i=1,\ldots,\mathfrak{s}$.

Let $A_{n}(x_i):=\{a_{n,i;1},\ldots,a_{n,i;\mathfrak{t}_{n,i}}\}$ with $a_{n,i;k}:=k-1$ and $\mathfrak{t}_{n,i} = \mathfrak{t}_{i} := K - y_i + 1$ for any $i=1,\ldots,\mathfrak{s}$ and $n=0,\ldots,N-1$.  For any $i=1,\ldots,\mathfrak{s}$, $k=1,\ldots,\mathfrak{t}_{i}$, and $n=0,\ldots,N-1$, let the component $p_{n,i;a_{n,i;k}}=(p_{n,i;a_{n,i;k}}(1),\break \ldots,p_{n,i;a_{n,i;k}}(\mathfrak{s}))$ of the vector $\boldsymbol{p}$ from (\mbox{\ref{clueing tp}}) be given by
\begin{equation}\label{Ex inv - transition prob}
	p_{n,i;a_{n,i;k}}(j) := \eta^{\mathfrak{p}_{n+1}}_{(y_i,a_{n,i;k})}(z_j)\eins_{\{y_j=y_i+a_{n,i;k}-z_j\}},	\qquad j=1,\ldots,\mathfrak{s}
\end{equation}
for some predetermined $\mathfrak{p}_{n+1}\in\R_{\ge 0,1}^{\N_0}$ and for $\eta^{\mathfrak{p}_{n+1}}_{(y,a)}(\cdot)$ introduced in (\mbox{\ref{Ex inv - def mapping eta}}).
In fact any element $\boldsymbol{p}$ of $\widetilde{\cal P}$ is generated via (\mbox{\ref{Ex inv - def mapping eta}})--(\mbox{\ref{Ex inv - transition prob}}) by some $N$-tuple $\mathfrak{p}=(\mathfrak{p}_{1},\ldots,\mathfrak{p}_{N})$ of counting densities $\mathfrak{p}_{1},\ldots,\mathfrak{p}_{N}$ on $\N_0$; here $\mathfrak{p}_{1},\ldots,\mathfrak{p}_{N}$ should be seen as the counting densities of $I_1,\ldots,I_N$. The value in (\mbox{\ref{Ex inv - transition prob}}) should be seen as the probability of a transition from state $(y_i,z_i)$ to state $(y_j,z_j)$ in time between $n$ and $n+1$ (this transition probability is even independent of $z_i$).

For any $i=1,\ldots,\mathfrak{s}$ and $k=1,\ldots,\mathfrak{t}_{i}$, set
\begin{eqnarray}
    r_0(x_i,a_{0,i;k}) & := & - u^{\sf ord}(a_{0,i;k}) - u^{\sf hol}(y_i,a_{0,i;k}),\label{Ex inv - reward time 0} \hspace{0.75cm}\\
    r_n(x_i,a_{n,i;k}) & := & u^{\sf rev}(z_i) - u^{\sf ord}(a_{n,i;k}) - u^{\sf hol}(y_i,a_{n,i;k}),\ n=1,\ldots,N-1,\label{Ex inv - reward time n} \hspace{0.75cm}\\
	r_N(x_i) & := & u^{\sf rev}(z_i) - u^{\sf hol}(y_i,0).\label{Ex inv - terminal reward} \hspace{0.75cm}
\end{eqnarray}
By an (admissible) order strategy we understand an $N$-tuple $\pi=(f_n)_{n=0}^{N-1}$ of maps $f_n:\{x_1,\ldots,x_\mathfrak{s}\}\to\{0,\ldots,K\}$ satisfying
$$
    f_n(x_i)= f_n(y_i,z_i)\in\{0,\ldots,K-y_i\}\quad\mbox{ for all }i=1,\ldots,\mathfrak{s}.
$$

Then for every fixed $\boldsymbol{p}\in\widetilde{\cal P}$ the inventory control problem introduced in Subsection \mbox{\ref{Subsec - Inventory Example - Basic inventory control model}} reads as
\begin{equation}\label{Ex inv - inventory control prob}
    {\cal V}_0^{\pi}(\boldsymbol{p}) \longrightarrow \max\ (\mbox{in $\pi$)\,!}
\end{equation}
where ${\cal V}_0^{\pi}(\boldsymbol{p}):=V_0^{\boldsymbol{p};\pi}(x_{i_0})$ is given by (\mbox{\ref{reward iteration - finite}}) with (\mbox{\ref{Ex inv - reward time 0}})--(\mbox{\ref{Ex inv - terminal reward}}) ($x_{i_0}\in E$ is the initial state). A strategy $\pi^{\boldsymbol{p}}$ is called an {\em optimal order strategy w.r.t.\ $\boldsymbol{p}$} if it solves the maximization problem (\mbox{\ref{Ex inv - inventory control prob}}).


\subsubsection{Numerical examples for the `Fr\'echet derivative'}\label{Subsec - Inventory Example - Numerical example}


Let us take up the numerical example at p.\,41 in \cite{Puterman1994} where $N:=3$, $K:=4$, $s_{\sf rev}:=8$, $c_{\sf ord}:=2$, $c_{\sf fix}:=4$, and $c_{\sf hol}:=1$. We fix $\mathfrak{p}:=(\mathfrak{p}_\bullet,\mathfrak{p}_\bullet,\mathfrak{p}_\bullet)$ with $\mathfrak{p}_\bullet:=(0,\tfrac{1}{4},\tfrac{1}{2},\tfrac{1}{4},0,0\ldots)$, and denote by $\boldsymbol{p}$ the unique element of $\widetilde{\cal P}$ generated by $\mathfrak{p}$ through (\mbox{\ref{Ex inv - def mapping eta}})--(\mbox{\ref{Ex inv - transition prob}}). This choice of $\mathfrak{p}$ means that in each period the demand is $1$, $2$, or $3$ with probability $\tfrac{1}{4}$, $\tfrac{1}{2}$, and $\tfrac{1}{4}$, respectively. Table \mbox{\ref{table:Ex inv - optimal strategy}} provides the (unique) optimal order strategy
$\pi^{\boldsymbol{p}}=(f_0^{\boldsymbol{p}},f_1^{\boldsymbol{p}}
,f_2^{\boldsymbol{p}})$, and the second column of Table \mbox{\ref{table:Ex inv - optimal value}} displays the maximal expected total reward ${\cal V}^{\pi^{\boldsymbol{p}}}_0(\boldsymbol{p})$ of the inventory control problem (\mbox{\ref{Ex inv - inventory control prob}}) for all possible initial inventory levels $y_0:=y_{i_0}\in\{0,\ldots,4\}$.  Moreover, the last two columns in Table \mbox{\ref{table:Ex inv - optimal value}} display the `Fr\'echet derivative' $\dot{\cal V}^{\pi^{\boldsymbol{p}}}_{0;\boldsymbol{p}}(\cdot)$ of ${\cal V}^{\pi^{\boldsymbol{p}}}_0$ at $\boldsymbol{p}$ evaluated at direction $\boldsymbol{q}_{(0)}-\boldsymbol{p}$ and at direction $\boldsymbol{q}_{(4)}-\boldsymbol{p}$ (calculated with the iteration scheme (\mbox{\ref{Frechet pi backward iteration scheme - finite}})), again for all possible initial inventory levels $y_0$. Here  $\boldsymbol{q}_{(0)}$ and $\boldsymbol{q}_{(4)}$ are generated through (\mbox{\ref{Ex inv - def mapping eta}})--(\mbox{\ref{Ex inv - transition prob}}) by $\mathfrak{q}_{(0)}:=(\mathfrak{q}_{(0)\bullet},\mathfrak{q}_{(0)\bullet},\mathfrak{q}_{(0)\bullet})$ and $\mathfrak{q}_{(4)}:=(\mathfrak{q}_{(4)\bullet},\mathfrak{q}_{(4)\bullet},\mathfrak{q}_{(4)\bullet})$ respectively, where $\mathfrak{q}_{(0)\bullet}:=(1,0,0,\ldots)$ and $\mathfrak{q}_{(4)\bullet}:=(0,0,0,0,1,0,0,\ldots)$. As the optimal strategy $\pi^{\boldsymbol{p}}$ is unique in our example, we even have $\dot{\cal V}_{0;\boldsymbol{p}}(\cdot)=\dot{\cal V}^{\pi^{\boldsymbol{p}}}_{0;\boldsymbol{p}}(\cdot)$.

Note that for $i\in\{0,4\}$ the value $\dot{\cal V}_{0;\boldsymbol{p}}(\boldsymbol{q}_{(i)}-\boldsymbol{p})$ (in our case it equals $\dot{\cal V}^{\pi^{\boldsymbol{p}}}_{0;\boldsymbol{p}}(\boldsymbol{q}_{(i)}-\boldsymbol{p})$) quantifies the first-order sensitivity of ${\cal V}_0(\boldsymbol{p})$ (respectively of ${\cal V}^{\pi^{\boldsymbol{p}}}_0(\boldsymbol{p})$) w.r.t.\ a change of the underlying probability transition function from $\boldsymbol{p}$ to $\boldsymbol{p}_{(i)}:=(1-\varepsilon)\boldsymbol{p}+\varepsilon\boldsymbol{q}_{(i)}$ with $\varepsilon\in(0,1)$ small. It can be easily seen that $\boldsymbol{p}_{(i)}$ is generated through (\mbox{\ref{Ex inv - def mapping eta}})--(\mbox{\ref{Ex inv - transition prob}}) where $\mathfrak{p}:=(\mathfrak{p}_\bullet,\mathfrak{p}_\bullet,\mathfrak{p}_\bullet)$ is replaced by $\mathfrak{p}_{(i)}:=(\mathfrak{p}_{(i)\bullet},\mathfrak{p}_{(i)\bullet},\mathfrak{p}_{(i)\bullet})$ with $\mathfrak{p}_{(i)\bullet}:=(1-\varepsilon)\mathfrak{p}_\bullet+\varepsilon\mathfrak{q}_{(i)\bullet}$ (take into account that the case differentiation in (\mbox{\ref{Ex inv - def mapping eta}}) does {\em not} depend on the counting density $\mathfrak{p}_{n+1}$). That is, the change from $\boldsymbol{p}$ to $\boldsymbol{p}_{(i)}$ means that the formerly impossible demand $i$ now gets assigned small but strictly positive probability $\varepsilon$ in each period.

\begin{table}[H]
\centering
\setlength{\tabcolsep}{.9mm}
\caption{Optimal order strategy $\pi^{\boldsymbol{p}}=(f_0^{\boldsymbol{p}},f_1^{\boldsymbol{p}},f_2^{\boldsymbol{p}})$ for $\boldsymbol{p}$ as above.}
\label{table:Ex inv - optimal strategy}
\begin{tabular}{c c c c c c c c c c c c c c}
\hline\noalign{\smallskip}
$(y,z)$ & $(0,0)$ & $(0,1)$ & $(0,2)$ & $(0,3)$ & $(0,4)$ & $(1,0)$ & $(1,1)$ & $(1,2)$ & $(1,3)$ & $(1,4)$ & $(2,0)$ & \hspace{2mm}$\cdots$\hspace{2mm} & $(4,4)$ \\
\noalign{\smallskip}\hline\noalign{\smallskip}
$f_0^{\boldsymbol{p}}$ & 4 & 4 & 4 & 4 & 4 & 3 & 3 & 3 & 3 & 3 & 0 & $\cdots$ & 0 \\
$f_1^{\boldsymbol{p}}$ & 4 & 4 & 4 & 4 & 4 & 3 & 3 & 3 & 3 & 3 & 0 & $\cdots$ & 0 \\
$f_2^{\boldsymbol{p}}$ & 2 & 2 & 2 & 2 & 2 & 0 & 0 & 0 & 0 & 0 & 0 & $\cdots$ & 0 \\
\noalign{\smallskip}\hline
\end{tabular}
\end{table}

\vspace{-1cm}

\begin{table}[H]
\centering
\caption{Optimal value ${\cal V}^{\pi^{\boldsymbol{p}}}_0(\boldsymbol{p})$ and the `Fr\'echet derivative' $\dot{\cal V}^{\pi^{\boldsymbol{p}}}_{0;\boldsymbol{p}}(\boldsymbol{q}_{(i)}-\boldsymbol{p})$ (in our example it equals $\dot{\cal V}_{0;\boldsymbol{p}}(\boldsymbol{q}_{(i)}-\boldsymbol{p})$) with $\boldsymbol{q}_{(i)}$ as above, $i\in\{0,4\}$, in dependence of the initial inventory level $y_0$.}
\label{table:Ex inv - optimal value}
\begin{tabular}{c c c c}
\hline\noalign{\smallskip}
$y_0$ & ${\cal V}^{\pi^{\boldsymbol{p}}}_0(\boldsymbol{p})$ & $\dot{\cal V}^{\pi^{\boldsymbol{p}}}_{0;\boldsymbol{p}}(\boldsymbol{q}_{(0)}-\boldsymbol{p})$ & $\dot{\cal V}^{\pi^{\boldsymbol{p}}}_{0;\boldsymbol{p}}(\boldsymbol{q}_{(4)}-\boldsymbol{p})$ \\
\noalign{\smallskip}\hline\noalign{\smallskip}
0 &  16.5313 & $-$34.0938 & 16.0313 \\
1 &  18.5313 & $-$34.0938 & 16.0313 \\
  2 &  23.1250 & $-$39.8125 & 14.0000 \\
  3 &  26.1094 & $-$37.3906 & 15.6094 \\
  4 &  28.5313 & $-$34.0938 & 16.0313 \\
\noalign{\smallskip}\hline
\end{tabular}
\end{table}

As appears from Table \mbox{\ref{table:Ex inv - optimal value}}, the negative effect of incorporating demand $0$ into the counting density $\mathfrak{p}_\bullet$ with small probability $\varepsilon$ is roughly twice as large as the positive effect of incorporating demand $4$ into $\mathfrak{p}_\bullet$ with the same small probability $\varepsilon$, no matter what the initial inventory level looks like. So, when worrying about robustness of the optimal value w.r.t.\ changes in the demand's counting density $\mathfrak{p}_\bullet$, it seems to be somewhat more important to analyse in detail the adequacy of the assumption that an absent demand is impossible than the adequacy of the assumption that a demand of $4$ is impossible.


\subsection{Embedding the discrete case into the set-up of Sections \mbox{\ref{Sec - Formal definition of MDM}}--\mbox{\ref{Sec - Hadamard differentiability chapter}} }\label{Subsec - Hadamard differentiability optimal value functional - discrete case}


In this subsection we will explain how the elaborations in Subsections \mbox{\ref{Subsec - Model components}}--\mbox{\ref{Subsec - Discrete case computation f-o sensitivity}} match our general theory introduced in Sections \mbox{\ref{Sec - Formal definition of MDM}}--\mbox{\ref{Sec - Hadamard differentiability chapter}}. Assume that the state space $E$ as well as the set of all admissible actions $A_n(x)$ for each point of time $n=0,\ldots,N-1$ and state $x\in E$ are finite. Let $\mathfrak{s}:=\# E\in\N$ and ${\cal E}:=\mathfrak{P}(E)$, and note that the sets $A_n$ as well as $D_n$ are finite for any $n=0,\ldots,N-1$.

Let us measure the distance between two probability measures $\mu$ and $\nu$ from ${\cal M}_1(E)$ by the total variation metric $d_{\rm{\scriptsize{TV}}}$, i.e.\ by
$$
    d_{\rm{\scriptsize{TV}}}(\mu,\nu) \,=\, \max_{B\in\mathfrak{P}(E)}\big|\mu[B] - \nu[B]\big| \,=\, \frac{1}{2}\sum_{y\in E}\big|\mu[\{y\}]-\nu[\{y\}]\big|.
$$
This fits the setting of Subsection \mbox{\ref{Subsec - Metric on set of probability measures}}  with $\M:=\M_{\rm{\scriptsize{TV}}}$ and $\psi:\equiv 1$; see Example \mbox{\ref{examples prob metric - tv}}. Since $E$ was assumed to be finite with $\mathfrak{s}:=\# E\in\N$, we may and do identify any probability measure $\mu\in{\cal M}_1(E)$ with some element $p_\mu=(p_\mu(1),\ldots,p_\mu(\mathfrak{s}))$ of $\R_{\ge 0,1}^\mathfrak{s}$ (with $\R_{\ge 0,1}^\mathfrak{s}$ as in Subsection \mbox{\ref{Subsec - Model components}}).
Hence the total variation distance $d_{\rm{\scriptsize{TV}}}$ between $\mu,\nu\in{\cal M}_1(E)$ can be identified (up to the factor $1/2$) with the $\ell_1$-distance between $p_\mu$ and $p_\nu$:
\begin{equation*}\label{characteriazion TV metric finite state space}
    d_{\rm{\scriptsize{TV}}}(\mu,\nu) \,=\, \frac{1}{2}\sum_{i=1}^\mathfrak{s}\big|p_\mu(i)-p_\nu(i)\big| \,=\, \frac{1}{2}\,\|p_\mu-p_\nu\|_{\ell_1}.
\end{equation*}
That is, the map $\Lambda:{\cal M}_1(E)\rightarrow\R_{\ge 0,1/2}^\mathfrak{s}$, $\mu\mapsto p_\mu/2$, provides a surjective isometry (here $\R_{\ge 0,1/2}^\mathfrak{s}$ is the set of all vectors from $\R^\mathfrak{s}$ whose entries are nonnegative and sum up to $1/2$), and therefore the metric spaces $({\cal M}_1(E),d_{\rm{\scriptsize{TV}}})$ and $(\R_{\ge 0,1/2}^\mathfrak{s},\|\cdot\|_{\ell_1})$ are isometrically isomorphic. This implies in particular that the set ${\cal M}_1(E)$ is compact w.r.t.\ $d_{\rm{\scriptsize{TV}}}$, because $\R_{\ge 0,1/2}^\mathfrak{s}$ is clearly compact w.r.t.\ $\|\cdot\|_{\ell_1}$.

For the distance between two transition functions we will employ the metric
$d_{\infty,\M_{\rm{\scriptsize{TV}}}}^1$, which is defined as in (\mbox{\ref{def metric transition functions}})  with $\psi:\equiv 1$. As the sets $D_0,\ldots,D_{N-1}$ are finite, we can identify the set ${\cal P}$ as a finite product of ${\cal M}_1(E)$:
$$
    {\cal P} \,=\, \times_{n=0}^{N-1}\times_{(x,a)\in D_n}{\cal M}_1(E).
$$
The metric $d_{\infty,\M_{\rm{\scriptsize{TV}}}}^1$ obviously metricizes the product topology on $\overline{\cal P}_{\eins}={\cal P}$ and, as seen above, the set ${\cal M}_1(E)$ is compact w.r.t.\ $d_{\rm{\scriptsize{TV}}}$. It follows from Tychonoff's theorem (see, e.g., \cite[Theorem 2.2.8]{Dudley2002})
that ${\cal P}$ is compact w.r.t.\ $d_{\infty,\M_{\rm{\scriptsize{TV}}}}^1$ and therefore in particular relatively compact w.r.t.\ $d_{\infty,\M_{\rm{\scriptsize{TV}}}}^1$. Hence, Definition \mbox{\ref{def Gateaux-Hadamard-Frechet differentiability}}(b)  of `Hadamard differentiability' (i.e.\ Definition \mbox{\ref{def S differentiability}}  with ${\cal S}:={\cal S}_{\rm rc}$) simplifies insofar as one can simply require that the convergence in (\mbox{\ref{def eq for S-like D}})  holds uniformly in {\it all} $\Q\in{\cal P}$ for every sequence $(\varepsilon_m)\in(0,1]^{\N}$.

Under the imposed assumptions we may via (\mbox{\ref{clueing tp}}) identify any transition function $\P=(P_n)_{n=0}^{N-1}$ from ${\cal P}=\times_{n=0}^{N-1}\times_{(x,a)\in D_n}{\cal M}_1(E)$ with an element $\boldsymbol{p}$ of the set $\widetilde{\cal P}$ defined in (\mbox{\ref{def widetilde cal P}}) with $\mathfrak{d}:=(\mathfrak{t}_{0,i_0}+\sum_{n=1}^{N-1}\sum_{i=1}^\mathfrak{s}\mathfrak{t}_{n;i})\mathfrak{s}$, where $\mathfrak{t}_{n;i}:=\# A_n(x_i)$ and $x_1,\ldots,x_\mathfrak{s}$ is a (finite) enumeration of $E$. Then, imposing (without loss of generality) the metric
$$
    \begin{aligned}
    d_{\infty,\ell_1}(\boldsymbol{p},\boldsymbol{q}) & := \frac{1}{2}\max\Big\{\max_{k=1,\ldots,\mathfrak{t}_{0,i_0}}\|p_{0,i_0;a_{0,i_0;k}}-q_{0,i_0;a_{0,i_0;k}}\|_{\ell_1},\\
    & \qquad\qquad\qquad\max_{n=1,\ldots,N-1}\max_{i=1,\ldots,\mathfrak{s}}\max_{k=1,\ldots\mathfrak{t}_{n;i}}\|p_{n,i;a_{n,i;k}}-q_{n,i;a_{n,i;k}}\|_{\ell_1}\Big\}
    \end{aligned}
$$
on $\widetilde{\cal P}$, it is apparent that Definition \mbox{\ref{def of total diff}} is a special case of Definition \mbox{\ref{def S differentiability}}  with ${\cal S}:={\cal S}_{\rm rc}$.

Note that in the finite setting there exists for any fixed $\P\in{\cal P}$ an optimal strategy $\pi^{\P}\in\Pi$ w.r.t.\ $\P$, which means that the set $\Pi(\P)$ is non-empty; see, e.g., \cite[Proposition 4.4.3]{Puterman1994}. Also note that $\psi:\equiv 1$ provides a bounding function for the MDM $(\boldsymbol{X},\boldsymbol{A},\Q,\Pi,\boldsymbol{r})$ for any $\Q\in{\cal P}$. Thus condition (a) of Theorem \mbox{\ref{Thm - hadamard cal v}}  is satisfied for $\psi:\equiv 1$. According to Remark \mbox{\ref{remark verification assumptions}}(ii)--(iii), conditions (b) and (c) of Theorem \mbox{\ref{Thm - hadamard cal v}}  are satisfied for $\M':=\overline{\M}_{\rm{\scriptsize{TV}}}$ and $\psi:\equiv 1$, where $\overline{\M}_{\rm{\scriptsize{TV}}}$ is defined as in Example \mbox{\ref{examples prob metric - tv}}. Hence, in the finite setting the assumptions of Theorem \mbox{\ref{Thm - hadamard cal v}}  (with $\M:=\M_{\rm{\scriptsize{TV}}}$, $\M':=\overline{\M}_{\rm{\scriptsize{TV}}}$, and $\psi:\equiv 1$) are always fulfilled so that the representation (\mbox{\ref{Frechet opt value discrete case}}) of the `Fr\'echet derivative' of the value functional (with fixed initial state $x_0\in E$) always follows from part (ii) of Theorem \mbox{\ref{Thm - hadamard cal v}}. Take into account that in the finite setting `Fr\'echet differentiability' and `Hadamard differentiability' are equivalent.


\section{Supplement: Auxiliary definitions and results to Section \mbox{\ref{Sec - Formal definition of MDM}} }\label{Sec - Supplement to Section 2}


In this section we supplement the definitions and results of Section \mbox{\ref{Sec - Formal definition of MDM}}. The precise meaning of the definition in display (\mbox{\ref{def P pi}})  of the probability measure $\pr^{x_0,\P;\pi}$ on $(\Omega,{\cal F}):=(E^{N+1},{\cal E}^{\otimes(N+1)})$ is in view of (\mbox{\ref{def prob kernel E to E}})
\begin{eqnarray}
    \lefteqn{\pr^{x_0,\P;\pi}[B]} \label{def P pi - exactly} \\
    & := & \int_E\int_E\cdots\int_E\int_E \eins_{B}(y_0,\ldots,y_N)\,P_{N-1}^\pi(y_{N-1},dy_N) \nonumber \\
    & & \quad P_{N-2}^\pi(y_{N-2},dy_{N-1})\cdots P_{0}^\pi(y_{0},dy_1)\,\delta_{x_0}(dy_0)\nonumber\\
    & = & \int_E\int_E\cdots\int_E\int_E \eins_{B}(y_0,\ldots,y_N)\,P_{N-1}\big((y_{N-1},f_{N-1}(y_{N-1})),dy_N\big)\nonumber\\
    & & \quad P_{N-2}\big((y_{N-2},f_{N-2}(y_{N-2})),dy_{N-1}\big)\cdots P_{0}\big((y_{0},f_{0}(y_{0})),dy_1\big)\,\delta_{x_0}(dy_0)\nonumber
\end{eqnarray}
for $B\in{\cal F}$, for any given $x_0\in E$, $\P=(P_n)_{n=0}^{N-1}\in{\cal P}$, and $\pi=(f_n)_{n=0}^{N-1}\in\Pi$.

By a (regular version of the) factorized conditional distribution of $X$ given $Y$ under $\pr^{x_0,\P;\pi}$ we mean a probability kernel $\pr^{x_0,\P;\pi}_{X\|Y}(\,\cdot\,,\bullet)$ for which for every $B\in{\cal E}$ the random variable $\omega\mapsto\pr^{x_0,\P;\pi}_{X\|Y}(Y(\omega),B)$ is a conditional probability of $\{X\in B\}$ given $Y$ under $\pr^{x_0,\P;\pi}$. This object is only $\pr_Y^{x_0,\P;\pi}$-a.s.\ unique. Thus the formulation of (ii)--(viii) in the following lemma is somewhat sloppy. Assertion (v) in fact means that the probability kernel $P_n((\,\cdot\,,f_n(\,\cdot\,)),\,\bullet\,)$ provides a (regular version of the) factorized conditional distribution of $X_{n+1}$ given $X_n$ under $\pr^{x_0,\P;\pi}$, and analogously for parts (ii)--(iv) and (vi)--(viii). Note that it is also customary to write $\pr^{x_0,\P;\pi}[X\in\,\bullet\,\|Y=\,\cdot\,]$ instead of $\pr^{x_0,\P;\pi}_{X\|Y}(\,\cdot\,,\bullet)$; see, for instance, (ii)--(iv) in Subsection \mbox{\ref{Subsec - Markov decision process}}.

\begin{lemma}\label{lemma on X and P}
For any $\P=(P_n)_{n=0}^{N-1}\in{\cal P}$, $\pi=(f_n)_{n=0}^{N-1}\in\Pi$,
$x_0,\widetilde{x}_0,x_1,\ldots,x_n\in E$ and $1 \leq n < k \leq N$ as well as $x_m\in E$ and $m=1,\ldots,N$ we have
\begin{enumerate}
\item[{\rm (i)}] $\pr^{x_0,\P;\pi}_{X_0}[\,\bullet\,] = \delta_{x_0}[\,\bullet\,]$.

\item[{\rm (ii)}] $\pr^{x_0,\P;\pi}_{X_0\| X_0}(\widetilde{x}_0,\,\bullet\,) = \delta_{x_0}[\,\bullet\,]$.

\item[{\rm (iii)}] $\pr^{x_0,\P;\pi}_{X_1\| X_0}(\widetilde{x}_0,B) = P_0\big((x_0,f_0(x_0)),B\big)$.

\item[{\rm (iv)}] $\pr^{x_0,\P;\pi}_{X_{n+1}\|(X_0,X_1,\ldots,X_n)}((\widetilde{x}_0,x_1,\ldots,x_n),\,\bullet\,) = P_n\big((x_n,f_n(x_n)),\,\bullet\,\big)$.

\item[{\rm (v)}] $\pr^{x_0,\P;\pi}_{X_{n+1}\| X_n}(x_n,\,\bullet\,) = P_n\big((x_n,f_n(x_n)),\,\bullet\,\big)$.

\item[{\rm (vi)}] $\pr^{x_0,\P;\pi}_{X_m\| X_0}(\widetilde{x}_0,\,\bullet\,) = \pr^{x_0,\P;\pi}_{X_m}[\,\bullet\,] = \pr^{x_0,\P;\pi}_{X_{1}\| X_{0}}\cdots\pr^{x_0,\P;\pi}_{X_m\| X_{m-1}}(x_0,\,\bullet\,)$.

\item[{\rm (vii)}] $\pr^{x_0,\P;\pi}_{X_k\| X_n}(x_n,\,\bullet\,) = \pr^{x_0,\P;\pi}_{X_{n+1}\| X_n}\cdots\pr^{x_0,\P;\pi}_{X_k\| X_{k-1}}(x_n,\,\bullet\,)$.

\item[{\rm (viii)}] $\pr^{x_0,\P;\pi}_{X_m\| X_m}(x_m,\,\bullet\,) = \delta_{x_m}[\,\bullet\,]$.
\end{enumerate}
\end{lemma}

For parts (vi) and (vii) in the preceding lemma note that the compositions on the right-hand side are for every $B\in{\cal E}$ defined by
\begin{eqnarray}\label{lemma on X and P - eq 10}
	\lefteqn{\pr^{x_0,\P;\pi}_{X_{1}\| X_{0}}\cdots\pr^{x_0,\P;\pi}_{X_m\| X_{m-1}}(x_0,B)} \nonumber \\
	& \hspace{-2mm} := \hspace{-2mm} & \int_E\cdots\int_E \eins_B(y_m)\,P_{m-1}\big((y_{m-1},f_{m-1}(y_{m-1})),dy_{m}\big)\cdots P_0\big((x_0,f_0(x_0)),dy_1\big) \nonumber \\
\end{eqnarray}
and
\begin{eqnarray}\label{lemma on X and P - eq 20}
	\lefteqn{\pr^{x_0,\P;\pi}_{X_{n+1}\| X_n}\cdots\pr^{x_0,\P;\pi}_{X_k\| X_{k-1}}(x_n,B)} \nonumber \\
	& \hspace{-2mm} := \hspace{-2mm} & \int_E\cdots\int_E \eins_B(y_k)\,P_{k-1}\big((y_{k-1},f_{k-1}(y_{k-1})),dy_k\big)\cdots P_n\big((x_n,f_n(x_n)),dy_{n+1}\big). \nonumber \\
\end{eqnarray}

\begin{proof}
First of all it is clear that assertion (i) holds. Thus it suffices to show assertions (ii)--(viii).

(ii): The claim holds true, because
\begin{eqnarray*}
    \lefteqn{\ex^{x_0,\P;\pi}\big[\delta_{X_0}[B]\eins_{B_{1}}(X_0)\big]}\\
    & = & \int_\Omega \delta_{X_0(\omega)}[B]\eins_{B_{1}}(X_0(\omega))\,\pr^{x_0,\P;\pi}(d\omega)\\
    & = & \int_E\int_E\cdots\int_E\delta_{y_0}[B]\eins_{B_{1}}(y_0)\,\\
    & & \quad P_{N-1}\big((y_{N-1},f_n(y_{N-1})),dy_N\big)\cdots P_0\big((y_0,f_0(y_0)),dy_1\big)\,\delta_{x_0}(dy_0)\\
    & = & \int_E\delta_{y_0}[B]\eins_{B_{1}}(y_0)\,\delta_{x_0}(dy_0)
    \,=\, \delta_{x_0}[B]\eins_{B_{1}}(x_0)\,=\,\delta_{x_0}[B\cap B_1]\\
    & = & \pr^{x_0,\P;\pi}[\{X_0\in B\}\cap\{X_0\in B_{1}\}]
\end{eqnarray*}
for any $B\in{\cal E}$ and $B_{1}\in{\cal E}$.

(iii): The claim holds true, because
\begin{eqnarray*}
    \lefteqn{\ex^{x_0,\P;\pi}\big[P_{0}\big((X_0,f_0(X_0)),B\big)\eins_{B_1}(X_0)\big]}\\
    & = & \int_\Omega P_{0}\big((X_0(\omega),f_0(X_0(\omega)),B\big)\eins_{B_1}(X_0(\omega))\,\pr^{x_0,\P;\pi}(d\omega)\\
    & = & \int_E\int_E\cdots\int_E P_{0}\big((y_0,f_0(y_0)),B\big)\eins_{B_1}(y_0)\\
   & & \quad P_{N-1}\big((y_{N-1},f_{N-1}(y_{N-1})),dy_N\big)\cdots P_0\big((y_0,f_0(y_0)),dy_1\big)\,\delta_{x_0}(dy_0)\\
    & = & \int_E\int_E \eins_{B}(y_1)\,P_0\big((y_0,f_0(y_0)),dy_1\big)\eins_{B_1}(y_0)\,\delta_{x_0}(dy_0)\\
    & = & \int_E\int_E \eins_{B_1\times B}(y_0,y_1)\,P_0\big((y_0,f_0(y_0)),dy_1\big)\,\delta_{x_0}(dy_0)\\
    & = & \pr^{x_0,\P;\pi}[\{X_1\in B\}\cap\{X_0\in B_1\}]
\end{eqnarray*}
for any $B\in{\cal E}$ and $B_1\in{\cal E}$.

(iv): The claim holds true, because
\begin{eqnarray*}
    \lefteqn{\ex^{x_0,\P;\pi}\big[P_n\big((X_n,f_n(X_n)),B\big)\eins_{B_{n+1}}(X_0,\ldots,X_n)\big]}\\
    & = & \int_\Omega P_n\big((X_n(\omega),f_n(X_n(\omega))),B\big)\eins_{B_{n+1}}(X_0(\omega),\ldots,X_n(\omega))\,\pr^{x_0,\P;\pi}(d\omega)\\
    & = & \int_E\int_E\cdots\int_E P_n\big((y_n,f_n(y_n)),B\big)\eins_{B_{n+1}}(y_0,\ldots,y_n)\\
    & & \quad P_{N-1}\big((y_{N-1},f_{N-1}(y_{N-1})),dy_N\big)\cdots P_0\big((y_0,f_0(y_0)),dy_1\big)\,\delta_{x_0}(dy_0)\\
    & = & \int_E\int_E\cdots\int_E \int_E\eins_{B}(y_{n+1}) P_n\big((y_n,f_n(y_n)),dy_{n+1}\big)\,\eins_{B_{n+1}}(y_0,\ldots,y_n)\\
    & & \quad P_{n-1}\big((y_{n-1},f_{n-1}(y_{n-1})),dy_n\big)\cdots P_0\big((y_0,f_0(y_0)),dy_1\big)\,\delta_{x_0}(dy_0)\\
    & = & \int_E\int_E\cdots\int_E \int_E\eins_{B_{n+1}\times B}(y_0,\ldots,y_n,y_{n+1})\,P_n\big((y_n,f_n(y_n)),dy_{n+1}\big)\\
    & & \quad P_{n-1}\big((y_{n-1},f_{n-1}(y_{n-1})),dy_n\big)\cdots P_0\big((y_0,f_0(y_0)),dy_1\big)\,\delta_{x_0}(dy_0)\\
    & = & \pr^{x_0,\P;\pi}[\{X_{n+1}\in B\}\cap\{(X_0,\ldots,X_n)\in B_{n+1}\}]
\end{eqnarray*}
for any $B\in{\cal E}$ and $B_{n+1}\in{\cal E}^{\otimes(n+1)}$.

(v): As in the proof of (iv) we obtain
$$
    \ex^{x_0,\P;\pi}\big[P_n\big((X_n,f_n(X_n)),B\big)\eins_{B_{1}}(X_n)\big]=\pr^{x_0,\P;\pi}[\{X_{n+1}\in B\}\cap\{X_n\in B_{1}\}]
$$
for any $B\in{\cal E}$ and $B_{1}\in{\cal E}$.

(vi): First of all, it is known from the Chapman--Kolmogorov relation (see, e.g., \cite[p.\,143]{Kallenberg2002}) that the identity
\begin{equation}\label{lemma on X and P - eq 30}
	\pr^{x_0,\P;\pi}_{X_{m}\| X_{j}}(x_{j},\,\bullet\,) = \int_E \pr^{x_0,\P;\pi}_{X_m\|X_l}(y',\,\bullet\,)\,\pr^{x_0,\P;\pi}_{X_l\|X_j}(x_j,dy')
\end{equation}
holds for any $x_j\in E$ and $0\leq j \leq l < m\leq N$. Hence, by iterating (\mbox{\ref{lemma on X and P - eq 30}}) we obtain by means of parts (iii) and (v) as well as (\mbox{\ref{lemma on X and P - eq 10}})
\begin{eqnarray}\label{lemma on X and P - eq 40}
	\lefteqn{\pr^{x_0,\P;\pi}_{X_m\| X_0}(\widetilde{x}_{0},B)} \nonumber \\
	& = & \int_E\cdots\int_E\pr^{x_0,\P;\pi}_{X_m\|X_{m-1}}(y_{m-1},B)\,\pr^{x_0,\P;\pi}_{X_{m-1}\|X_{m-2}}(y_{m-2},dy_{m-1}) \cdots \pr^{x_0,\P;\pi}_{X_1\|X_0}(\widetilde{x}_0,dy_1) \nonumber\\
	& = & \int_E\cdots\int_E P_{m-1}\big((y_{m-1},f_{m-1}(y_{m-1})),B\big)\,P_{m-2}\big((y_{m-2},f_{m-2}(y_{m-2})),dy_{m-1}\big) \nonumber \\
	& &  \quad \cdots P_0\big((x_0,f_0(x_0)),dy_1\big) \nonumber \\
	& = & \int_E\cdots\int_E \eins_B(y_m)\,P_{m-1}\big((y_{m-1},f_{m-1}(y_{m-1})),dy_{m}\big)\cdots P_0\big((x_0,f_0(x_0)),dy_1\big) \nonumber \\
	& = & \pr^{x_0,\P;\pi}_{X_{1}\| X_{0}}\cdots\pr^{x_0,\P;\pi}_{X_m\| X_{m-1}}(x_0,B)
\end{eqnarray}
for any $B\in{\cal E}$. Moreover, as an immediate consequence of the characterization of the (regular version of the) factorized conditional distribution, we have in view of (\mbox{\ref{lemma on X and P - eq 40}}) and part (i)
\begin{eqnarray*}
	\pr^{x_0,\P;\pi}_{X_{m}}[B]
	& = & \int_E \pr^{x_0,\P;\pi}_{X_{m}\| X_{0}}(y',B)\,\pr^{x_0,\P;\pi}_{X_0}(dy')
	\,=\, \int_E \pr^{x_0,\P;\pi}_{X_{1}\| X_{0}}\cdots\pr^{x_0,\P;\pi}_{X_m\| X_{m-1}}(y',B)\,\delta_{x_0}(dy') \\
	& = & \pr^{x_0,\P;\pi}_{X_{1}\| X_{0}}\cdots\pr^{x_0,\P;\pi}_{X_m\| X_{m-1}}(x_0,B)
\end{eqnarray*}
for any $B\in{\cal E}$.

(vii): As in the proof of (vi) we obtain by iterating (\mbox{\ref{lemma on X and P - eq 30}}) along with part (v) and (\mbox{\ref{lemma on X and P - eq 20}})
$$
	\pr^{x_0,\P;\pi}_{X_k\| X_n}(x_n,B) = \pr^{x_0,\P;\pi}_{X_{n+1}\| X_n}\cdots\pr^{x_0,\P;\pi}_{X_{k}\| X_{k-1}}(x_n,B)
$$
for any $B\in{\cal E}$.

(viii): Analogously to the proof of (ii) we obtain by means of part (vi)
$$
\ex^{x_0,\P;\pi}\big[\delta_{X_m}[B]\eins_{B_1}(X_m)\big] = \pr^{x_0,\P;\pi}[\{X_m\in B\}\cap\{X_m\in B_1\}]
$$
for any $B\in{\cal E}$ and $B_1\in{\cal E}$. This completes the proof.
\end{proof}

Note that the factorized conditional distributions in parts (ii)--(iii) and (vi) of Lemma \mbox{\ref{lemma on X and P}} are constant w.r.t.\ $\widetilde{x}_0\in E$. Also note that in view of part (vii) of Lemma \mbox{\ref{lemma on X and P}} the probability measure $\pr^{x_0,\P;\pi}_{X_k\| X_n}(x_n,\,\bullet\,)$ can be seen as a $(k-n)$-step transition probability from stages $n$ to $k$ given state $x_n$.

Recall that $\M(E)$ stands for the set of all $({\cal E},{\cal B}(\R))$-measurable maps in $\R^E$ and that $\ex_{n,x_n}^{x_0,\P;\pi}$ refers to the expectation w.r.t.\ the factorized conditional distribution $\pr^{x_0,\P;\pi}[\,\bullet\,\| X_n=x_n]$. Moreover we denote by $L^1(\Omega,{\cal F},\pr^{x_0,\P;\pi})$ the usual $L^1$-space on $(\Omega,{\cal F},\pr^{x_0,\P;\pi})$.

\begin{lemma}\label{lemma on integrals of functions of X}
Let $x_0\in E$, $\P=(P_n)_{n=0}^{N-1}\in{\cal P}$, and $\pi=(f_n)_{n=0}^{N-1}\in\Pi$. Moreover let $h\in\M(E)$ such that $h(X_n)\in L^1(\Omega,{\cal F},\pr^{x_0,\P;\pi})$ for all $n=0,\ldots,N$. Then for any
$\widetilde{x}_0,x_n\in E$ and $1\le n<k\le N$ as well as $x_m\in E$ and $m=1,\ldots,N$ we have
\begin{enumerate}
\item[{\rm (i)}] $\ex^{x_0,\P;\pi}[h(X_0)] = h(x_0)$.

\item[{\rm (ii)}] $\ex_{0,\widetilde{x}_0}^{x_0,\P;\pi}[h(X_0)] = h(x_0)$.

\item[{\rm (iii)}] $\ex_{m,x_m}^{x_0,\P;\pi}[h(X_m)] = h(x_m)$.

\item[{\rm (iv)}] $\ex_{0,\widetilde{x}_0}^{x_0,\P;\pi}[h(X_m)] = \ex^{x_0,\P;\pi}[h(X_m)] = \int_E h(y_m)\,\pr^{x_0,\P;\pi}_{X_{1}\| X_{0}}\cdots\pr^{x_0,\P;\pi}_{X_m\| X_{m-1}}(x_0,dy_m)$.

\item[{\rm (v)}] $\ex_{n,x_n}^{x_0,\P;\pi}[h(X_k)] = \int_E h(y_k)\,\pr^{x_0,\P;\pi}_{X_{n+1}\| X_n}\cdots\pr^{x_0,\P;\pi}_{X_k\| X_{k-1}}(x_n,dy_k)$.
\end{enumerate}
Moreover the right-hand side of parts (iv) and (v) can be represented as
\begin{eqnarray*}
	\lefteqn{\int_E h(y_m)\,\pr^{x_0,\P;\pi}_{X_{1}\| X_{0}}\cdots\pr^{x_0,\P;\pi}_{X_m\| X_{m-1}}(x_0,dy_m)}\\
	& = & \int_E\cdots\int_E h(y_m)\,P_{m-1}\big((y_{m-1},f_{m-1}(y_{m-1})),dy_m\big)\cdots P_0\big((x_0,f_0(x_0)),dy_1\big)
\end{eqnarray*}
and
\begin{eqnarray*}
	\lefteqn{\int_E h(y_k)\,\pr^{x_0,\P;\pi}_{X_{n+1}\| X_n}\cdots\pr^{x_0,\P;\pi}_{X_k\| X_{k-1}}(x_n,dy_k)}\\
	& = & \int_E\cdots\int_E h(y_k)\,P_{k-1}\big((y_{k-1},f_{k-1}(y_{k-1})),dy_k\big)\cdots P_n\big((x_n,f_n(x_n)),dy_{n+1}\big).
\end{eqnarray*}
\end{lemma}

\begin{proof}
First of all, it is easily seen that the identities
\begin{equation}\label{formula fact cond exp - eq 10}
\ex^{x_0,\P;\pi}[h(X_{m})] = \int_E h(y')\,\pr^{x_0,\P;\pi}_{X_m}(dy')
\end{equation}
and
\begin{equation}\label{formula fact cond exp - eq 20}
\ex_{j,x_j}^{x_0,\P;\pi}[h(X_{m})] = \int_E h(y')\,\pr^{x_0,\P;\pi}_{X_m\|X_j}(x_j,dy')
\end{equation}
hold for any $x_j\in E$ and $0\le j \le m\le N$. 

(i): The claim is an immediate consequence of (\mbox{\ref{formula fact cond exp - eq 10}}) and part (i) of Lemma \mbox{\ref{lemma on X and P}}.

(ii)--(iii): The assertions follow from (\mbox{\ref{formula fact cond exp - eq 20}}) along with parts (ii) and (viii) of Lemma \mbox{\ref{lemma on X and P}}, respectively.

(iv): For the assertions it suffices in view of (\mbox{\ref{formula fact cond exp - eq 10}})--(\mbox{\ref{formula fact cond exp - eq 20}}) to show that
\begin{equation}\label{formula fact cond exp - eq 30}
\int_E h(y_m)\,\pr^{x_0,\P;\pi}_{X_m\|X_0}(\widetilde{x}_0,dy_m) = \int_E h(y_m)\,\pr^{x_0,\P;\pi}_{X_{1}\| X_{0}}\cdots\pr^{x_0,\P;\pi}_{X_m\| X_{m-1}}(x_0,dy_m)
\end{equation}
and
\begin{equation}\label{formula fact cond exp - eq 40}
\int_E h(y_m)\,\pr^{x_0,\P;\pi}_{X_m}(dy_m) = \int_E h(y_m)\,\pr^{x_0,\P;\pi}_{X_{1}\| X_{0}}\cdots\pr^{x_0,\P;\pi}_{X_m\| X_{m-1}}(x_0,dy_m).
\end{equation}
Clearly, in view of part (vi) of Lemma \mbox{\ref{lemma on X and P}}, the assertions in (\mbox{\ref{formula fact cond exp - eq 30}}) and (\mbox{\ref{formula fact cond exp - eq 40}}) are valid for indicator functions and thus by linearity for simple functions. The latter assertions can be extended by the Monotone Convergence theorem to arbitrary nonnegative maps $h\in\M(E)$. Since the integrals on the left-hand sides of (\mbox{\ref{formula fact cond exp - eq 30}}) and (\mbox{\ref{formula fact cond exp - eq 40}}) exist and are finite (recall that $h(X_n)\in L^1(\Omega,{\cal F},\pr^{x_0,\P;\pi})$ for all $n=0,\ldots,N$ by assumption), it follows that the equalities in (\mbox{\ref{formula fact cond exp - eq 30}}) and (\mbox{\ref{formula fact cond exp - eq 40}}) hold even for all $h\in\M(E)$.

(v): Analogously to the proof of (\mbox{\ref{formula fact cond exp - eq 30}}) we obtain by means of (\mbox{\ref{formula fact cond exp - eq 20}})
$$
\ex_{n,x_n}^{x_0,\P;\pi}[h(X_k)] = \int_E h(y_k)\,\pr^{x_0,\P;\pi}_{X_{n+1}\| X_n}\cdots\pr^{x_0,\P;\pi}_{X_k\| X_{k-1}}(x_n,dy_k).
$$

The additional assertions can be verified easily by means of (\mbox{\ref{lemma on X and P - eq 10}}) and (\mbox{\ref{lemma on X and P - eq 20}}) with the same arguments as in the proof of (\mbox{\ref{formula fact cond exp - eq 30}}) and (\mbox{\ref{formula fact cond exp - eq 40}}). This completes the proof.
\end{proof}

Note that (for any given $x_0\in E$, $\P\in{\cal P}$, and $\pi\in\Pi$) the assumption $h(X_n)\in L^1(\Omega,{\cal F},\pr^{x_0,\P;\pi})$ (for some $h\in\M(E)$ and any $n=0,\ldots,N$) is not trivially satisfied. It holds, for example, if $\psi$ is a bounding function for the MDM $(\boldsymbol{X},\boldsymbol{A},\P,\Pi,\boldsymbol{r})$ (in the sense of Definition \mbox{\ref{def bounding function}}  with ${\cal P}':=\{\P\}$) and if $h\in\M_\psi(E)$ (with $\M_\psi(E)$ as in Subsection \mbox{\ref{Subsec - Bounding functions}}). In this case it can be easily verified by means of part (c) of Definition \mbox{\ref{def bounding function}}  (with ${\cal P}':=\{\P\}$) that indeed $h(X_n)\in L^1(\Omega,{\cal F},\pr^{x_0,\P;\pi})$ for all $n=0,\ldots,N$.


\section{Supplement: Proofs of lemmas in Section \mbox{\ref{Sec - Hadamard differentiability chapter}} }\label{Sec - Proof of results from Section 4}


\subsection{Proof of Lemma \mbox{\ref{lemma suff cond for stand cond - pre}} }\label{Subsec - proof lemma suff cond for stand cond - pre}


Fix $x_0\in E$.
By assumption there exist finite constants $K_1,K_3>0$ such that in view of part (v) of Lemma \mbox{\ref{lemma on integrals of functions of X}} as well as parts (a) and (c) of Definition \mbox{\ref{def bounding function}}
\begin{eqnarray*}
    \lefteqn{\ex_{n,x_n}^{x_0,\P;\pi}\big[|r_k(X_k,f_k(X_k))|\big]} \nonumber\\
    & \le & \ex_{n,x_n}^{x_0,\P;\pi}[K_1\psi(X_k)] \nonumber\\
    & = & K_1\int_E\cdots\int_E\int_E\psi(y_k)\,P_{k-1}\big((y_{k-1},f_{k-1}(y_{k-1})),dy_k\big) \nonumber\\
    & & \qquad P_{k-2}\big((y_{k-2},f_{k-2}(y_{k-2})),dy_{k-1}\big)\cdots P_n\big((x_n,f_n(x_n)),dy_{n+1}\big) \nonumber\\
    & \le  & K_1K_3^{k-n}\psi(x_n)
\end{eqnarray*}
for any $x_n\in E$, $\P=(P_n)_{n=0}^{N-1}\in{\cal P}'$, $\pi=(f_n)_{n=0}^{N-1}\in\Pi$, and $1\le n< k\le N-1$. Moreover in view of part (iii) of Lemma \mbox{\ref{lemma on integrals of functions of X}} and part (a) of Definition \mbox{\ref{def bounding function}}  we have
$$
    \ex_{n,x_n}^{x_0,\P;\pi}\big[|r_n(X_n,f_n(X_n))|\big]\,=\,|r_n(x_n,f_n(x_n))|\,\le\, K_1\psi(x_n)
$$
for any $x_n\in E$, $\P\in{\cal P}'$, $\pi=(f_n)_{n=0}^{N-1}\in\Pi$, and $n=1,\ldots,N-1$. Similarly, we find by assumption some finite constant $K_2>0$ such that in view of parts (iii) and (v) of Lemma \mbox{\ref{lemma on integrals of functions of X}} as well as parts (b) and (c) of Definition \mbox{\ref{def bounding function}}
$$
    \ex_{n,x_n}^{x_0,\P;\pi}\big[|r_N(X_N)|\big]\,\le\,K_2K_3^{N-n}\psi(x_n)
$$
for any $x_n\in E$, $\P\in{\cal P}'$, $\pi\in\Pi$, and $n=1,\ldots,N$. In the same way we obtain with parts (ii) and (iv) of Lemma \mbox{\ref{lemma on integrals of functions of X}} and the characteristic properties of the bounding function $\psi$
$$
	\ex_{0,x_0}^{x_0,\P;\pi}\big[|r_k(X_k,f_k(X_k))|\big]\,\le\, K_1K_3^k\psi(x_0)
$$
and
$$
	\ex_{0,x_0}^{x_0,\P;\pi}\big[|r_N(X_N)|\big]\,\le\,K_2K_3^{N}\psi(x_0)
$$
for any $\P\in{\cal P}'$, $\pi=(f_n)_{n=0}^{N-1}\in\Pi$, and $k=0,\ldots,N-1$. Then Assumption {\bf (A)}  holds (uniformly in $\P\in{\cal P}'$).
Moreover by choosing $C_n:=K_1\sum_{k=n}^{N-1}K_3^{k-n} + K_2K_3^{N-n}$ we have $\|V_n^{\P;\pi}\|_{\psi}\le C_n$ and hence $V_n^{\P;\pi}(\cdot)\in\M_\psi(E)$ for every $\P\in{\cal P}'$, $\pi\in\Pi$, and $n=0,\ldots,N$. This completes the proof. \hfill\proofendsign


\subsection{Proof of Lemma \mbox{\ref{lemma HD characterization}} }\label{Subsec - Proof Lemma hd characterization}


(i): Let ${\cal V}$ be `Hadamard differentiable' at $\P$ w.r.t.\ $(\M,\phi)$ with `Hadamard derivative' $\dot{\cal V}_{\P}$. To show that (\mbox{\ref{HD characterization}})  holds, pick a triplet $(\Q, (\Q_m), (\varepsilon_m))\in{\cal P}_\psi\times{\cal P}_\psi^\N\times(0,1]^\N$ with $d_{\infty,\M}^{\phi}(\Q_m,\Q)\to 0$ and $\varepsilon_m\to 0$. Then, the set ${\cal K}:=\{\Q_m:m\in\N\}\,(\subseteq{\cal P}_\psi)$ is clearly relatively compact. Using this and the assumption we obtain
\begin{eqnarray*}
    \lefteqn{\limsup_{m\to\infty}\Big\|\frac{{\cal V}(\P+\varepsilon_m(\Q_m-\P))-{\cal V}(\P)}{\varepsilon_m}-\dot{\cal V}_{\P}(\Q-\P)\Big\|_L}\\
    & \le & \limsup_{m\to\infty}\Big\|\frac{{\cal V}(\P+\varepsilon_m(\Q_m-\P))-{\cal V}(\P)}{\varepsilon_m}-\dot{\cal V}_{\P}(\Q_m-\P)\Big\|_L\\
    & & +\,\limsup_{m\to\infty}\big\|\dot{\cal V}_{\P}(\Q_m-\P)-\dot{\cal V}_{\P}(\Q-\P)\big\|_L\\
    & = & 0+0\,=\,0.
\end{eqnarray*}

(ii): Assume that there exists an $(\M,\phi)$-continuous map $\dot{\cal V}_{\P}:{\cal P}_\psi^{\P;\pm}\rightarrow L$ such that (\mbox{\ref{HD characterization}})  holds for each triplet $(\Q, (\Q_m), (\varepsilon_m))\in{\cal P}_\psi\times{\cal P}_\psi^\N\times(0,1]^\N$ with $d_{\infty,\M}^{\phi}(\Q_m,\Q)\to 0$ and $\varepsilon_m\to 0$. Assume by way of contradiction that $\dot{\cal V}_{\P}$ is {\em not} the `Hadamard derivative' of ${\cal V}$ at $\P$ w.r.t.\ $(\M,\phi)$, i.e.\ that there is some relatively compact set ${\cal K}\subseteq{\cal P}_\psi$ and a sequence $(\varepsilon_m)\in(0,1]^\N$ with $\varepsilon_m\to 0$ such that (\mbox{\ref{def eq for S-like D}})  does {\em not} hold uniformly in $\Q\in{\cal K}$. Then there exist $\delta>0$ and $(\Q_m)\in{\cal K}^\N$ such that
\begin{equation}\label{hadamard differentiability consequence - PROOF - 10}
    \Big\|\frac{{\cal V}(\P+\varepsilon_m(\Q_m-\P))-{\cal V}(\P)}{\varepsilon_m} - \dot{\cal V}_{\P}(\Q_m-\P)\Big\|_L\ge\delta\quad\mbox{for all }m\in\N.
\end{equation}
Since ${\cal K}$ is relatively compact, we can find a subsequence $(\Q_m')$ of $(\Q_m)$ such that $d_{\infty,\M}^{\phi}(\Q_m',\Q')\to 0$ for some $\Q'\in{\cal P}_\psi$. Along with the $(\M,\phi)$-continuity of the map $\dot{\cal V}_{\P}:{\cal P}_\psi^{\P;\pm}\rightarrow L$ and (\mbox{\ref{hadamard differentiability consequence - PROOF - 10}}) (with $\Q_m$ replaced by $\Q_m'$), we obtain
\begin{eqnarray*}
    \lefteqn{\liminf_{m\to\infty}\Big\|\frac{{\cal V}(\P+\varepsilon_m(\Q_m'-\P))-{\cal V}(\P)}{\varepsilon_m} - \dot{\cal V}_{\P}(\Q'-\P)\Big\|_L}\\
    & = & \liminf_{m\to\infty}\Big\|\frac{{\cal V}(\P+\varepsilon_m(\Q_m'-\P))-{\cal V}(\P)}{\varepsilon_m} - \dot{\cal V}_{\P}(\Q_m'-\P)\Big\|_L\\
    & & +~ \liminf_{m\to\infty}\big\|\dot{\cal V}_{\P}(\Q_m'-\P) - \dot{\cal V}_{\P}(\Q'-\P)\big\|_L\\
    & = & \liminf_{m\to\infty}\Big\|\frac{{\cal V}(\P+\varepsilon_m(\Q_m'-\P))-{\cal V}(\P)}{\varepsilon_m} - \dot{\cal V}_{\P}(\Q_m'-\P)\Big\|_L + 0 ~\ge~\delta
\end{eqnarray*}
which contradicts the assumption (\mbox{\ref{HD characterization}}). The proof is now complete.
\hfill\proofendsign


\section{Supplement: Proof of Theorem \mbox{\ref{Thm - hadamard cal v}} } \label{Sec - Proof of Theorem HD}


Under assumption (a) of Theorem \mbox{\ref{Thm - hadamard cal v}}, the value functional ${\cal V}_n^{x_n}$ admits for any $x_n\in E$ and $n=0,\ldots,N$ the representation
\begin{equation}\label{representation cal v}
	{\cal V}_n^{x_n} = \Psi \circ \Upsilon_n^{x_n}
\end{equation}
with maps $\Upsilon_n^{x_n}:{\cal P}_\psi\rightarrow \ell^\infty(\Pi)$ and $\Psi:\ell^\infty(\Pi)\rightarrow \R$ defined by
\begin{equation}\label{def of Upsilon}
	\Upsilon_n^{x_n}(\P):=\big({\cal V}_n^{x_n;\pi}(\P)\big)_{\pi\in\Pi}\quad\mbox{ and }\quad \Psi\big((w(\pi))_{\pi\in\Pi}\big):=\sup_{\pi\in\Pi}w(\pi),
\end{equation}
where $\ell^\infty(\Pi)$ stands for the space of all bounded real-valued functions on $\Pi$ equipped with the  sup-norm $\|\cdot\|_{\infty}$. It is easily seen that assumption (a) ensures that the map $\Upsilon_n^{x_n}$ is well defined for any $x_n\in E$ and $n=0,\ldots,N$, i.e.\ that $({\cal V}_n^{x_n;\pi}(\P))_{\pi\in\Pi}\in\ell^\infty(\Pi)$ for any $x_n\in E$, $\P\in{\cal P}_\psi$, and $n=0,\ldots,N$; see Lemma \mbox{\ref{lemma suff cond for stand cond - pre}}  (with ${\cal P}':=\{\P\}$). In Subsection \mbox{\ref{Subsec - differentiablility of upsilon}} we will show that under the assumptions of Theorem \mbox{\ref{Thm - hadamard cal v}}  and for any $x_n\in E$ and $n=0,\ldots,N$ the map $\Upsilon_n^{x_n}$ is `Fr\'echet differentiable' at $\P$ w.r.t.\ $(\M,\psi)$ with `Fr\'echet derivative' $\dot\Upsilon_{n;\P}^{x_n}:{\cal P}_\psi^{\P;\pm}\rightarrow\ell^\infty(\Pi)$ given by
\begin{equation}\label{def derivative of Upsilon}
    \dot\Upsilon_{n;\P}^{x_n}(\Q-\P):=\big(\dot{\cal V}_{n;\P}^{x_n;\pi}(\Q-\P)\big)_{\pi\in\Pi}
\end{equation}
(the well-definiteness of $\dot\Upsilon_{n;\P}^{x_n}$ is again ensured by assumption (a)).
Together with the Hadamard differentiability of the map $\Psi$ (which is known from \cite{Roemisch2004}), this implies assertion (ii) of Theorem \mbox{\ref{Thm - hadamard cal v}} ; see Subsection \mbox{\ref{Subsec - hadamard cal v}} for details. Assertion (i) of Theorem \mbox{\ref{Thm - hadamard cal v}}  is an immediate consequence of Theorem \mbox{\ref{Thm - hadamard upsilon}} below.
\hfill\proofendsign


\subsection{`Fr\'echet differentiability' of $\Upsilon_n^{x_n}$}\label{Subsec - differentiablility of upsilon}


The following theorem is a direct consequence of Lemmas \mbox{\ref{lemma continuity hd}} and \mbox{\ref{lemma diff quot hd}} ahead.

\begin{theorem}\label{Thm - hadamard upsilon}
Let $\M\subseteq\M_\psi(E)$, and fix $\P\in{\cal P}_\psi$. Let $\M'$ be any generator of $d_\M$ and assume that conditions (a)--(c) of Theorem \mbox{\ref{Thm - hadamard cal v}}  (with this $\M'$) hold. Then for any $x_n\in E$ and $n=0,\ldots,N$ the map $\Upsilon_n^{x_n}:{\cal P}_\psi\rightarrow \ell^\infty(\Pi)$ defined by (\mbox{\ref{def of Upsilon}}) is `Fr\'echet differentiable' at $\P$ w.r.t.\ $(\M,\psi)$ with `Fr\'echet derivative' $\dot\Upsilon_{n;\P}^{x_n}:{\cal P}_\psi^{\P;\pm}\rightarrow\ell^\infty(\Pi)$ given by (\mbox{\ref{def derivative of Upsilon}}).
\end{theorem}

\begin{lemma}\label{lemma continuity hd}
Under the assumptions of Theorem \mbox{\ref{Thm - hadamard upsilon}} and for any fixed $x_n\in E$ and $n=0,\ldots,N$, the map $\dot\Upsilon_{n;\P}^{x_n}:{\cal P}_\psi^{\P;\pm}\rightarrow\ell^\infty(\Pi)$ given by (\mbox{\ref{def derivative of Upsilon}}) is $(\M,\psi)$-continuous.
\end{lemma}

\begin{proof}
As a simple consequence of the definition of the Minkowski functional $\rho_{\M'}$ (see (\mbox{\ref{def minkowski functional}})) we have
\begin{equation}\label{minkowski - eq10}
    \Big|\int_E h\,d\mu - \int_E h\,d\nu\Big|\,\le\,\rho_{\M'}(h)\cdot d_{\M}(\mu,\nu)\quad\mbox{for all }h\in\M_\psi(E),~\mu,\nu\in{\cal M}_1^\psi(E),
\end{equation}
because $\M'$ $(\subseteq\M_\psi(E))$ is a generator of $d_{\M}$ by assumption.
Now,
let $(\Q_m)$ be any sequence in ${\cal P}_\psi$ which converges to some $\Q\in{\cal P}_\psi$ w.r.t.\ $d_{\infty,\M}^\psi$. Using the representation (\mbox{\ref{hadamard derivative ex tot costs - II}}),
we obtain for any $m\in\N$
\begin{eqnarray*}
    \lefteqn{\|\dot\Upsilon_{n;\P}^{x_n}(\Q_m-\P) - \dot\Upsilon_{n;\P}^{x_n}(\Q-\P)\|_\infty}\\
    & = & \sup_{\pi\in\Pi}\big|\dot{\cal V}_{n;\P}^{x_n;\pi}(\Q_m-\P) - \dot{\cal V}_{n;\P}^{x_n;\pi}(\Q-\P)\big| \\
    & = & \sup_{\pi=(f_n)_{n=0}^{N-1}\in\Pi}\Big\{\Big|\sum_{k=n}^{N-1}\int_E\cdots\int_E\int_E V_{k+1}^{\P;\pi}(y_{k+1})\,(Q^m_k - P_k)\big((y_k,f_k(y_k)),dy_{k+1}\big) \\
    & & \qquad P_{k-1}\big((y_{k-1},f_{k-1}(y_{k-1})),dy_k\big)\cdots P_n\big((x_n,f_n(x_n)),dy_{n+1}\big) \\
    & & -~ \sum_{k=n}^{N-1}\int_E\cdots\int_E\int_E V_{k+1}^{\P;\pi}(y_{k+1})\,(Q_k - P_k)\big((y_k,f_k(y_k)),dy_{k+1}\big) \\
    & & \qquad P_{k-1}\big((y_{k-1},f_{k-1}(y_{k-1})),dy_k\big)\cdots P_n\big((x_n,f_n(x_n)),dy_{n+1}\big)\Big|\Big\} \\
   & = & \sup_{\pi=(f_n)_{n=0}^{N-1}\in\Pi}\Big\{\Big|\sum_{k=n}^{N-1}\int_E\cdots\int_E\int_E V_{k+1}^{\P;\pi}(y_{k+1})\,(Q^m_k - Q_k)\big((y_k,f_k(y_k)),dy_{k+1}\big) \\
    & & \qquad P_{k-1}\big((y_{k-1},f_{k-1}(y_{k-1})),dy_k\big)\cdots P_n\big((x_n,f_n(x_n)),dy_{n+1}\big)\Big|\Big\}\\
    & \leq & \sum_{k=n}^{N-1}\sup_{\pi=(f_n)_{n=0}^{N-1}\in\Pi}\Big\{\int_E\cdots\int_E\Big|\int_E V_{k+1}^{\P;\pi}(y_{k+1})\,(Q^m_k - Q_k)\big((y_k,f_k(y_k)),dy_{k+1}\big)\Big|\\
    & & \qquad P_{k-1}\big((y_{k-1},f_{k-1}(y_{k-1})),dy_k\big)\cdots P_n\big((x_n,f_n(x_n)),dy_{n+1}\big)\Big\}.
\end{eqnarray*}
It follows from (\mbox{\ref{minkowski - eq10}}) and part (v) of Lemma \mbox{\ref{lemma on integrals of functions of X}} that for any $k=n+1,\ldots,N-1$ and $m\in\N$
\begin{eqnarray} \label{eq: hd continuity - eq 10}
    \lefteqn{\sup_{\pi=(f_n)_{n=0}^{N-1}\in\Pi}\Big\{\int_E\cdots\int_E\Big|\int_E V_{k+1}^{\P;\pi}(y_{k+1})\,(Q^m_k - Q_k)\big((y_k,f_k(y_k)),dy_{k+1}\big)\Big|} \nonumber \\
    & & \quad P_{k-1}\big((y_{k-1},f_{k-1}(y_{k-1})),dy_k\big)\cdots P_n\big((x_n,f_n(x_n)),dy_{n+1}\big)\Big\} \nonumber \\
    & \leq & \sup_{\pi=(f_n)_{n=0}^{N-1}\in\Pi }\Big\{\rho_{\M'}\big(V_{k+1}^{\P;\pi}\big)\cdot\sup_{x\in E}\,\frac{1}{\psi(x)}\,d_{\M}\Big(Q^m_k\big((x,f_k(x)),\,\bullet\,\big),Q_k\big((x,f_k(x)),\,\bullet\,\big)\Big) \nonumber \\
	& & \quad \cdot\int_E\cdots\int_E \psi(y_k)\,P_{k-1}\big((y_{k-1},f_{k-1}(y_{k-1})),dy_k\big)
 	\cdots P_n\big((x_n,f_n(x_n)),dy_{n+1}\big)\Big\} \nonumber \\
    & \le & \sup_{f_k\in F_k}\,\sup_{x\in E}\frac{1}{\psi(x)}\,d_{\M}\Big(Q^m_k\big((x,f_k(x)),\,\bullet\,\big),Q_k\big((x,f_k(x)),\,\bullet\,\big)\Big) \nonumber \\
	& & \quad\cdot\, \sup_{\pi\in\Pi}\rho_{\M'}\big(V_{k+1}^{\P;\pi}\big) \,\cdot\, \sup_{\pi\in\Pi}\ex_{n,x_n}^{x_0,\P;\pi}\big[\psi(X_k)\big] \nonumber \\
	& \le & \sup_{(x,a)\in D_k}\frac{1}{\psi(x)}\,d_{\M}\Big(Q^m_k\big((x,a),\,\bullet\,\big),Q_k\big((x,a),\,\bullet\,\big)\Big) \nonumber \\
	& & \quad\cdot\, \sup_{\pi\in\Pi}\rho_{\M'}\big(V_{k+1}^{\P;\pi}\big) \,\cdot\, \sup_{\pi\in\Pi}\ex_{n,x_n}^{x_0,\P;\pi}\big[\psi(X_k)\big]\nonumber\\
	& \le & d_{\infty,\M}^\psi(\Q_m,\Q)\,\cdot\,\sup_{\pi\in\Pi}\rho_{\M'}\big(V_{k+1}^{\P;\pi}\big) \,\cdot\, \sup_{\pi\in\Pi}\ex_{n,x_n}^{x_0,\P;\pi}\big[\psi(X_k)\big]
\end{eqnarray}
because $V_{k+1}^{\P;\pi}(\cdot)\in\M_\psi(E)$ for any $\pi\in\Pi$ due to Lemma \mbox{\ref{lemma suff cond for stand cond - pre}}  (with ${\cal P}':=\{\P\}$). Similarly, for any $m\in\N$
\begin{eqnarray} \label{eq: hd continuity - eq 20}
    \lefteqn{\sup_{\pi=(f_n)_{n=0}^{N-1}\in\Pi}\Big\{\Big|\int_E V_{n+1}^{\P;\pi}(y_{n+1})\,(Q^m_n - Q_n)\big((x_n,f_n(x_n)),dy_{n+1}\big)\Big|\Big\}} \nonumber \\
    & \leq & \sup_{\pi=(f_n)_{n=0}^{N-1}\in\Pi }\Big\{\rho_{\M'}\big(V_{n+1}^{\P;\pi}\big) \nonumber \\
    & & \qquad \cdot\,\sup_{x\in E}\frac{1}{\psi(x)}\,d_{\M}\Big(Q^m_n\big((x,f_n(x)),\,\bullet\,\big),Q_n\big((x,f_n(x)),\,\bullet\,\big)\Big) \psi(x_n)\Big\}  \nonumber \\
    & \le & \sup_{f_n\in F_n}\,\sup_{x\in E}\frac{1}{\psi(x)}\,d_{\M}\Big(Q^m_n\big((x,f_n(x)),\,\bullet\,\big),Q_n\big((x,f_n(x)),\,\bullet\,\big)\Big) \nonumber \\
    & & \qquad\cdot\, \sup_{\pi\in\Pi}\rho_{\M'}\big(V_{n+1}^{\P;\pi}\big) \,\cdot\, \psi(x_n) \nonumber \\
	& \le & \sup_{(x,a)\in D_n}\frac{1}{\psi(x)}\,d_{\M}\Big(Q^m_n\big((x,a),\,\bullet\,\big),Q_n\big((x,a),\,\bullet\,\big)\Big) \nonumber \\
	& & \qquad\cdot\, \sup_{\pi\in\Pi}\rho_{\M'}\big(V_{n+1}^{\P;\pi}\big) \,\cdot\, \psi(x_n) \nonumber\\
	& \le & d_{\infty,\M}^\psi(\Q_m,\Q) \,\cdot\, \sup_{\pi\in\Pi}\rho_{\M'}\big(V_{n+1}^{\P;\pi}\big) \,\cdot\, \psi(x_n).
\end{eqnarray}
The second factor in the last line of both (\mbox{\ref{eq: hd continuity - eq 10}}) and (\mbox{\ref{eq: hd continuity - eq 20}}) is (independent of $m$ and) finite due to assumption (b). Moreover, the finiteness of the third factor in the last line of formula display (\mbox{\ref{eq: hd continuity - eq 10}}) (which is also independent of $m$) follows from part (v) of Lemma \mbox{\ref{lemma on integrals of functions of X}} and assumption (a).
Therefore, we arrive at $\|\dot\Upsilon_{n;\P}^{x_n}(\Q_m-\P)-\dot\Upsilon_{n;\P}^{x_n}(\Q-\P)\|_\infty\to 0$ as $m\to\infty$.
\end{proof}

\begin{lemma}\label{lemma char bounding function d-bounded sets}
Under the assumptions of Theorem \mbox{\ref{Thm - hadamard upsilon}} let ${\cal K}\subseteq{\cal P}_\psi$ be a bounded set. Then $\psi$ is a bounding function for the family of MDMs $\{(\boldsymbol{X},\boldsymbol{A},\Q,\Pi,\boldsymbol{r}):\Q\in{\cal K}\}$.
\end{lemma}

\begin{proof}
Conditions (a) and (b) of Definition \mbox{\ref{def bounding function}}  (which are independent of any transition function) are satisfied due to assumption (a) of Theorem \mbox{\ref{Thm - hadamard cal v}}.
Thus it suffices to show that condition (c) of Definition \mbox{\ref{def bounding function}}  is satisfied for any bounded set ${\cal K}$  (playing the role of ${\cal P}'$).
For any bounded set ${\cal K}$  we can find some $\P'=(P'_n)_{n=0}^{N-1}\in{\cal P}_\psi$ and $\delta>0$ such that $d_{\infty,\M}^\psi(\Q,\P')\le \delta$ for every $\Q\in{\cal K}$. Letting $K_3>0$ denote the finite constant in condition (c) of Definition \mbox{\ref{def bounding function}}  for the singleton ${\cal P}':=\{\P'\}$, and using (\mbox{\ref{minkowski - eq10}}) as well as assumption (c) of Theorem \mbox{\ref{Thm - hadamard cal v}}, we obtain for any $(x,a)\in D_n$, $\Q=(Q_n)_{n=0}^{N-1}\in{\cal K}$, and $n=0,\ldots,N-1$
\begin{eqnarray*}\label{lemma diff quot hd - eq10}
	\lefteqn{\int_E \psi(y)\,Q_n\big((x,a),dy\big)} \\
	& \le & \Big|\int_E \psi(y)\,(Q_n-P'_n)\big((x,a),dy\big)\Big| + \int_E \psi(y)\,P'_n\big((x,a),dy\big) \\
	& \le & \rho_{\M'}(\psi) \cdot \frac{1}{\psi(x)}\,d_{\M}\Big(Q_n\big((x,a),\,\bullet\,\big),P'_n\big((x,a),\,\bullet\,\big)\Big)\cdot \psi(x) + K_3\psi(x) \\
	& \le &	\rho_{\M'}(\psi)\cdot d_{\infty,\M}^\psi(\Q,\P')\cdot \psi(x) + K_3\psi(x)
	~ \le ~ 	\widetilde{K}_3\psi(x)
\end{eqnarray*}
for $\widetilde{K}_3:=\rho_{\M'}(\psi)\cdot \delta + K_3$,
because $\psi\in\M_\psi(E)$. This completes the proof.
\end{proof}

\begin{lemma}\label{lemma diff quot hd}
Under the assumptions of Theorem \mbox{\ref{Thm - hadamard upsilon}} and for any fixed $x_n\in E$ and $n=0,\ldots,N$,
\begin{equation*}\label{lemma diff quot hd - EQ}
    \begin{aligned}
        \lim_{m\to\infty}\Big\|\frac{\Upsilon_n^{x_n}(\P+\varepsilon_m(\Q-\P))-\Upsilon_n^{x_n}(\P)}{\varepsilon_m} - \dot\Upsilon_{n;\P}^{x_n}(\Q-\P)\Big\|_\infty=0\\
        \quad\mbox{uniformly in $\Q\in{\cal K}$}
    \end{aligned}
\end{equation*}
for every bounded set ${\cal K}\subseteq{\cal P}_\psi$  and every sequence $(\varepsilon_m)\in(0,1]^\N$ with $\varepsilon_m\to 0$.
\end{lemma}

\begin{proof}
Let ${\cal K}\subseteq{\cal P}_\psi$ be a fixed bounded set  and $(\varepsilon_m)\in(0,1]^\N$ such that $\varepsilon_m\to 0$. First of all, note that it can be verified easily by means of assumption (a) of Theorem \mbox{\ref{Thm - hadamard cal v}}  and Lemma \mbox{\ref{lemma char bounding function d-bounded sets}}
that $\Upsilon_n^{x_n}(\P+\varepsilon_m(\Q-\P))\,(=({\cal V}_n^{x_n;\pi}(\P+\varepsilon_m(\Q-\P)))_{\pi\in\Pi})\in\ell^{\infty}(\Pi)$ as well as $\dot\Upsilon_{n;\P}^{x_n}(\Q-\P)\,(=(\dot{\cal V}_{n;\P}^{x_n;\pi}(\Q-\P))_{\pi\in\Pi})\in\ell^{\infty}(\Pi)$ for any $m\in\N$ and $\Q\in{\cal K}$.
In view of Lemma \mbox{\ref{lemma on integrals of functions of X}}, we get for any $m\in\N$, $\Q=(Q_n)_{n=0}^{N-1}\in{\cal K}$, and $\pi=(f_n)_{n=0}^{N-1}\in\Pi$
\begin{eqnarray*}
    \lefteqn{\Big|\frac{{\cal V}_n^{x_n;\pi}(\P+\varepsilon_m(\Q-\P))-{\cal V}_n^{x_n;\pi}(\P)}{\varepsilon_m} - \dot{\cal V}_{n;\P}^{x_n;\pi}(\Q-\P)\Big|}\\
    & = &  \Big|\frac{1}{\varepsilon_m}\sum_{k=n}^{N-1}\Big(\ex_{n,x_n}^{x_0,\P+\varepsilon_m(\Q-\P);\pi}\big[r_k(X_k,f_k(X_k))\big] -\ex_{n,x_n}^{x_0,\P;\pi}\big[r_k(X_k,f_k(X_k))\big]\Big) \\
    & & +\,\frac{1}{\varepsilon_m}\Big(\ex_{n,x_n}^{x_0,\P+\varepsilon_m(\Q-\P);\pi}\big[r_N(X_N)\big] -\ex_{n,x_n}^{x_0,\P;\pi}\big[r_N(X_N)\big]\Big) - \dot{\cal V}_{n;\P}^{x_n;\pi}(\Q-\P)\Big| \\
    & = & \Big|\sum_{k=n+1}^{N-1}\sum_{j=n}^{k-1}\int_E\cdots\int_E r_k(y_k,f_k(y_k))\,P_{k-1}\big((y_{k-1},f_{k-1}(y_{k-1})),dy_k\big)\\
    & & \qquad \cdots (Q_j-P_j)\big((y_j,f_{j}(y_j)),dy_{j+1}\big)\cdots P_n\big((x_n,f_n(x_n)),dy_{n+1}\big)\\
    & & +\,\frac{1}{\varepsilon_m}\sum_{k=n+2}^{N-1}\sum_{\stackrel{J\subseteq\{n,\ldots,k-1\}}{1<|J|\leq k-n}}\varepsilon_m^{|J|}\int_E\int_E\cdots\int_E r_k(y_k,f_k(y_k))\\
    & & \qquad \xi_{k-1,J}^{\Q}\big((y_{k-1},f_{k-1}(y_{k-1})),dy_k\big)\\
    & & \qquad \qquad \cdots\xi_{n+1,J}^{\Q}\big((y_{n+1},f_{n+1}(y_{n+1})),dy_{n+2}\big)\,\xi_{n,J}^{\Q}\big((x_n,f_n(x_n)),dy_{n+1}\big)\\
    & & +\,\sum_{j=n}^{N-1}\int_E\int_E\cdots\int_E r_N(y_N)\,P_{N-1}\big((y_{N-1},f_{N-1}(y_{N-1})),dy_N\big)\\
    & & \qquad \cdots(Q_j-P_j)\big((y_j,f_{j}(y_j)),dy_{j+1}\big)\cdots P_n\big((x_n,f_n(x_n)),dy_{n+1}\big)\\
    & & +\,\frac{1}{\varepsilon_m}\sum_{\stackrel{J\subseteq\{n,\ldots,N-1\}}{1 <|J|\leq N-n}}\varepsilon_m^{|J|}\int_E\int_E\cdots\int_E r_N(y_N)\,\xi_{N-1,J}^{\Q}\big((y_{N-1},f_{N-1}(y_{N-1})),dy_N\big)\\
    & & \qquad \cdots\xi_{n+1,J}^{\Q}\big((y_{n+1},f_{n+1}(y_{n+1})),dy_{n+2}\big)\,\xi_{n,J}^{\Q}\big((x_n,f_n(x_n)),dy_{n+1}\big)\\
	& & -\,\dot{\cal V}_{n;\P}^{x_n;\pi}(\Q-\P) \Big|\\
    & \le & \big|\dot{\cal V}_{n;\P}^{x_n;\pi}(\Q-\P) - \dot{\cal V}_{n;\P}^{x_n;\pi}(\Q-\P)\big| \\
    & & +\,\sum_{k=n+2}^{N-1}\sum_{\stackrel{J\subseteq\{n,\ldots,k-1\}}{1<|J|\leq k-n}}\varepsilon_m^{|J|-1}\Big|\int_E\int_E\cdots\int_E r_k(y_k,f_k(y_k))\\
    & & \qquad\xi_{k-1,J}^{\Q}\big((y_{k-1},f_{k-1}(y_{k-1})),dy_k\big)\\
    & & \qquad\qquad \cdots\xi_{n+1,J}^{\Q}\big((y_{n+1},f_{n+1}(y_{n+1})),dy_{n+2}\big)\,\xi_{n,J}^{\Q}\big((x_n,f_n(x_n)),dy_{n+1}\big)\Big|\\
    & & +\,\sum_{\stackrel{J\subseteq\{n,\ldots,N-1\}}{1 <|J|\leq N-n}}\varepsilon_m^{|J|-1}\Big|\int_E\int_E\cdots\int_E r_N(y_N)\,\xi_{N-1,J}^{\Q}\big((y_{N-1},f_{N-1}(y_{N-1})),dy_N\big)\\
    & & \qquad \cdots\xi_{n+1,J}^{\Q}\big((y_{n+1},f_{n+1}(y_{n+1})),dy_{n+2}\big)\,\xi_{n,J}^{\Q}\big((x_n,f_n(x_n)),dy_{n+1}\big)\Big|\\
    & =: & S_1(\Q,\pi) + S_2(m,\Q,\pi) + S_3(m,\Q,\pi),
\end{eqnarray*}
where $S_1(\Q,\pi)=0$ and $\xi_{j,J}^{\Q}$ is for any subset $J\subseteq\{0,\ldots,N-1\}$ given by
$$
    \xi_{j,J}^{\Q}:=
    \left\{\begin{array}{lll}
        Q_j-P_j & , & j\in J\\
        P_j & , & \mbox{otherwise}
    \end{array}
    \right..
$$
In view of assumption (a) of Theorem \mbox{\ref{Thm - hadamard cal v}} and Lemma \mbox{\ref{lemma char bounding function d-bounded sets}}
there exist finite constants $K_1,K_3,\widetilde{K}_3>0$ such that for every $m\in\N$, $\Q\in{\cal K}$, and $\pi\in\Pi$
\begin{eqnarray*}
    S_2(m,\Q,\pi)
    & \le & \varepsilon_m\cdot \Big\{ K_1\sum_{k=n+2}^{N-1}\sum_{\stackrel{J\subseteq\{n,\ldots,k-1\}}{1<|J|\leq k-n}}\hspace{-1.5mm}\varepsilon_m^{|J|-2}\binom{k-n}{|J|}K_3^{k-n-|J|} \\
    & & \qquad\quad \cdot\,\sum_{l=0}^{|J|}\binom{|J|}{l}K_3^l\widetilde{K}_3^{|J|-l}\psi(x_n) \Big\}.
\end{eqnarray*}
Hence $\lim_{m\to\infty}S_2(m,\Q,\pi)=0$ uniformly in $\Q\in{\cal K}$ and $\pi\in\Pi$. Analogously we find some finite constant $K_2>0$ such that
\begin{eqnarray*}
    S_3(m,\Q,\pi)
    & \le & \varepsilon_m\cdot \Big\{K_2\sum_{\stackrel{J\subseteq\{n,\ldots,N-1\}}{1<|J|\leq N-n}}\hspace{-1.5mm}\varepsilon_m^{|J|-2}\binom{N-n}{|J|}K_3^{N-n-|J|} \\
    & & \qquad\quad \cdot\,\sum_{l=0}^{|J|}\binom{|J|}{l}K_3^l\widetilde{K}_3^{|J|-l}\psi(x_n) \Big\}
\end{eqnarray*}
for every $m\in\N$, $\Q\in{\cal K}$, and $\pi\in\Pi$, and thus $\lim_{m\to\infty}S_3(m,\Q,\pi)=0$ uniformly in $\Q\in{\cal K}$ and $\pi\in\Pi$.
Hence, the assertion follows.
\end{proof}


\subsection{`Hadamard differentiability' of ${\cal V}_n^{x_n}$}\label{Subsec - hadamard cal v}


We intend to show that the value functional ${\cal V}_n^{x_n}$ is `Hadamard differentiable' at $\boldsymbol{P}$ w.r.t.\ $(\M,\psi)$ with `Hadamard derivative' $\dot{\cal V}_{n;\boldsymbol{P}}^{x_n}$ given by (\mbox{\ref{hd value function - 0}}).

The key will be (\mbox{\ref{representation cal v}}) which says that ${\cal V}_n^{x_n}$ can be represented as a composition of the functionals $\Psi$ and $\Upsilon_n^{x_n}$ defined in (\mbox{\ref{def of Upsilon}}). Proposition 1 in \cite{Roemisch2004} ensures that $\Psi$ is Hadamard differentiable (in the sense of \cite{Roemisch2004}) at every $(w(\pi))_{\pi\in\Pi}\in\ell^\infty(\Pi)$ with (possibly nonlinear) Hadamard derivative $\dot\Psi_{(w(\pi))_{\pi\in\Pi}}:\ell^\infty(\Pi)\to\R$ given by
\begin{equation}\label{proof hadamard cal v - eq 10}
  \dot\Psi_{(w(\pi))_{\pi\in\Pi}}\big((z(\pi))_{\pi\in\Pi}\big)
  := \lim_{\delta\searrow 0}\sup_{\pi\in\Pi((w(\pi))_{\pi\in\Pi},\delta)} z(\pi),
\end{equation}
where $\Pi((w(\pi))_{\pi\in\Pi},\delta)$ denotes the set of all $\pi\in\Pi$ for which $\sup_{\sigma\in\Pi}w(\sigma) - \delta\le w(\pi)$. Moreover Theorem \mbox{\ref{Thm - hadamard upsilon}} implies that $\Upsilon_n^{x_n}$ is in particular `Hadamard differentiable' at $\boldsymbol{P}$ w.r.t.\ $(\M,\psi)$ with `Hadamard derivative' $\dot\Upsilon_{n;\P}^{x_n}$ given by (\mbox{\ref{def derivative of Upsilon}}).

In view of (\mbox{\ref{representation cal v}}) and the shape of $\dot\Psi_{(w(\pi))_{\pi\in\Pi}}$ and $\dot\Upsilon_{n;\P}^{x_n}$, `Hadamard differentiability' of ${\cal V}_n^{x_n}$ at $\P$ w.r.t.\ $(\M,\psi)$ with `Hadamard derivative' $\dot{\cal V}_{n;\boldsymbol{P}}^{x_n}$ given by (\mbox{\ref{hd value function - 0}}) (resp.\ (\mbox{\ref{hd value function}}))  can be identified with `Hadamard differentiability' of the map $\Psi\circ\Upsilon_n^{x_n}:{\cal P}_\psi\to\R$ at $\P$ w.r.t.\ $(\M,\psi)$ with `Hadamard derivative' $\dot{(\Psi\circ\Upsilon_n^{x_n})}_{\P}:{\cal P}_\psi^{\P;\pm}\to\R$ given by
\begin{equation}\label{proof hadamard cal v - eq 20}
    \dot{\big(\Psi\circ\Upsilon_n^{x_n}\big)}_{\P}(\Q - \P) := \dot\Psi_{\Upsilon_n^{x_n}(\P)}\circ\dot{\Upsilon}_{n;\P}^{x_n}(\Q - \P).
\end{equation}
Take into account that by (\mbox{\ref{def derivative of Upsilon}}) and (\mbox{\ref{proof hadamard cal v - eq 10}})
\begin{eqnarray*}
	\dot{\big(\Psi\circ\Upsilon_n^{x_n}\big)}_{\P}(\Q - \P)
	 & = & \dot\Psi_{({\cal V}_n^{x_n;\pi}(\P))_{\pi\in\Pi}}\big(\big(\dot{\cal V}_{n;\P}^{x_n;\pi}(\Q-\P)\big)_{\pi\in\Pi}\big) \\
	 & = & \lim_{\delta\searrow 0}\sup_{\pi\in\Pi(\P;\delta)}\dot{\cal V}_{n;\P}^{x_n;\pi}(\Q-\P)
\end{eqnarray*}
for $\Q-\P\in{\cal P}_\psi^{\P;\pm}$, and that, if in addition the set $\Pi(\P)$ is non-empty,
$$
 	\dot{\big(\Psi\circ\Upsilon_n^{x_n}\big)}_{\P}(\Q - \P)
	\,=\, \sup_{\pi\in\Pi(\P)}\dot{\cal V}_{n;\P}^{x_n;\pi}(\Q-\P)
$$
for every $\Q-\P\in{\cal P}_\psi^{\P;\pm}$.

In the remainder of the proof we will show that the composite map $\Psi\circ\Upsilon_n^{x_n}$ is `Hadamard differentiable' at $\P$ w.r.t.\ $(\M,\psi)$ with `Hadamard derivative' $\dot{(\Psi\circ\Upsilon_n^{x_n})}_{\P}$ given by (\mbox{\ref{proof hadamard cal v - eq 20}}). We first note that the map $\dot{(\Psi\circ\Upsilon_n^{x_n})}_{\P}$ is $(\M,\psi)$-continuous by Lemma \mbox{\ref{lemma continuity hd}} and the $(\|\cdot\|_{\infty},|\cdot|)$-continuity of the mapping $(z(\pi))_{\pi\in\Pi}\mapsto\dot\Psi_{\Upsilon_n^{x_n}(\P)}((z(\pi))_{\pi\in\Pi})$. In view of part (ii) of Lemma \mbox{\ref{lemma HD characterization}}, for the desired `Hadamard differentiability' of $\Psi\circ\Upsilon_n^{x_n}$ at $\P$ it therefore suffices to show that
$$
    \lim_{m\to\infty}\Big|\frac{\Psi\circ\Upsilon_n^{x_n}(\P+\varepsilon_m(\Q_m-\P))- \Psi\circ\Upsilon_n^{x_n}(\P)}{\varepsilon_m} - \dot{\big(\Psi\circ\Upsilon_n^{x_n}\big)}_{\P}(\Q-\P)\Big| = 0
$$
for any fixed triplet $(\Q,(\Q_{m}),(\varepsilon_m))\in{\cal P}_\psi\times{\cal P}_\psi^\N\times(0,1]^\N$ with $d_{\infty,\M}^\psi(\Q_m,\Q)\to 0$ and $\varepsilon_m\to 0$. For any such fixed triplet and any $m\in\N$ we have
$$
	\frac{\Psi\circ\Upsilon_n^{x_n}(\P+\varepsilon_m(\Q_m-\P))- \Psi\circ\Upsilon_n^{x_n}(\P)}{\varepsilon_m} = \frac{\Psi(\Upsilon_n^{x_n}(\P) + \varepsilon_m v_m) - \Psi(\Upsilon_n^{x_n}(\P))}{\varepsilon_m}\,,
$$
where $v_m :=  \varepsilon_m^{-1}(\Upsilon_n^{x_n}(\P+\varepsilon_m(\Q_m-\P)) - \Upsilon_n^{x_n}(\P))\,(\in\ell^\infty(\Pi))$. If we set $v:=\dot\Upsilon_{n;\P}^{x_n}(\Q - \P)\,(\in\ell^\infty(\Pi))$, then by Theorem \mbox{\ref{Thm - hadamard upsilon}} and by part (i) of Lemma \mbox{\ref{lemma HD characterization}}
\begin{eqnarray*}
	\lefteqn{\lim_{m\to\infty}\|v_m - v\|_{\infty}}\\
	& = & \lim_{m\rightarrow\infty}\Big\|\frac{\Upsilon_n^{x_n}(\P+\varepsilon_m(\Q_m-\P)) - \Upsilon_n^{x_n}(\P)}{\varepsilon_m} -\dot\Upsilon_{n;\P}^{x_n}(\Q - \P)\Big\|_{\infty} = 0.
\end{eqnarray*}
Thus, since $\Psi$ is Hadamard differentiable at (in particular) $\Upsilon_n^{x_n}(\P)\,(\in\ell^\infty(\Pi))$ (see the discussion above), we obtain
\begin{eqnarray*}
    \lefteqn{\lim_{m\to\infty}\Big|\frac{\Psi\circ\Upsilon_n^{x_n}(\P+\varepsilon_m(\Q_m-\P))- \Psi\circ\Upsilon_n^{x_n}(\P)}{\varepsilon_m} - \dot{\big(\Psi\circ\Upsilon_n^{x_n}\big)}_{\P}(\Q-\P)\Big| } \\
    & = & \lim_{m\to\infty}\Big|\frac{\Psi(\Upsilon_n^{x_n}(\P) + \varepsilon_m v_m) - \Psi(\Upsilon_n^{x_n}(\P))}{\varepsilon_m} - \dot\Psi_{\Upsilon_n^{x_n}(\P)}(v)\Big| \,=\, 0. \hspace{1.5cm}
\end{eqnarray*}
This finishes the proof.
\hfill\proofendsign


\section{Supplement: Proofs of results from Section \mbox{\ref{Sec - Finance Example}} }\label{Sec - Proof of results from Section 5}


\subsection{Proof of Lemma \mbox{\ref{Ex fin - lemma existence solution reduced opt prob}} }\label{Subsec - Proof of Lem ex solution red opt prob}


Let $\P\in{\cal P}_\psi$ and $n=0,\ldots,N-1$ be fixed. Define a map $\mathfrak{f}^{\P}_n:\R_{\ge 0}\times[0,1]\to\R_{\ge 0}$ through
\begin{equation}\label{proof lem ex solution red opt prob - eq10}
	\mathfrak{f}^{\P}_n(y,\gamma) \,:=\, u_\alpha(1+\gamma(y/\mathfrak{r}_{n+1} - 1)).
\end{equation}
Note that $\mathfrak{f}^{\P}_n(\,\cdot\,,\gamma)$ is clearly Borel measurable for any $\gamma\in[0,1]$, and it is easily seen that
\begin{equation}\label{proof lem ex solution red opt prob - eq20}
	|\mathfrak{f}^{\P}_n(y,\gamma)| \,=\, u_\alpha\big((1-\gamma)+\gamma(y/\mathfrak{r}_{n+1})\big) \,\le\, u_\alpha(1 + y)
\end{equation}
for every $y\in\R_{\ge 0}$ and $\gamma\in[0,1]$. Therefore, the function $\mathfrak{f}^{\P}_n$ is absolutely dominated by the Borel measurable function $\mathfrak{h}:\R_{\ge 0}\to\R_{\ge 0}$ given by $\mathfrak{h}(y):=u_\alpha(1+y)$.
Set $\overline{\mathfrak{m}}_{\P}:=\max_{k=0,\ldots,N-1}\int_{\R_{\ge 0}}u_\alpha\,d\mathfrak{m}_{k+1}^{\P}$ and note that $\overline{\mathfrak{m}}_{\P}\in\R_{>0}$. Since $\mathfrak{h}$ satisfies
\begin{equation}\label{proof lem ex solution red opt prob - eq30}
	\int_{\R_{\ge 0}}\mathfrak{h}(y)\,\mathfrak{m}_{n+1}^{\P}(dy)
	\,\le\, 1 + \int_{\R_{\ge 0}}u_\alpha(y)\,\mathfrak{m}_{n+1}^{\P}(dy)
	\,\le\, 1 + \overline{\mathfrak{m}}_{\P}
	\,<\, \infty
\end{equation}
(i.e.\ $\mathfrak{h}$ is $\mathfrak{m}_{n+1}^{\P}$-integrable) and $\mathfrak{f}^{\P}_n(y,\,\cdot\,)$ is continuous on $[0,1]$ for any $y\in\R_{\ge 0}$, we may apply the continuity lemma (see, e.g., \cite[Lemma 16.1]{Bauer2001}) to obtain that the mapping $\mathfrak{F}_n^{\P}:[0,1]\rightarrow\R_{>0}$ given by $\mathfrak{F}_n^{\P}(\gamma):=\int_{\R_{\ge 0}}\mathfrak{f}^{\P}_n(y,\gamma)\,\mathfrak{m}_{n+1}^{\P}(dy)$ is continuous. Along with the compactness of the set $[0,1]$ this ensures the existence of a solution $\gamma_n^{\P}\in[0,1]$ to the optimization problem (\mbox{\ref{Ex fin - one-stage max problem}}). Moreover it can be verified easily by means of part (c) of Assumption {\bf (FM)}  that $\mathfrak{F}_n^{\P}$ is strictly concave; take into account that $\int_{\R_{\ge 0}}\mathfrak{f}^{\P}_n(y,\gamma)\,\mathfrak{m}_{n+1}^{\P}(dy)$ can be seen for any $\gamma\in[0,1]$ as the expectation of $u_\alpha(1+\gamma(\mathfrak{R}_{n+1}/\mathfrak{r}_{n+1} - 1))$ under $\pr$. This implies that the solution $\gamma_n^{\P}$ is even unique.
\hfill\proofendsign


\subsection{Proof of Theorem \mbox{\ref{Ex fin - Thm opt trading strat}} }\label{Subsec - Finance Example - proof Thm opt trad strat}


(i): We intend to apply Theorem \mbox{\ref{Thm - existence opt strategy}} (see Section \mbox{\ref{Sec - Existence of optimal strategies}}).
Let $\M_n^{\P}:=\M'$ and  $F_n':=F'$ for any $n=0,\ldots,N-1$, where
\begin{eqnarray}\label{proof Theorem opt trad strat - eq10}
	\M' & := & \big\{h\in\R^{\R_{\ge 0}}: h(x) = \vartheta\, u_\alpha(x/\kappa),\,x\in\R_{\ge 0},\mbox{ for some }\vartheta\in\R_{>0},\kappa\in\R_{\ge 1}\big\},\nonumber\\
	F' & := &  \big\{f\in F: f(x)=\gamma\,x,\,x\in\R_{\ge 0},\ \mbox{for some }\gamma\in[0,1]\big\}
\end{eqnarray}
with $F:=F_n$ (recall that $F_n=\mathbb{F}_n$ and that $\mathbb{F}_n$ is independent of $n$). It is easily seen that $\M_n^{\P}=\M'$ is a subset of $\M_n^{\P}(\R_{\ge 0})$  for any $n=0,\ldots,N-1$, where $\M_n^{\P}(\R_{\ge 0})$ is defined as in (\mbox{\ref{assumption integrability appendix}}) in Section \mbox{\ref{Sec - Existence of optimal strategies}}. Moreover we obviously have $F_n'=F'\subseteq F_n$ for any $n=0,\ldots,N-1$.

Below we will show that conditions (a)--(c) of Theorem \mbox{\ref{Thm - existence opt strategy}} are met. Thus we may apply part (i) of Theorem \mbox{\ref{Thm - existence opt strategy}} (Bellman equation) to obtain part (i) of Theorem \mbox{\ref{Ex fin - Thm opt trading strat}}. In fact, for $n=N$ we have
$$
	V_N^{\P}(x_N) \,=\, r_N(x_N) \,=\, \mathfrak{v}_N^{\P}\, u_\alpha(x_N/B_N)
$$
for any $x_N\in\R_{\ge 0}$, where $\mathfrak{v}_N^{\P}:=1$. Now, suppose that the assertion holds for $k\in\{n+1,\ldots,N\}$. Then, using again part (i) of Theorem \mbox{\ref{Thm - existence opt strategy}}, we have for any $x_n\in\R_{\ge 0}$
\begin{eqnarray}\label{proof Theorem opt trad strat - eq20}
	V_n^{\P}(x_n)
	& = & {\cal T}_{n}^{\P}V_{n+1}^{\P}(x_{n}) \,=\, \sup_{f_n\in F_n}\,{\cal T}_{n,f_n}^{\P}V_{n+1}^{\P}(x_{n})\nonumber\\
    & = & \sup_{f_n\in F_n}\,\int_{\R_{\ge 0}} V_{n+1}^{\P}(y)\,P_n\big((x_{n},f_n(x_{n})),dy\big) \nonumber\\
    & = & \sup_{f_n\in F_n}\,\int_{\R_{\ge 0}} \mathfrak{v}_{n+1}^{\P}\,u_{\alpha}(y/B_{n+1})\,P_n\big((x_{n},f_n(x_{n})),dy\big) \nonumber\\
    & = & \mathfrak{v}_{n+1}^{\P}\sup_{f_n\in F_n}\,\int_{\R_{\ge 0}} u_{\alpha}\Big(\frac{\mathfrak{r}_{n+1}x_n + f_n(x_n)(y-\mathfrak{r}_{n+1})}{\mathfrak{r}_{n+1}B_{n}}\Big)\,\mathfrak{m}_{n+1}^{\P}(dy). \hspace{1cm}
\end{eqnarray}
For $x_n=0$ we have $f_n(x_n)=0$ for any $f_n\in F_n$ and therefore (in view of (\mbox{\ref{proof Theorem opt trad strat - eq20}})) $V_n^{\P}(x_n)=0$. For $x_n\in\R_{>0}$ we obtain from (\mbox{\ref{proof Theorem opt trad strat - eq20}})
\begin{eqnarray}\label{proof Theorem opt trad strat - eq30}
	V_n^{\P}(x_n)
    & = & \mathfrak{v}_{n+1}^{\P}\,u_{\alpha}(x_n/B_n)\,\sup_{f_n\in F_n}\,\int_{\R_{\ge 0}} u_{\alpha}\Big(1+\frac{f_n(x_n)}{x_{n}}\Big(\frac{y}{\mathfrak{r}_{n+1}}-1\Big)\Big)\,\mathfrak{m}_{n+1}^{\P}(dy) \nonumber\\
    & = & \mathfrak{v}_{n+1}^{\P}\,u_\alpha(x_n/B_n)\,\sup_{\gamma\in[0,1]}\,\int_{\R_{\ge 0}} u_{\alpha}\Big(1+\gamma\Big(\frac{y}{\mathfrak{r}_{n+1}}-1\Big)\Big)\,\mathfrak{m}_{n+1}^{\P}(dy) \nonumber\\
	& = & \mathfrak{v}_{n+1}^{\P}\,u_\alpha(x_n/B_n)\,v_n^{\P}
	\,=\, \mathfrak{v}_n^{\P}\,u_\alpha(x_n/B_n),
\end{eqnarray}
where we used for the second ``$=$'' that the value of $f_n(x_n)$ ranges over the interval $[0,x_n]$ when $f_n$ ranges over $F_n$; we can then indeed replace $f_n(x_n)$ by $\gamma x_{n}$ when ``$\sup_{f_n\in F_n}$'' is replaced by ``$\sup_{\gamma\in[0,1]}$''. For the last step we employed $\mathfrak{v}_n^{\P}= \mathfrak{v}_{n+1}^{\P}v_n^{\P}$.
Hence we have verified the representation of the value function asserted in part (i). It remains to show that conditions (a)--(c) of Theorem \mbox{\ref{Thm - existence opt strategy}} (in Section \mbox{\ref{Sec - Existence of optimal strategies}}) are indeed satisfied.

(a): In view of (\mbox{\ref{Ex fin - terminal reward function}})  we obtain $r_N\in\M'$ by choosing $\vartheta:=1$ ($\in\R_{>0}$) and $\kappa:=B_N$ ($\in\R_{\ge 1}$). In particular, $r_N\in\M_{N-1}^{\P}$.

(b): Let $n\in\{1,\ldots,N-1\}$ and $h\in\M_n^{\P}=\M'$, i.e.\ $h(x) = \vartheta\,u_\alpha(x/\kappa)$, $x\in\R_{\ge 0}$, for some $\vartheta\in\R_{>0}$ and $\kappa\in\R_{\ge 1}$. Then as in (\mbox{\ref{proof Theorem opt trad strat - eq20}}) we obtain for any $x\in\R_{\ge 0}$
\begin{eqnarray}\label{proof Theorem opt trad strat - eq40}
	{\cal T}_n^{\P}h(x)
    & = & \sup_{f_n\in F_n}\,{\cal T}_{n,f_n}^{\P}h(x) \nonumber \\
    & = & \vartheta\,\sup_{f_n\in F_n}\,\int_{\R_{\ge 0}} u_{\alpha}\Big(\frac{\mathfrak{r}_{n+1}x + f_n(x)(y-\mathfrak{r}_{n+1})}{\kappa}\Big)\,\mathfrak{m}_{n+1}^{\P}(dy).
\end{eqnarray}
For $x=0$ we have $f_n(x)=0$ for any $f_n\in F_n$ and therefore (in view of (\mbox{\ref{proof Theorem opt trad strat - eq40}})) ${\cal T}_n^{\P}h(x)=0$. For $x\in\R_{>0}$ we obtain from (\mbox{\ref{proof Theorem opt trad strat - eq40}}) (analogously to (\mbox{\ref{proof Theorem opt trad strat - eq30}}))
\begin{eqnarray}\label{proof Theorem opt trad strat - eq45}
	{\cal T}_n^{\P}h(x)
    & = & \vartheta\,\mathfrak{r}_{n+1}^\alpha\,u_{\alpha}(x/\kappa)\,\sup_{f_n\in F_n}\,\int_{\R_{\ge 0}}u_{\alpha}\Big(1+\frac{f_n(x)}{x}\Big(\frac{y}{\mathfrak{r}_{n+1}}-1\Big)\Big)\,\mathfrak{m}_{n+1}^{\P}(dy) \nonumber\\
	& = & \vartheta\,\mathfrak{r}_{n+1}^\alpha\,u_\alpha(x/\kappa)\,\sup_{\gamma\in[0,1]}\,\int_{\R_{\ge 0}}u_\alpha\Big(1+\gamma\Big(\frac{y}{\mathfrak{r}_{n+1}} - 1\Big)\Big)\,\mathfrak{m}_{n+1}^{\P}(dy) \nonumber\\
	& = & \vartheta\,\mathfrak{r}_{n+1}^\alpha\,u_\alpha(x/\kappa)\,v_n^{\P}
	\,=\, \widetilde{\vartheta}\, u_\alpha(x/\kappa),
\end{eqnarray}
where $\widetilde{\vartheta}:=\vartheta\mathfrak{r}_{n+1}^\alpha v_n^{\P}\in\R_{>0}$ is finite due to (\mbox{\ref{proof lem ex solution red opt prob - eq10}})--(\mbox{\ref{proof lem ex solution red opt prob - eq30}}). Altogether we have shown that ${\cal T}_n^{\P}h\in\M'$. In particular, ${\cal T}_n^{\P} h\in\M_{n-1}^{\P}$.

(c): Let $n\in\{0,\ldots,N-1\}$ and $h\in\M_n^{\P}=\M'$ (with corresponding $\vartheta$ and $\kappa$ as in (b)). Moreover, let $f_n^{\P}$ be the map as defined in (\mbox{\ref{Ex fin - optimal trading strategy}}), and note that $f_n^{\P}\in F_n$. Then, similarly to (\mbox{\ref{proof Theorem opt trad strat - eq40}}), we have for any $x\in\R_{\ge 0}$ and $f_n\in F_n$
$$
	{\cal T}_{n,f_n}^{\P}h(x) \,=\, \vartheta\,\int_{\R_{\ge 0}} u_{\alpha}\Big(\frac{\mathfrak{r}_{n+1}x + f_n(x)(y-\mathfrak{r}_{n+1})}{\kappa}\Big)\,\mathfrak{m}_{n+1}^{\P}(dy).
$$
For $x=0$ we obviously have ${\cal T}_{n,f_n}^{\P}h(x)=0$ and thus ${\cal T}_{n,f_n^{\P}}^{\P}h(x)={\cal T}_n^{\P}h(x)$.
For $x\in\R_{>0}$ we have similarly to (\mbox{\ref{proof Theorem opt trad strat - eq45}}) that for any $f_n\in F_n$
$$
	{\cal T}_{n,f_n}^{\P}h(x) \,=\, \vartheta\,\mathfrak{r}_{n+1}^\alpha\,u_{\alpha}(x/\kappa)\int_{\R_{\ge 0}}u_{\alpha}\Big(1+\frac{f_n(x)}{x}\Big(\frac{y}{\mathfrak{r}_{n+1}}-1\Big)\Big)\,\mathfrak{m}_{n+1}^{\P}(dy).
$$
By Lemma \mbox{\ref{Ex fin - lemma existence solution reduced opt prob}}, the map $\gamma\mapsto\int_{\R_{\ge 0}}u_{\alpha}(1+\gamma(y/\mathfrak{r}_{n+1}-1))\,\mathfrak{m}_{n+1}^{\P}(dy)$ has exactly one maximal point, $\gamma_n^{\P}$, in $[0,1]$. Thus, since the second line in (\mbox{\ref{proof Theorem opt trad strat - eq45}}) coincides with ${\cal T}_n^{\P}h(x)$, we obtain ${\cal T}_{n,f_n^{\P}}^{\P}h(x)={\cal T}_n^{\P}h(x)$ also for any $x\in\R_{>0}$.
Therefore the map $f_n^{\P}$ provides a maximizer $f_n^{\P}\in F_n$ of $h$ with $f_n^{\P}\in F'_n$.

(ii): In the proof of (i) we have seen that the assumptions of Theorem \mbox{\ref{Thm - existence opt strategy}} are fulfilled. Thus, part (i) of this theorem gives $V_{n+1}^{\P}\in\M_n^{\P}$ for any $n=0,\ldots,N-1$. In particular, the above elaborations under (c) show that for any $n=0,\ldots,N-1$ the map $f_n^{\P}$ defined by (\mbox{\ref{Ex fin - optimal trading strategy}})  provides a maximizer $f_n^{\P}\in F_n$ of $V_{n+1}^{\P}$ with $f_n^{\P}\in F'_n$. Hence, part (iii) of Theorem \mbox{\ref{Thm - existence opt strategy}} ensures that the strategy $\pi^{\P}:=(f_n^{\P})_{n=0}^{N-1}\in\Pi_{\rm lin}$ forms an optimal trading strategy w.r.t.\ $\P$.

For the second part of the assertion we assume that there exists another optimal trading strategy $\widetilde{\pi}^{\P}$ w.r.t.\ $\P$ with $\widetilde{\pi}^{\P}\in\Pi_{\rm lin}$. Then, by definition of $\Pi_{\rm lin}$, there exists $\widetilde{\boldsymbol{\gamma}}^{\P}=(\widetilde{\gamma}_n^{\P})_{n=0}^{N-1}\in[0,1]^N$ such that $\widetilde{\pi}^{\P}=\pi_{\widetilde{\boldsymbol{\gamma}}^{\P}}:=(f_n^{\widetilde{\boldsymbol{\gamma}}^{\P}})_{n=0}^{N-1}$. 
In particular, we have $V_0^{\P}(x_0) = V_0^{\P;\pi_{\widetilde{\boldsymbol{\gamma}}^{\P}}}(x_0)$ for any $x_0\in\R_{\ge 0}$. Along with part (i) of this theorem and Lemma \mbox{\ref{lemma Ex-fin repr exp reward}} (see Subsection \mbox{\ref{Subsec - Finance Example - auxiliary lemmas proof Thm hd}}), this implies
$
	\mathfrak{v}_0^{\P}\, u_\alpha(x_0/B_0) = \mathfrak{v}_0^{\P;\pi_{\widetilde{\boldsymbol{\gamma}}^{\P}}}\, u_\alpha(x_0/B_0)
$
for every $x_0\in\R_{>0}$ and thus $\mathfrak{v}_0^{\P} = \mathfrak{v}_0^{\P;\pi_{\widetilde{\boldsymbol{\gamma}}^{\P}}}$, i.e.\
\begin{equation}\label{proof Theorem opt trad strat - eq50}
	\prod_{k=0}^{N-1}v_k^{\P} \,=\, \prod_{k=0}^{N-1}v_k^{\P;\widetilde{\gamma}_k^{\P}}.
\end{equation}
Below we will show that (\mbox{\ref{proof Theorem opt trad strat - eq50}}) implies
\begin{equation}\label{proof Theorem opt trad strat - eq60}
	v_n^{\P} \,=\, v_n^{\P;\widetilde{\gamma}_n^{\P}} \quad\mbox{for all }n=0,\ldots,N-1.
\end{equation}
Then it follows from (\mbox{\ref{proof Theorem opt trad strat - eq60}}) that for any $n=0,\ldots,N-1$ the fraction $\widetilde{\gamma}_n^{\P}\in[0,1]$ is a solution to the optimization problem (\mbox{\ref{Ex fin - one-stage max problem}}). However, according to Lemma \mbox{\ref{Ex fin - lemma existence solution reduced opt prob}}, this optimization problem has exactly one solution, $\gamma_n^{\P}$, in $[0,1]$. Hence $\widetilde{\gamma}_n^{\P}=\gamma_n^{\P}$ for any $n=0,\ldots,N-1$ and we arrive at $\widetilde{\pi}^{\P} =\pi^{\P}$ which implies that $\pi^{\P}$ is unique among all $\pi\in\Pi_{\rm lin}(\P)$.

It remains to show that (\mbox{\ref{proof Theorem opt trad strat - eq50}}) implies (\mbox{\ref{proof Theorem opt trad strat - eq60}}). Assume by way of contradiction that (\mbox{\ref{proof Theorem opt trad strat - eq60}}) does {\em not} hold, i.e.\ there exists $n\in\{0,\ldots,N-1\}$ such that $v_n^{\P} \neq v_n^{\P;\widetilde{\gamma}_n^{\P}}$. Then
$$
	v_n^{\P} \,=\, \sup_{\gamma\in[0,1]}v_n^{\P;\gamma} \,>\, v_n^{\P;\widetilde{\gamma}_n^{\P}}
$$
because the reverse inequality would lead to a contradiction of the maximality of $v_n^{\P}$. By assumption (\mbox{\ref{proof Theorem opt trad strat - eq50}}), this implies that there exists $k\in\{0,\ldots,N-1\}$ with $k\neq n$ such that
$$
	v_k^{\P} \,=\, \sup_{\gamma\in[0,1]}v_k^{\P;\gamma} \,<\, v_k^{\P;\widetilde{\gamma}_k^{\P}}.
$$
This, however, contradicts the maximality of $v_k^{\P}$. Hence (\mbox{\ref{proof Theorem opt trad strat - eq50}}) indeed implies (\mbox{\ref{proof Theorem opt trad strat - eq60}}).
\hfill\proofendsign


\subsection{Proof of Theorem \mbox{\ref{Ex fin - Thm Hadamard}} }\label{Subsec - Finance Example - proof Thm hd}


The following Lemmas \mbox{\ref{lemma Ex-fin repr exp reward}}--\mbox{\ref{lemma Ex-fin recursion hd exp reward}} involve the map $V_n^{\P;\pi}$ given by (\mbox{\ref{exp tot reward}}). In the specific setting of Subsection \mbox{\ref{Subsec - Finance Example - Embedding FM into MDM}}  this map admits the representations
\begin{equation}\label{proof Ex finance Theorem hd - eq10}
V_n^{\P;\pi}(x_n) = \ex_{n,x_n}^{x_0,\P;\pi}[r_N(X_N)]
\end{equation}
for any $x_n\in\R_{\ge 0}$, $\P\in{\cal P}_\psi$, $\pi\in\Pi$, and $n=0,\ldots,N$.


\subsubsection{Auxiliary lemmas}\label{Subsec - Finance Example - auxiliary lemmas proof Thm hd}


\begin{lemma}\label{lemma Ex-fin repr exp reward}
Let $\P=(P_n)_{n=0}^{N-1}\in{\cal P}_\psi$ and $\boldsymbol{\gamma}=(\gamma_n)_{n=0}^{N-1}\in[0,1]^N$ be fixed. Then the map $V_n^{\P;\pi_{\boldsymbol{\gamma}}}$ given by (\mbox{\ref{proof Ex finance Theorem hd - eq10}}) admits the representation
\begin{equation}\label{lemma Ex-fin repr exp reward - eq10}
	V_n^{\P;\pi_{\boldsymbol{\gamma}}}(x_n) = \mathfrak{v}_n^{\P;\pi_{\boldsymbol{\gamma}}}\, u_\alpha(x_n/B_n)
\end{equation}
for any $x_n\in\R_{\ge 0}$ and $n=0,\ldots,N$, where $\mathfrak{v}_n^{\P;\pi_{\boldsymbol{\gamma}}} := \prod_{k=n}^{N-1}v_k^{\P;\boldsymbol{\gamma}} = \prod_{k=n}^{N-1}v_k^{\P;\gamma_k}$.
\end{lemma}

\begin{proof}
We prove the assertion in (\mbox{\ref{lemma Ex-fin repr exp reward - eq10}}) by (backward) induction on $n$. For $n=N$ we obtain by means of (\mbox{\ref{proof Ex finance Theorem hd - eq10}}), part (iii) of Lemma \mbox{\ref{lemma on integrals of functions of X}}, and (\mbox{\ref{Ex fin - terminal reward function}})
$$
	V_N^{\P;\pi_{\boldsymbol{\gamma}}}(x_N) \,=\, r_N(x_N) \,=\, \mathfrak{v}_N^{\P;\pi_{\boldsymbol{\gamma}}}\, u_\alpha(x_N/B_N)
$$
for any $x_N\in\R_{\ge 0}$, where $\mathfrak{v}_N^{\P;\pi_{\boldsymbol{\gamma}}}:=1$. Now, suppose that the assertion in (\mbox{\ref{lemma Ex-fin repr exp reward - eq10}}) holds for $k\in\{n+1,\ldots,N\}$. Note that $V_{n+1}^{\P;\pi_{\boldsymbol{\gamma}}}(\cdot)\in\M'$ (with $\M'$ defined as in (\mbox{\ref{proof Theorem opt trad strat - eq10}})) by choosing $\vartheta:=\mathfrak{v}_{n+1}^{\P;\pi_{\boldsymbol{\gamma}}}$ ($\in\R_{>0}$) as well as $\kappa:=B_{n+1}$ ($\in\R_{\ge 1}$), and that it can be verified easily that $\M'$ is a subset of $\M_n^{\P}(\R_{\ge 0})$, where $\M_n^{\P}(\R_{\ge 0})$ is defined as in (\mbox{\ref{assumption integrability appendix}}) in the Appendix \mbox{\ref{Sec - Existence of optimal strategies}}.
Then, in view of part (i) of Proposition \mbox{\ref{proposition reward iteration}}, for any $x_n\in\R_{\ge 0}$ we get
\begin{eqnarray}\label{proof lemma Ex-fin repr exp reward - eq10}
	V_n^{\P;\pi_{\boldsymbol{\gamma}}}(x_n)
	& = & {\cal T}_{n,f_n^{\boldsymbol{\gamma}}}^{\P}V_{n+1}^{\P;\pi_{\boldsymbol{\gamma}}}(x_n) \nonumber \\
	& = & \int_{\R_{\ge 0}}V_{n+1}^{\P;\pi_{\boldsymbol{\gamma}}}(y)\,P_n\big((x_n,f_n^{\boldsymbol{\gamma}}(x_n)),dy\big) \nonumber \\
	& = & \int_{\R_{\ge 0}}\mathfrak{v}_{n+1}^{\P;\pi_{\boldsymbol{\gamma}}}\,u_\alpha(y/B_{n+1})\,P_n\big((x_n,f_n^{\boldsymbol{\gamma}}(x_n)),dy\big) \nonumber \\
	& = & \mathfrak{v}_{n+1}^{\P;\pi_{\boldsymbol{\gamma}}}\,\int_{\R_{\ge 0}}u_\alpha\Big(\frac{\mathfrak{r}_{n+1}x_n + f_n^{\boldsymbol{\gamma}}(x_n)(y - \mathfrak{r}_{n+1})}{\mathfrak{r}_{n+1}B_n}\Big)\,\mathfrak{m}_{n+1}^{\P}(dy) \nonumber \\
	& = & \mathfrak{v}_{n+1}^{\P;\pi_{\boldsymbol{\gamma}}}\,\int_{\R_{\ge 0}}u_\alpha\Big(\frac{\mathfrak{r}_{n+1}x_n + \gamma_n\,x_n(y - \mathfrak{r}_{n+1})}{\mathfrak{r}_{n+1}B_n}\Big)\,\mathfrak{m}_{n+1}^{\P}(dy) \nonumber \\
	& = & \mathfrak{v}_{n+1}^{\P;\pi_{\boldsymbol{\gamma}}}\,u_\alpha(x_n/B_n)\,\int_{\R_{\ge 0}}u_\alpha\Big(1 + \gamma_n\Big(\frac{y}{\mathfrak{r}_{n+1}} - 1\Big)\Big)\,\mathfrak{m}_{n+1}^{\P}(dy) \nonumber \\
	& = & \mathfrak{v}_{n+1}^{\P;\pi_{\boldsymbol{\gamma}}}\,u_\alpha(x_n/B_n)\,v_n^{\P;\boldsymbol{\gamma}}
	\,=\, \mathfrak{v}_n^{\P;\pi_{\boldsymbol{\gamma}}}\, u_\alpha(x_n/B_n),
\end{eqnarray}
where we used for the fifth ``$=$'' the definition of the map $f_n^{\boldsymbol{\gamma}}$ in (\mbox{\ref{Ex fin - linear decision rules}}). For the last step we employed $\mathfrak{v}_n^{\P;\pi_{\boldsymbol{\gamma}}}= \mathfrak{v}_{n+1}^{\P;\pi_{\boldsymbol{\gamma}}}v_n^{\P;\boldsymbol{\gamma}}$.
Thus we have verified the representation of the map $V_n^{\P;\pi_{\boldsymbol{\gamma}}}$ in (\mbox{\ref{lemma Ex-fin repr exp reward - eq10}}).
\end{proof}

\begin{lemma}\label{lemma Ex-fin verification assumptions Thm hd}
Let $\M_{\rm{\scriptsize{H\ddot{o}l}},\alpha}$ be defined as in Example \mbox{\ref{examples prob metric - hoelder}}, and let $\psi$ be the gauge function from (\mbox{\ref{Ex fin - gauge function}}). Then the following three assertions hold.
\vspace{-1mm}
\begin{enumerate}
	\item[{\rm (i)}] $\psi$ is a bounding function for the MDM $(\boldsymbol{X},\boldsymbol{A},\Q,\Pi,\boldsymbol{r})$ for any $\Q\in{\cal P}_\psi$.

	\item[{\rm (ii)}] For any fixed $\P\in{\cal P}_\psi$ we have $\sup_{\pi\in\Pi_{\rm lin}}\rho_{\M_{\rm{\scriptsize{H\ddot{o}l}},\alpha}}(V_n^{\P;\pi})<\infty$ for every $n=1,\ldots,N$.
	
	\item[{\rm (iii)}] $\rho_{\M_{\rm{\scriptsize{H\ddot{o}l}},\alpha}}(\psi)<\infty$.
\end{enumerate}
\end{lemma}

\begin{proof}
(i): Fix $\Q=(Q_n)_{n=0}^{N-1}\in{\cal P}_\psi$. Since $r_n\equiv 0$ for any $n=0,\ldots,N-1$, there exists a finite constant $K_1>0$ such that
$$
	|r_n(x,a)| \,\le\, K_1 \,\le\, K_1\big(1+u_\alpha(x)\big) \,=\, K_1\psi(x)
$$
for every $(x,a)\in D_n$ and $n=0,\ldots,N-1$.

Moreover, in view of (\mbox{\ref{Ex fin - terminal reward function}}), we can find some finite constant $K_2>0$ such that
$$
	|r_N(x)| \,=\, \big(1/u_\alpha(B_N)\big)\,u_\alpha(x) \,\le\, u_\alpha(x) \,\le\, K_2\psi(x)
$$
for every $x\in\R_{\ge 0}$ and $n=0,\ldots,N-1$.

Next, set $\overline{\mathfrak{r}}:=\max_{k=0,\ldots,N-1}\mathfrak{r}_{k+1}$ and note that $\overline{\mathfrak{r}}\in\R_{\ge 1}$. Using displays (\mbox{\ref{Ex fin - gauge function}})--(\mbox{\ref{Ex fin - set trans func psi}}), we find some finite constant $K_3>0$ (depending on $\Q$) such that
\begin{eqnarray*}
	\int_{\R_{\ge 0}}\psi(y)\,Q_n\big((x,a),dy\big)
	& = & 1 + \int_{\R_{\ge 0}} u_\alpha\big(\mathfrak{r}_{n+1}x + a(y - \mathfrak{r}_{n+1})\big)\,\mathfrak{m}^{\Q}_{n+1}(dy) \\
	& = & 1 + (\mathfrak{r}_{n+1})^\alpha\,\int_{\R_{\ge 0}}u_\alpha\Big(x + a\Big(\frac{y}{\mathfrak{r}_{n+1}} - 1\Big)\Big)\,\mathfrak{m}^{\Q}_{n+1}(dy) \\
	& \le & 1 + \overline{\mathfrak{r}}^\alpha\,u_\alpha(x)\, \int_{\R_{\ge 0}}u_\alpha(1 + y)\,\mathfrak{m}^{\Q}_{n+1}(dy) \\
	& \le & 1 + \overline{\mathfrak{r}}^\alpha\,u_\alpha(x)\, \Big(1 + \int_{\R_{\ge 0}}u_\alpha(y)\,\mathfrak{m}^{\Q}_{n+1}(dy)\Big) \\
	& \le & 1 + \overline{\mathfrak{r}}^\alpha\,u_\alpha(x)\, (1 + \overline{\mathfrak{m}}_{\Q})
	\,\le\, K_3\psi(x)
\end{eqnarray*}
for every $(x,a)\in D_n$ and $n=0,\ldots,N-1$, where $\overline{\mathfrak{m}}_{\Q}$ is defined as in Subsection \mbox{\ref{Subsec - Proof of Lem ex solution red opt prob}}. Take into account that $\alpha\in(0,1)$ introduced in (\mbox{\ref{Ex fin - def power utility}})  is fixed.
Consequently, conditions (a)--(c) of Definition \mbox{\ref{def bounding function}}  are satisfied for ${\cal P}':=\{\Q\}$.

(ii): Fix $n\in\{1,\ldots,N\}$. Since any $\boldsymbol{\gamma}=(\gamma_n)_{n=0}^{N-1}\in[0,1]^N$ induces a linear trading strategy $\pi=\pi_{\boldsymbol{\gamma}}:=(f_n^{\boldsymbol{\gamma}})_{n=0}^{N-1}\in\Pi_{\rm lin}$ through (\mbox{\ref{Ex fin - linear decision rules}}), it suffices in view of Example \mbox{\ref{examples prob metric - hoelder}}  to show that
\begin{equation}\label{proof lemma Ex-fin verification assumptions Thm hd - eq10}
	\sup_{\boldsymbol{\gamma}=(\gamma_n)_{n=0}^{N-1}\in[0,1]^N}\|V_n^{\P;\pi_{\boldsymbol{\gamma}}}\|_{\rm{\scriptsize{H\ddot{o}l}},\alpha} < \infty.
\end{equation}
First of all, it is easily seen that the terminal reward function $r_N$ given by (\mbox{\ref{Ex fin - terminal reward function}})  is contained in $\M_{\rm{\scriptsize{H\ddot{o}l}},\alpha}$. Thus $\|r_N\|_{\rm{\scriptsize{H\ddot{o}l}},\alpha}\le 1$.
Moreover, in view of Lemma \mbox{\ref{lemma Ex-fin repr exp reward}} and (\mbox{\ref{Ex fin - terminal reward function}}), we have $V_n^{\P;\pi_{\boldsymbol{\gamma}}}(\cdot) = \mathfrak{v}_n^{\P;\pi_{\boldsymbol{\gamma}}}\,r_N(\cdot)$ for any $\boldsymbol{\gamma}=(\gamma_n)_{n=0}^{N-1}\in[0,1]^N$, where $\mathfrak{v}_n^{\P;\pi_{\boldsymbol{\gamma}}}:=\prod_{k=n}^{N-1}v_k^{\P;\boldsymbol{\gamma}}$. Then in view of (\mbox{\ref{proof lem ex solution red opt prob - eq10}})--(\mbox{\ref{proof lem ex solution red opt prob - eq30}})
\begin{eqnarray*}
	\lefteqn{\|V_n^{\P;\pi_{\boldsymbol{\gamma}}}\|_{\rm{\scriptsize{H\ddot{o}l}},\alpha}
	\, = \, \|\mathfrak{v}_n^{\P;\pi_{\boldsymbol{\gamma}}}\,r_N\|_{\rm{\scriptsize{H\ddot{o}l}},\alpha}
	\,=\, |\mathfrak{v}_n^{\P;\pi_{\boldsymbol{\gamma}}}|\,\|r_N\|_{\rm{\scriptsize{H\ddot{o}l}},\alpha}
	\,=\, \prod_{k=n}^{N-1}|v_k^{\P;\boldsymbol{\gamma}}|\,\|r_N\|_{\rm{\scriptsize{H\ddot{o}l}},\alpha}}	\\
	& \quad \le & ~ \prod_{k=n}^{N-1}\int_{\R_{\ge 0}}u_\alpha\Big(1+\gamma_k\Big(\frac{y}{\mathfrak{r}_{k+1}^{\P}} - 1\Big)\Big)\,\mathfrak{m}_{k+1}^{\P}(dy)
	\,\le\, (1 + \overline{\mathfrak{m}}_{\P})^{N-n} \hspace{1.8cm}
\end{eqnarray*}
for any $\boldsymbol{\gamma}=(\gamma_n)_{n=0}^{N-1}\in[0,1]^N$, where $\overline{\mathfrak{m}}_{\P}$ is defined as in Subsection \mbox{\ref{Subsec - Proof of Lem ex solution red opt prob}} and we used in the second ``$=$'' the absolute homogeneity of the semi-norm $\|\cdot\|_{\rm{\scriptsize{H\ddot{o}l}},\alpha}$ (as defined in Example \mbox{\ref{examples prob metric - hoelder}}).
Hence, we arrive at (\mbox{\ref{proof lemma Ex-fin verification assumptions Thm hd - eq10}}).

(iii): It can be shown easily that the gauge function $\psi$ belongs to $\M_{\rm{\scriptsize{H\ddot{o}l}},\alpha}$. Thus, in view of Example \mbox{\ref{examples prob metric - hoelder}}, we have $\rho_{\M_{\rm{\scriptsize{H\ddot{o}l}},\alpha}}(\psi) = \|\psi\|_{\rm{\scriptsize{H\ddot{o}l}},\alpha} \le 1 < \infty$.
\end{proof}

\begin{lemma}\label{lemma Ex-fin recursion hd exp reward}
Let $\P=(P_n)_{n=0}^{N-1}\in{\cal P}_\psi$ and $\boldsymbol{\gamma}=(\gamma_n)_{n=0}^{N-1}\in[0,1]^N$ be fixed. Then the solution $(\dot V_k^{\P,\Q;\pi_{\boldsymbol{\gamma}}})_{k=0}^N$ of the backward iteration scheme (\mbox{\ref{Hadamard pi backward iteration scheme}})  admits the representation
\begin{equation}\label{lemma Ex-fin recursion hd exp reward - eq10}
	\dot V_n^{\P,\Q;\pi_{\boldsymbol{\gamma}}}(x_n) = \dot{\mathfrak{v}}_n^{\P,\Q;\pi_{\boldsymbol{\gamma}}}\, u_\alpha(x_n/B_n)
\end{equation}
for any $x_n\in\R_{\ge 0}$, $\Q=(Q_n)_{n=0}^{N-1}\in{\cal P}_\psi$, and $n=0,\ldots,N$, where
$$
\dot{\mathfrak{v}}_n^{\P,\Q;\pi_{\boldsymbol{\gamma}}} := \sum_{k=n}^{N-1}v_{N-1}^{\P;\boldsymbol{\gamma}}\cdots(v_k^{\Q;\boldsymbol{\gamma}} - v_k^{\P;\boldsymbol{\gamma}})\cdots v_n^{\P;\boldsymbol{\gamma}}.
$$
\end{lemma}

\begin{proof}
Fix $\Q=(Q_n)_{n=0}^{N-1}\in{\cal P}_\psi$. We prove the assertion in (\mbox{\ref{lemma Ex-fin recursion hd exp reward - eq10}}) by (backward) induction on $n$. Note that in view of Lemmas \mbox{\ref{lemma Ex-fin verification assumptions Thm hd}}(i) and \mbox{\ref{lemma suff cond for stand cond - pre}}  (with ${\cal P}':=\{\Q\}$) all occurring integrals in the following (exist and) are finite; see the discussion in Remark \mbox{\ref{remark hd - iteration scheme}}. For $n=N$, the assertion in (\mbox{\ref{lemma Ex-fin recursion hd exp reward - eq10}}) is valid because of (\mbox{\ref{Hadamard pi backward iteration scheme}})  and by the choice $\dot{\mathfrak{v}}_N^{\P,\Q;\pi_{\boldsymbol{\gamma}}} := 0$. Now, assume that the assertion in (\mbox{\ref{lemma Ex-fin recursion hd exp reward - eq10}}) holds for $k\in\{n+1,\ldots,N\}$. Then, analogously to (\mbox{\ref{proof lemma Ex-fin repr exp reward - eq10}}), we obtain by means of  (\mbox{\ref{Hadamard pi backward iteration scheme}})  and Lemma \mbox{\ref{lemma Ex-fin repr exp reward}}
\begin{eqnarray*}
 	\dot V_n^{\P,\Q;\pi_{\boldsymbol{\gamma}}}(x_n)
 	& = & \int_{\R_{\ge 0}}\dot V_{n+1}^{\P,\Q;\pi_{\boldsymbol{\gamma}}}(y)\,P_n\big((x_n,f_n^{\boldsymbol{\gamma}}(x_n)),dy\big) \\
    & & +~ \int_{\R_{\ge 0}}V_{n+1}^{\P;\pi_{\boldsymbol{\gamma}}}(y)\,(Q_n - P_n)\big((x_n,f_n^{\boldsymbol{\gamma}}(x_n)),dy\big) \\
    & = & \int_{\R_{\ge 0}}\dot{\mathfrak{v}}_{n+1}^{\P,\Q;\pi_{\boldsymbol{\gamma}}}\,u_\alpha(y/B_{n+1})\,P_n\big((x_n,f_n^{\boldsymbol{\gamma}}(x_n)),dy\big) \\
    & & +~ \int_{\R_{\ge 0}}\mathfrak{v}_{n+1}^{\P;\pi_{\boldsymbol{\gamma}}}\,u_\alpha(y/B_{n+1})\,(Q_n - P_n)\big((x_n,f_n^{\boldsymbol{\gamma}}(x_n)),dy\big) \\
	& = & \dot{\mathfrak{v}}_{n+1}^{\P,\Q;\pi_{\boldsymbol{\gamma}}}\,\int_{\R_{\ge 0}}u_\alpha\Big(\frac{\mathfrak{r}_{n+1}x_n + f_n^{\boldsymbol{\gamma}}(x_n)(y - \mathfrak{r}_{n+1})}{\mathfrak{r}_{n+1}B_n}\Big)\,\mathfrak{m}_{n+1}^{\P}(dy) \\
	& & +~ \mathfrak{v}_{n+1}^{\P;\pi_{\boldsymbol{\gamma}}}\cdot\Big(\int_{\R_{\ge 0}}u_\alpha\Big(\frac{\mathfrak{r}_{n+1}x_n + f_n^{\boldsymbol{\gamma}}(x_n)(y - \mathfrak{r}_{n+1})}{\mathfrak{r}_{n+1}B_n}\Big)\,\mathfrak{m}_{n+1}^{\Q}(dy) \\
	& & \quad -~ \int_{\R_{\ge 0}}u_\alpha\Big(\frac{\mathfrak{r}_{n+1}x_n + f_n^{\boldsymbol{\gamma}}(x_n)(y - \mathfrak{r}_{n+1})}{\mathfrak{r}_{n+1}B_n}\Big)\,\mathfrak{m}_{n+1}^{\P}(dy) \Big) \\
	& = & \dot{\mathfrak{v}}_{n+1}^{\P,\Q;\pi_{\boldsymbol{\gamma}}}\,u_\alpha(x_n/B_n)\,\int_{\R_{\ge 0}}u_\alpha\Big(1 + \gamma_n\Big(\frac{y}{\mathfrak{r}_{n+1}} - 1\Big)\Big)\,\mathfrak{m}_{n+1}^{\P}(dy) \\
	& & +~ \mathfrak{v}_{n+1}^{\P;\pi_{\boldsymbol{\gamma}}}\,u_\alpha(x_n/B_n)\cdot \Big(\int_{\R_{\ge 0}}u_\alpha\Big(1 + \gamma_n\Big(\frac{y}{\mathfrak{r}_{n+1}} - 1\Big)\Big)\,\mathfrak{m}_{n+1}^{\Q}(dy) \\
	& & \quad -~ \int_{\R_{\ge 0}}u_\alpha\Big(1 + \gamma_n\Big(\frac{y}{\mathfrak{r}_{n+1}} - 1\Big)\Big)\,\mathfrak{m}_{n+1}^{\P}(dy)\Big) \\
	& = & \dot{\mathfrak{v}}_{n+1}^{\P,\Q;\pi_{\boldsymbol{\gamma}}}\,u_\alpha(x_n/B_n)\,v_n^{\P;\boldsymbol{\gamma}} + \mathfrak{v}_{n+1}^{\P;\pi_{\boldsymbol{\gamma}}}\,u_\alpha(x_n/B_n)\,(v_n^{\Q;\boldsymbol{\gamma}} - v_n^{\P;\boldsymbol{\gamma}}) \\
	& = & \dot{\mathfrak{v}}_n^{\P,\Q;\pi_{\boldsymbol{\gamma}}}\, u_\alpha(x_n/B_n)
\end{eqnarray*}
for every $x_n\in\R_{\ge 0}$, where
\begin{eqnarray*}
	\dot{\mathfrak{v}}_n^{\P,\Q;\pi_{\boldsymbol{\gamma}}}
	& := & \dot{\mathfrak{v}}_{n+1}^{\P,\Q;\pi_{\boldsymbol{\gamma}}}\,v_n^{\P;\boldsymbol{\gamma}} + \mathfrak{v}_{n+1}^{\P;\pi_{\boldsymbol{\gamma}}}\,(v_n^{\Q;\boldsymbol{\gamma}} - v_n^{\P;\boldsymbol{\gamma}}) \\
	& = & \sum_{k=n+1}^{N-1}v_{N-1}^{\P;\boldsymbol{\gamma}}\cdots(v_k^{\Q;\boldsymbol{\gamma}} - v_k^{\P;\boldsymbol{\gamma}})\cdots v_{n+1}^{\P;\boldsymbol{\gamma}}\,v_n^{\P;\boldsymbol{\gamma}} \\
	 & & \quad +\ \prod_{k=n+1}^{N-1}v_k^{\P;\boldsymbol{\gamma}}\,(v_n^{\Q;\boldsymbol{\gamma}} - v_n^{\P;\boldsymbol{\gamma}}) \\
	& = & \sum_{k=n}^{N-1}v_{N-1}^{\P;\boldsymbol{\gamma}}\cdots(v_k^{\Q;\boldsymbol{\gamma}} - v_k^{\P;\boldsymbol{\gamma}})\cdots v_n^{\P;\boldsymbol{\gamma}}.
\end{eqnarray*}
\end{proof}


\subsubsection{Main part of the proof}


Let $\Q\in{\cal P}_\psi$ be arbitrary but fixed. First of all, note that Lemma \mbox{\ref{lemma Ex-fin verification assumptions Thm hd}} ensures that assumptions (a)--(c) of Theorem \mbox{\ref{Thm - hadamard cal v}}  are satisfied for $\M:=\M':=\M_{\rm{\scriptsize{H\ddot{o}l}},\alpha}$, $\psi$ given by (\mbox{\ref{Ex fin - gauge function}}), and $\Pi_{\rm lin}$ instead of $\Pi$. Take into account that a bounding function (see Definition \mbox{\ref{def bounding function}}) is independent of the set of all (admissible) strategies.

(i): It is an immediate consequence of part (i) of Theorem \mbox{\ref{Thm - hadamard cal v}}  that the functional ${\cal V}_0^{x_0;\pi_{\boldsymbol{\gamma}}}$ defined by (\mbox{\ref{Ex fin - value functional Pi lin}})  is `Fr\'echet differentiable' at $\P$ w.r.t.\ $\M_{\rm{\scriptsize{H\ddot{o}l}},\alpha}$. The corresponding `Fr\'echet derivative' $\dot{\cal V}_{0;\P}^{x_0;\pi_{\boldsymbol{\gamma}}}$ of ${\cal V}_0^{x_0;\pi_{\boldsymbol{\gamma}}}$ at $\P$ admits in view of Remark \mbox{\ref{remark hd - iteration scheme}}  and Lemma \mbox{\ref{lemma Ex-fin recursion hd exp reward}} (recall that $B_0= 1$) the representation
$$
	\dot{\cal V}_{0;\P}^{x_0;\pi_{\boldsymbol{\gamma}}}(\Q-\P) \,=\, \dot V_0^{\P,\Q;\pi_{\boldsymbol{\gamma}}}(x_0) \,=\, \dot{\mathfrak{v}}_0^{\P,\Q;\pi_{\boldsymbol{\gamma}}}\,u_\alpha(x_0),
$$
where $\dot{\mathfrak{v}}_0^{\P,\Q;\pi_{\boldsymbol{\gamma}}}:=\sum_{k=0}^{N-1}v_{N-1}^{\P;\boldsymbol{\gamma}}\cdots(v_k^{\Q;\boldsymbol{\gamma}} - v_k^{\P;\boldsymbol{\gamma}})\cdots v_0^{\P;\boldsymbol{\gamma}}$.

(ii): For any $n=0,\ldots,N-1$ let $\gamma_n^{\P}\in[0,1]$ be the unique solution to the optimization problem (\mbox{\ref{Ex fin - one-stage max problem}}), and set $\boldsymbol{\gamma}^{\P}:=(\gamma_n^{\P})_{n=0}^{N-1}\in[0,1]^N$. Then it follows from the first assertion in part (ii) of Theorem \mbox{\ref{Ex fin - Thm opt trading strat}}  that the linear trading strategy $\pi^{\P} = \pi_{\boldsymbol{\gamma}^{\P}}:=(f_n^{\boldsymbol{\gamma}^{\P}})_{n=0}^{N-1}\in\Pi_{\rm lin}$ defined by (\mbox{\ref{Ex fin - linear decision rules}})  is optimal w.r.t.\ $\P$. Therefore,
the value functional ${\cal V}_0^{x_0}$ defined by (\mbox{\ref{Ex fin - value functional Pi lin}})  admits in view of Remark \mbox{\ref{remark hd value function subset}}  the representation
\begin{equation}\label{proof Ex-Fin Thm Hadamard - eq10}
	{\cal V}_0^{x_0}(\P) \,=\, \sup_{\pi\in\Pi_{\rm lin}}{\cal V}_0^{x_0;\pi}(\P).
\end{equation}
As a consequence, part (ii) of Theorem \mbox{\ref{Thm - hadamard cal v}}  implies that the value functional ${\cal V}_0^{x_0}$ is `Hadamard differentiable' at $\P$ w.r.t.\ $\M_{\rm{\scriptsize{H\ddot{o}l}},\alpha}$ with `Hadamard derivative' $\dot{\cal V}_{0;\P}^{x_0}$ given by
\begin{equation}\label{proof Ex-Fin Thm Hadamard - eq20}
	\dot{\cal V}_{0;\P}^{x_0}(\Q-\P) \,=\, \sup_{\pi\in\Pi_{\rm lin}(\P)}\dot{\cal V}_{0;\P}^{x_0;\pi}(\Q-\P).
\end{equation}
By the second assertion in part (ii) of Theorem \mbox{\ref{Ex fin - Thm opt trading strat}}  we have $\Pi_{\rm lin}(\P)=\{\pi_{\boldsymbol{\gamma}^{\P}}\}$ and therefore the representation of the `Hadamard derivative' $\dot{\cal V}_{0;\P}^{x_0}$ in (\mbox{\ref{proof Ex-Fin Thm Hadamard - eq20}}) simplifies to
$$
	\dot{\cal V}_{0;\P}^{x_0}(\Q-\P) \,=\, \sup_{\pi\in\Pi_{\rm lin}(\P)}\dot{\cal V}_{0;\P}^{x_0;\pi}(\Q-\P) 
\,=\, \dot{\cal V}_{0;\P}^{x_0;\pi_{\boldsymbol{\gamma}^{\P}}}(\Q-\P).
$$
This completes the proof of Theorem \mbox{\ref{Ex fin - Thm Hadamard}}.
\hfill\proofendsign


\subsection{Proof of Remark \mbox{\ref{Ex fin - remark relatively compactness}} }\label{Subsec - Finance Example - proof remark relatively compactness}


Fix $\tau\in\{0,\ldots,N-1\}$, $\delta\in\R_{>0}$, and $\alpha\in(0,1)$. We will here show that the set ${\cal K}_{\tau,\delta}:=\{\Q_{\Delta,\tau}=(Q_{\Delta,\tau;n})_{n=0}^{N-1}: \Delta\in[0,\delta]\}$ $(\subseteq{\cal P}_\psi)$ is compact w.r.t.\ $d_{\infty,\M_{\rm{\scriptsize{H\ddot{o}l}},\alpha}}^\psi$, which implies that ${\cal K}_{\tau,\delta}$ is relatively compact w.r.t.\ $d_{\infty,\M_{\rm{\scriptsize{H\ddot{o}l}},\alpha}}^\psi$.

Consider any sequence in ${\cal K}_{\tau,\delta}$. That is, in other words, pick any sequence $(\Delta_m)\in[0,\delta]^\N$ and consider the sequence $(\Q_{\Delta_m,\tau})_{m\in\N}\in{\cal K}_{\tau,\delta}^\N$. Since $[0,\delta]$ is compact and thus sequentially compact w.r.t.\ the Euclidean distance, we can find a subsequence $(\Delta_m')$ of $(\Delta_m)$ and some $\Delta_0\in[0,\delta]$ such that $\Delta_m'\to \Delta_0$. Then $(\Q_{\Delta'_m,\tau})$ is a subsequence of $(\Q_{\Delta_m,\tau})$, and $\Q_{\Delta_0,\tau}\in{\cal K}_{\tau,\delta}$.
Thus in view of displays (\mbox{\ref{Ex fin - set trans func psi}}), (\mbox{\ref{Ex fin - Hadamard BSM - eq10}}), (\mbox{\ref{Ex fin - def mapping eta}}), and (\mbox{\ref{Ex fin - gauge function}})
\begin{eqnarray*}
	\lefteqn{\Big|\int_{\R_{\ge 0}}h(y)\,Q_{\Delta_m',\tau;n}\big((x,a),dy\big) - \int_{\R_{\ge 0}}h(y)\,Q_{\Delta_0,\tau;n}\big((x,a),dy\big)\Big|} \\
	& = & \Big|\int_{\R_{\ge 0}}h\big(\eta_{\tau,(x,a)}(y)\big)\,\delta_{\Delta_m'}(dy) - \int_{\R_{\ge 0}}h\big(\eta_{\tau,(x,a)}(y)\big)\,\delta_{\Delta_0}(dy)\Big| \\
	& = & \big|h\big(\eta_{\tau,(x,a)}(\Delta_m')\big) - h\big(\eta_{\tau,(x,a)}(\Delta_0)\big)\big| ~ \le ~ \big|\eta_{\tau,(x,a)}(\Delta_m') - \eta_{\tau,(x,a)}(\Delta_0)\big|^\alpha \\
	& = & a^\alpha\,|\Delta_m' - \Delta_0|^\alpha ~ \le ~ x^\alpha\,|\Delta_m' - \Delta_0|^\alpha  ~ \le ~ \psi(x)\,|\Delta_m' - \Delta_0|^\alpha
\end{eqnarray*}
for any $h\in\M_{\rm{\scriptsize{H\ddot{o}l}},\alpha}$, $(x,a)\in D_n$, $n=0,\ldots,N-1$, and $m\in\N$. This implies $d_{\infty,\M_{\rm{\scriptsize{H\ddot{o}l}},\alpha}}^\psi(\Q_{\Delta_m',\tau},\Q_{\Delta_0,\tau})\to 0$. Hence, the assertion follows.
\hfill\proofendsign


\section{Supplement: Existence of optimal strategies} \label{Sec - Existence of optimal strategies}


Consider the setting of Subsection \mbox{\ref{Subsec - MDM}}, that is, let $(\boldsymbol{X},\boldsymbol{A},\P,\Pi,\boldsymbol{r})$ be a MDM in the sense of Definition \mbox{\ref{def MDM}}  with fixed transition function $\P=(P_n)_{n=0}^{N-1}\in{\cal P}$. In this section we will recall from \cite{BaeuerleRieder2011} a statement on the existence of optimal strategies in the sense of Definition \mbox{\ref{def opt strategy}} ; see Theorem \mbox{\ref{Thm - existence opt strategy}} below. Moreover Proposition \mbox{\ref{proposition reward iteration}} below recalls the so-called {\em reward iteration} from \cite{BaeuerleRieder2011} which is used for the proof of Theorem \mbox{\ref{Thm - existence opt strategy}} (see \cite[p.\,23]{BaeuerleRieder2011}) and in our elaborations in Sections \mbox{\ref{Sec - Hadamard differentiability chapter}}--\mbox{\ref{Sec - Finance Example}}.

Recall that we used $E$ to denote the state space of the MDP $\boldsymbol{X}$ and that $E$ was equipped with a $\sigma$-algebra ${\cal E}$. For any $n=0,\ldots,N-1$ we used $\mathbb{F}_n$ to denote the set of {\em all} decision rules at time $n$ and we fixed some $F_n\subseteq\mathbb{F}_n$ which was regarded as the set of all {\em admissible} decision rules at time $n$. We referred to $\Pi:=F_0\times\cdots\times F_{N-1}$ as the set of all {\em admissible} strategies, and we defined $\M(E)$ to be the set of all $({\cal E},{\cal B}(\R))$-measurable functions in $\R^E$.

For any $n=0,\ldots,N-1$, let $\M_n^{\P}(E)$ be the set of all $h\in\M(E)$ satisfying
\begin{equation}\label{assumption integrability appendix}
	\int_E |h(y)|\,P_n\big((x,f_n(x)),dy\big)<\infty\quad \mbox{ for all }x\in E \mbox{ and }f_n\in F_n.
\end{equation}
For any $h\in\M_n^{\P}(E)$, $n=0,\ldots,N-1$, and $f_n\in F_n$ we may define maps ${\cal T}^{\P}_{n,f_n}h:E\rightarrow\R$ and ${\cal T}_n^{\P}h:E\rightarrow(-\infty,\infty]$ by
\begin{equation}\label{def integral operators}
     {\cal T}^{\P}_{n,f_n} h(x) := r_n(x,f_n(x)) + \int_E h(y)\,P_n\big((x,f_n(x)),dy\big) \quad\mbox{and}\quad {\cal T}^{\P}_n h(x) := \sup_{f_n\in F_n}{\cal T}_{n,f_n}^{\P} h(x).
\end{equation}
Note that ${\cal T}^{\P}_{n,f_n}$ and ${\cal T}^{\P}_n$ can be seen as maps from $\M_n^{\P}(E)$ to $\M(E)$ and from $\M_n^{\P}(E)$ to $(-\infty,\infty]^E$ respectively, and that ${\cal T}^{\P}_n$ is also called {\em maximal reward operator at time $n$}.

Finally, recall from (\mbox{\ref{exp tot reward}})  the definition of the map $V_n^{\P;\pi}$. This map can be computed via the so-called {\em reward iteration}:

\begin{proposition}\label{proposition reward iteration}
Let $\pi=(f_n)_{n=0}^{N-1}\in\Pi$ be fixed. If $V_{n+1}^{\P;\pi}(\cdot)\in\M_n^{\P}(E)$ for any $n=0,\ldots,N-1$, then the following two assertions hold.
\begin{enumerate}
    \item[{\rm (i)}] $V_N^{\P;\pi} = r_N$, and $V_n^{\P;\pi} = {\cal T}_{n,f_n}^{\P}V_{n+1}^{\P;\pi}$ for $n=0,\ldots,N-1$.
    \item[{\rm (ii)}] $V_n^{\P;\pi} = {\cal T}_{n,f_n}^{\P}{\cal T}_{n+1,f_{n+1}}^{\P}\cdots{\cal T}_{N-1,f_{N-1}}^{\P}r_N$ for $n=0,\ldots,N-1$.
\end{enumerate}
\end{proposition}

\begin{proof}
The proof of Theorem 2.3.4 in \cite{BaeuerleRieder2011} can be transferred verbatim.
\end{proof}

Note that the assumption $V_{n+1}^{\P;\pi}(\cdot)\in\M_n^{\P}(E)$ (for any $n=0,\ldots,N-1$) is not trivially satisfied. It holds, for example, if the MDM $(\boldsymbol{X},\boldsymbol{A},\P,\Pi,\boldsymbol{r})$ possesses a bounding function $\psi$ (in the sense of Definition \mbox{\ref{def bounding function}}  with ${\cal P}':=\{\P\}$). This is ensured by Lemma \mbox{\ref{lemma suff cond for stand cond - pre}}  with ${\cal P}':=\{\P\}$, taking into account that by (c) of Definition \mbox{\ref{def bounding function}}  we have $\M_\psi(E)\subseteq\M_n^{\P}(E)$ (with $\M_\psi(E)$ as in Subsection \mbox{\ref{Subsec - Bounding functions}}) for any $n=0,\ldots,N-1$. In some cases the assumption in Proposition \mbox{\ref{proposition reward iteration}} can also be shown directly; see e.g.\ the proof of Lemma \mbox{\ref{lemma Ex-fin repr exp reward}} in Subsection \mbox{\ref{Subsec - Finance Example - auxiliary lemmas proof Thm hd}}.

Theorem \mbox{\ref{Thm - existence opt strategy}} below is concerned with the existence of optimal strategies. It invokes the following definition.

\begin{definition} \label{def opt decision rule}
For any $n=0,\ldots,N-1$, a decision rule $f_n^{\P}\in F_n$ is called a {\em maximizer of $h\in\M_n^{\P}(E)$} if ${\cal T}_{n,f_n^{\P}}^{\P}h(x) = {\cal T}_n^{\P}h(x)$ for all $x\in E$.
\end{definition}

The following result which is also known as {\em structure theorem} provides sufficient conditions for the existence of optimal strategies. Recall from (\mbox{\ref{value function}})  the definition of the value function $V_n^{\P}$.

\begin{theorem}\label{Thm - existence opt strategy}
Suppose that there exist for any $n=0,\ldots,N-1$ sets $\M_n^{\P}\subseteq\M_n^{\P}(E)$ and $F'_n\subseteq F_n$ such that the following conditions hold.
\begin{enumerate}
\item[{\rm (a)}] $r_N\in\M_{N-1}^{\P}$.

\item[{\rm (b)}] For any $n=1,\ldots,N-1$ and $h\in\M_{n}^{\P}$ we have ${\cal T}_{n}^{\P} h\in\M_{n-1}^{\P}$.

\item[{\rm (c)}] For any $n=0,\ldots,N-1$ and $h\in\M_n^{\P}$, there exists a maximizer $f_n^{\P}\in F_n$ of $h$ with $f_n^{\P}\in F'_n$.
\end{enumerate}
Then the following three assertions are valid:
\begin{enumerate}
    \item[{\rm (i)}] $V_{0}^{\P}\in\M(E)$, and $V_{n+1}^{\P}\in\M_n^{\P}$ for any $n=0,\ldots,N-1$. Moreover $V_N^{\P} = r_N$, and $V_n^{\P} = {\cal T}_n^{\P} V_{n+1}^{\P}$ for any $n=0,\ldots,N-1$.

    \item[{\rm (ii)}] $V_n^{\P}= {\cal T}_n^{\P}{\cal T}_{n+1}^{\P}\cdots{\cal T}_{N-1}^{\P}r_N$ for any $n=0,\ldots,N-1$.

    \item[{\rm (iii)}] For any $n=0,\ldots,N-1$ there exists a maximizer $f_n^{\P}\in F_n$ of $V_{n+1}^{\P}$ with $f_n^{\P}\in F'_n$. Any such maximizers $f_0^{\P},\ldots,f_{N-1}^{\P}$ form an optimal strategy $\pi^{\P}:=(f_n^{\P})_{n=0}^{N-1}\in\Pi$ w.r.t.\ $\P$ in the MDM $(\boldsymbol{X},\boldsymbol{A},\P,\Pi,\boldsymbol{r})$.
\end{enumerate}
\end{theorem}

\begin{proof}
The proof of Theorem 2.3.8 in \cite{BaeuerleRieder2011} can be transferred verbatim.
\end{proof}

The iteration scheme in part (i) of Theorem \mbox{\ref{Thm - existence opt strategy}} is known as {\em Bellman equation}. Note that conditions  (a)--(c) of Theorem \mbox{\ref{Thm - existence opt strategy}} are not trivially satisfied. It is discussed in Subsection 2.4 of the monograph  \cite{BaeuerleRieder2011} that these conditions hold in so-called structured MDMs. In some situations, however, these conditions can be verified directly; see Subsection \mbox{\ref{Subsec - Finance Example - proof Thm opt trad strat}} (proof of Theorem \mbox{\ref{Ex fin - Thm opt trading strat}}) for an example. For original work on the existence of optimal strategies in MDM see, for instance, \cite{Hinderer1970,Schael1975}. Also note that Theorem \mbox{\ref{Thm - existence opt strategy}} shows that a solution to the (Markov decision) optimization problem (\mbox{\ref{maximization problem}})  can be obtained by solving iteratively $N$ (one-stage) optimization problems.

\begin{remarknorm}\label{remark existence optimal strat}
(i) Under conditions (a)--(c) of Theorem \mbox{\ref{Thm - existence opt strategy}}, part (i) of Theorem \mbox{\ref{Thm - existence opt strategy}} implies that the value function $V_n^{\P}(\cdot)$ is $({\cal E},{\cal B}(\R))$-measurable for any $n=0,\ldots,N$. The measurability of the value function is also ensured if the sets $F_n,\ldots,F_{N-1}$ are at most countable; take into account that the right-hand side of (\mbox{\ref{value function}})  includes the map $V_n^{\P;\pi}$ (as defined in (\mbox{\ref{exp tot reward}})) which depends only on the last $N-n$ components $(f_n,\ldots,f_{N-1})$ of the strategy $\pi=(f_n)_{n=0}^{N-1}\in\Pi$.
The measurability of the value function has been discussed in the literature several times; see, for instance, \cite{Hinderer1970,Schael1975}.

(ii) It follows from Theorem \mbox{\ref{Thm - existence opt strategy}} that any $N$-tuple $(f_n^{\P})_{n=0}^{N-1}$ of maximizers provides an optimal strategy $\pi^{\P}$ w.r.t.\ $\P$ in the MDM $(\boldsymbol{X},\boldsymbol{A},\P,\Pi,\boldsymbol{r})$ via $\pi^{\P}:=(f_n^{\P})_{n=0}^{N-1}$. The reverse statement, however, is not true since even under the assumptions of Theorem \mbox{\ref{Thm - existence opt strategy}} optimal strategies are {\em not} necessarily composed of maximizers; see, e.g., \cite[Example 2.3.10]{BaeuerleRieder2011}. Hence, Theorem \mbox{\ref{Thm - existence opt strategy}} provides only a sufficient criterion for the existence of optimal strategies.

(iii) In view of the second part of (ii), an optimal strategy in a MDM can in general be non-unique. However, this does not exclude that in specific situations there is {\em exactly} one optimal strategy. For an example see Subsection \mbox{\ref{Subsec - Finance Example - Comp opt trading strat}}.

(iv) In the case where we are interested in minimizing expected total costs in the MDM $(\boldsymbol{X},\boldsymbol{A},\P,\Pi,\boldsymbol{r})$ (see Remark \mbox{\ref{remark non Markovian strat}}(ii)),
the integral operator ${\cal T}^{\P}_n$ is given by (\mbox{\ref{def integral operators}}) with ``$\sup$'' replaced by ``$\inf$'' and in Definition \mbox{\ref{def opt decision rule}} we have to replace ``maximizer'' by ``minimizer''.
{\hspace*{\fill}$\Diamond$\par\bigskip}
\end{remarknorm}


\section{Supplement: Topology generated by the H{\"o}lder-$\alpha$ metric}\label{Sec - hoelder metric psi weak topology}


We use the notation and terminology introduced in Subsection \mbox{\ref{Subsec - Metric on set of probability measures}}. In particular, the H{\"o}lder-$\alpha$ metric $d_{\rm{\scriptsize{H\ddot{o}l}},\alpha}$ was introduced in Example \mbox{\ref{examples prob metric - hoelder}} of Subsection \mbox{\ref{Subsec - Metric on set of probability measures}}.

\begin{lemma}\label{lemma - d beta hoelder - psi weak topology}
Assume that $(E,d_E)$ is a complete and separable metric space, and let $\alpha\in(0,1]$ and $x'\in E$ be arbitrary but fixed. Then the H{\"o}lder-$\alpha$ metric $d_{\rm{\scriptsize{H\ddot{o}l}},\alpha}$ introduced in Example \mbox{\ref{examples prob metric - hoelder}}  metricizes the $\psi$-weak topology on ${\cal M}_1^\psi(E)$ for $\psi(x):= 1 + d_E(x,x')^\alpha$.
\end{lemma}

\begin{proof}
As the $\psi$-weak topology is metrizable (see, e.g., Corollary A.45 in \cite{FoellmerSchied2011}), it suffices to show that for any choice of $\mu,\mu_1,\mu_2\ldots\in{\cal M}_1^\psi(E)$ we have $\mu_n\to\mu$ $\psi$-weakly if and only if $d_{\M_{\rm{\scriptsize{H\ddot{o}l}},\alpha}}(\mu_n,\mu)\to 0$.

First assume that $d_{\M_{\rm{\scriptsize{H\ddot{o}l}},\alpha}}(\mu_n,\mu)\to 0$. As $\mu_n\to\mu$ $\psi$-weakly if and only if $\mu_n\to\mu$ weakly and $\int_E\psi\,d\mu_n\to\int_E\psi\,d\mu$ (see, e.g., Lemma 2.1 in \cite{Kraetschmeretal2017}), it suffices to show that $\mu_n\to\mu$ weakly and $\int_E\psi\,d\mu_n\to\int_E\psi\,d\mu$. Any bounded $h\in\R^E$ with $\|h\|_{\rm{\scriptsize{Lip}}}<\infty$ satisfies $\|h\|_{\rm{\scriptsize{H\ddot{o}l}},\alpha}\le C_h:=\max\{\|h\|_{\rm{\scriptsize{Lip}}},2\|h\|_\infty\}$. Since $h/C_h$ lies in $\M_{\rm{\scriptsize{H\ddot{o}l}},\alpha}$, our assumption implies $\int_E h\,d\mu_n\to\int_Eh\,d\mu$. That is, $\int_E h\,d\mu_n\to\int_Eh\,d\mu$ for any bounded and Lipschitz continuous $h\in\R^E$. By the Portmanteau theorem we can conclude $\mu_n\to\mu$ weakly. Moreover, as $\psi$ lies in $\M_{\rm{\scriptsize{H\ddot{o}l}},\alpha}$, our assumption also implies $\int_E\psi\,d\mu_n\to\int_E\psi\,d\mu$.

Conversely, assume that $\mu_n\to\mu$ $\psi$-weakly. We have to show that for every $\varepsilon>0$ there exists some $n_0\in\N$ such that
\begin{equation}\label{Charakterisierung psi-schwache Topologie - Proof - 10}
    \sup_{h\in\M_{\rm{\scriptsize{H\ddot{o}l}},\alpha}}\Big|\int_E h\,d\mu_n-\int_E h\,d\mu\Big|\le\varepsilon\qquad\mbox{for all }n\ge n_0.
\end{equation}
For any $K>0$, the left hand side of (\mbox{\ref{Charakterisierung psi-schwache Topologie - Proof - 10}}) is bounded above by
\begin{equation}\label{Charakterisierung psi-schwache Topologie - Proof - 20}
    \sup_{h\in\M_{\rm{\scriptsize{H\ddot{o}l}},\alpha}}\Big|\int_E h_K\,d\mu_n-\int_E h_K\,d\mu\Big|+\sup_{h\in\M_{\rm{\scriptsize{H\ddot{o}l}},\alpha}}\Big|\int_E h^K\,d\mu_n-\int_E h^K\,d\mu\Big|
\end{equation}
with $h_K:=h\eins_{\{|h|\le K\}}+K\eins_{\{h> K\}}-K\eins_{\{h<-K\}}$, and $h^K:=h-h_K$. Without loss of generality we may and do assume that $h(x_0)=0$ for all $h\in\M_{\rm{\scriptsize{H\ddot{o}l}},\alpha}$; take into account that $|\int_E h\,d\mu_n-\int_E h\,d\mu|$ remains unchanged when a constant is added to $h$. Then $|h(x)|=|h(x)-h(x_0)|\le d_E(x,x_0)^\alpha\le\psi(x)$ for all $h\in\M_{\rm{\scriptsize{H\ddot{o}l}},\alpha}$. In particular, $|h^K|\le |h|\eins_{\{|h|>K\}}\le \psi\eins_{\{\psi>K\}}$. Thus the second summand in (\mbox{\ref{Charakterisierung psi-schwache Topologie - Proof - 20}}) is bounded above by
\begin{equation}\label{Charakterisierung psi-schwache Topologie - Proof - 30}
    \int_E\psi\eins_{\{\psi>K\}}\,d\mu_n+\int_E\psi\eins_{\{\psi>K\}}\,d\mu
\end{equation}
Now we can choose $K>0$ so large that the second summand in (\mbox{\ref{Charakterisierung psi-schwache Topologie - Proof - 30}}) is at most $\varepsilon/5$. The first summand in (\mbox{\ref{Charakterisierung psi-schwache Topologie - Proof - 30}}) is bounded above by
\begin{equation}\label{Charakterisierung psi-schwache Topologie - Proof - 40}
    \Big|\int_E\psi\eins_{\{\psi>K\}}\,d\mu_n-\int_E\psi\eins_{\{\psi>K\}}\,d\mu\Big|+\int_E\psi\eins_{\{\psi>K\}}\,d\mu
\end{equation}
The second summand in (\mbox{\ref{Charakterisierung psi-schwache Topologie - Proof - 40}}) is at most $\varepsilon/5$ (see above) and the first summand in (\mbox{\ref{Charakterisierung psi-schwache Topologie - Proof - 40}}) is bounded above by
\begin{equation}\label{Charakterisierung psi-schwache Topologie - Proof - 50}
    \Big|\int_E\psi\,d\mu_n-\int_E\psi\,d\mu\Big|+\Big|\int_E\psi\eins_{\{\psi\le K\}}\,d\mu_n-\int_E\psi\eins_{\{\psi\le K\}}\,d\mu\Big|.
\end{equation}
The first summand in (\mbox{\ref{Charakterisierung psi-schwache Topologie - Proof - 50}}) converges to $0$ as $n\to\infty$, because $\mu_n\to\mu$ $\psi$-weakly. Thus we can find $n_0\in\N$ such that it is bounded above by $\varepsilon/5$ for every $n\ge n_0$. Since $\mu\circ\psi^{-1}$ as a probability measure on the real line has at most countably many atom, we may and do assume that $K>0$ is chosen such that $\mu[\{\psi=K\}]=0$. Since $\mu_n\to\mu_0$ ($\psi$-weakly and thus) weakly, it follows by the portmanteau theorem that the second summand in (\mbox{\ref{Charakterisierung psi-schwache Topologie - Proof - 50}}) converges to $0$ as $n\to\infty$. By possibly increasing $n_0$ we obtain that the second summand in (\mbox{\ref{Charakterisierung psi-schwache Topologie - Proof - 50}}) is at most $\varepsilon/5$ for all $n\ge n_0$. So far we have shown that the second summand in (\mbox{\ref{Charakterisierung psi-schwache Topologie - Proof - 20}}) is bounded above by $4\varepsilon/5$ for all $n\ge n_0$. As the functions of $\M_{{\rm{\scriptsize{H\ddot{o}l}}},\alpha;K}:=\{h_K:h\in\M_{\rm{\scriptsize{H\ddot{o}l}},\alpha}\}$ are uniformly bounded and equicontinuous, Corollary 11.3.4 in \cite{Dudley2002} ensures that one can increase $n_0$ further such that the first summand in (\mbox{\ref{Charakterisierung psi-schwache Topologie - Proof - 20}}) is bounded above by $\varepsilon/5$ for all $n\ge n_0$. That is, we arrive at (\mbox{\ref{Charakterisierung psi-schwache Topologie - Proof - 10}}).
\end{proof}


\end{document}